\documentclass[12pt]{amsart}

\pdfoutput=1

\usepackage[text={420pt,660pt},centering]{geometry}

\usepackage{color}
\usepackage{esint,amssymb}
\usepackage{graphicx}
\usepackage{MnSymbol}
\usepackage{mathtools}
\usepackage[colorlinks=true, pdfstartview=FitV, linkcolor=blue, citecolor=blue, urlcolor=blue,pagebackref=false]{hyperref}
\usepackage{microtype}

\usepackage{bm}
\usepackage{scalerel} 
\usepackage{dsfont}
\usepackage[font={footnotesize}]{caption}

\usepackage{lipsum}
\definecolor{darkgreen}{rgb}{0,0.5,0}
\definecolor{darkblue}{rgb}{0,0,0.7}
\definecolor{darkred}{rgb}{0.9,0.1,0.1}

\makeatletter 
\newcommand{\negphantom}{\v@true\h@true\negph@nt} 
\newcommand{\neghphantom}{\v@false\h@true\negph@nt} 
\newcommand{\negph@nt}{\ifmmode\expandafter\mathpalette 
  \expandafter\mathnegph@nt\else\expandafter\makenegph@nt\fi} 
\newcommand{\makenegph@nt}[1]{%
  \setbox\z@\hbox{\color@begingroup#1\color@endgroup}\finnegph@nt} 
\newcommand{\finnegph@nt}{%
  \setbox\tw@\null 
  \ifv@ \ht\tw@\ht\z@\dp\tw@\dp\z@\fi \ifh@\wd\tw@-\wd\z@\fi\box\tw@} 
\newcommand{\mathnegph@nt}[2]{%
  \setbox\z@\hbox{$\m@th #1{#2}$}\finnegph@nt} 
\makeatother

\newcommand{\Hminusul}{\hat{\phantom{H}}\negphantom{H}\underline{H}^{-1}}

\newtheorem{proposition}{Proposition}
\newtheorem{theorem}[proposition]{Theorem}
\newtheorem{lemma}[proposition]{Lemma}
\newtheorem{corollary}[proposition]{Corollary}

\theoremstyle{definition}
\newtheorem{remark}[proposition]{Remark}

\newcommand{\cref}[1]{Corollary~\ref{c.#1}}

\numberwithin{equation}{section}
\numberwithin{proposition}{section}

\newcommand{\A}{\mathcal{A}}
\newcommand{\Ahom}{\overline{\A}}

\newcommand{\Z}{\mathbb{Z}}
\newcommand{\N}{\mathbb{N}}
\newcommand{\R}{\mathbb{R}}

\newcommand{\E}{\mathbb{E}}
\renewcommand{\P}{\mathbb{P}}
\newcommand{\F}{\mathcal{F}}
\newcommand{\Zd}{\mathbb{Z}^d}
\newcommand{\Rd}{{\mathbb{R}^d}}

\newcommand{\ep}{\varepsilon}

\renewcommand{\a}{\mathbf{a}}

\renewcommand{\a}{\mathbf{a}}
\renewcommand{\b}{\mathbf{b}}

\newcommand{\ahom}{{\overbracket[1pt][-1pt]{\a}}}  

\newcommand{\cu}{{\scaleobj{1.2	}{\square}}}

\renewcommand{\L}{\mathcal{L}}
\renewcommand{\fint}{\strokedint}

\newcommand{\Ll}{\left}
\newcommand{\Rr}{\right}

\DeclareMathOperator{\dist}{dist}

\DeclareMathOperator*{\esssup}{ess\,sup}

\newcommand{\X}{\mathcal{X}}

\renewcommand{\bar}{\overline}

\renewcommand{\tilde}{\widetilde}

\renewcommand{\O}{\mathcal{O}}

\DeclareMathOperator{\data}{data}

\renewcommand{\hat}{\widehat}

\newcommand{\smallpara}{\varrho}

\newcommand{\uhom}{{u_{\mathrm{hom}}}}
\newcommand{\whom}{{w_{\mathrm{hom}}}}
\newcommand{\Lhom}{\overline{L}}

\renewcommand{\data}{\mathrm{data}}

\begin{document}

\title[Homogenization, linearization and regularity]{Homogenization, linearization and large-scale regularity for nonlinear elliptic equations}

\begin{abstract}
We consider nonlinear, uniformly elliptic equations with random, highly oscillating coefficients satisfying a finite range of dependence. We prove that homogenization and linearization commute in the sense that the linearized equation (linearized around an arbitrary solution) homogenizes to the linearization of the homogenized equation (linearized around the corresponding solution of the homogenized equation). We also obtain a quantitative estimate on the rate of this homogenization. These results lead to a better understanding of \emph{differences} of solutions to the nonlinear equation, which is of fundamental importance in quantitative homogenization. In particular, we obtain a large-scale $C^{0,1}$ estimate for differences of solutions---with optimal stochastic integrability. Using this estimate, we prove a large-scale $C^{1,1}$ estimate for solutions, also with optimal stochastic integrability. Each of these regularity estimates are new even in the periodic setting. As a second consequence of the large-scale regularity for differences, we improve the smoothness of the homogenized Lagrangian by showing that it has the same regularity as the heterogeneous Lagrangian, up to $C^{2,1}$.
\end{abstract}

\author[S. Armstrong]{Scott Armstrong}
\address[S. Armstrong]{Courant Institute of Mathematical Sciences, New York University, 251 Mercer St., New York, NY 10012}
\email{scotta@cims.nyu.edu}

\author[S. J. Ferguson]{Samuel J. Ferguson}
\address[S. J. Ferguson]{Courant Institute of Mathematical Sciences, New York University, 251 Mercer St., New York, NY 10012}
\email{sjf370@nyu.edu}

\author[T. Kuusi]{Tuomo Kuusi}
\address[T. Kuusi]{Department of Mathematics and Statistics, University of Helsinki}
 \email{tuomo.kuusi@helsinki.fi}

\date{\today}

\maketitle
\setcounter{tocdepth}{1}

\section{Introduction}

\subsection{Motivation and informal summary of main results}

We are motivated by the goal of developing a quantitative theory of stochastic homogenization for nonlinear elliptic equations in divergence form. 
Such a theory has been developed in recent years for \emph{linear} equations which is by now rather satisfactory: see for instance~\cite{AKMbook,GO6} and the references therein. 
In the linear case, the phenomenon of improved regularity of solutions on large-scales plays an important role, for instance by providing control of ``small errors'' in a sufficiently strong norm. In the nonlinear setting, such small errors are typically not solutions of the equation but rather the \emph{difference} of two solutions which satisfy a linearized equation. In this paper, we obtain quantitative homogenization estimates for such linearized equations and obtain a large-scale $C^{0,1}$--type estimate for differences of solutions.

\smallskip

The equations we analyze take the form
\begin{equation} 
\label{e.pde}
- \nabla \cdot \left( D_p L(\nabla u(x),x) \right) = 0 \quad \mbox{in} \ U \subseteq \Rd, \quad d\geq 2, 
\end{equation}
 where the Lagrangian $L=L(p,x)$ is assumed to be uniformly convex and $C^{2,\gamma}$ in the variable~$p$. Of course, this equation is variational: $u\in H^1(U)$ is a solution of~\eqref{e.pde} if and only if it is a local minimizer of the integral functional 
\begin{equation*} \label{}
v \mapsto \int_U L(\nabla v(x),x) \,dx. 
\end{equation*}
We further assume that~$L$ is a stochastic object and that its law~$\P$ is $\Zd$--stationary and has a unit range of dependence (with respect to the variable~$x$). The goal is to understand the \emph{statistics} of the solutions, under the probability measure~$\P$, and on \emph{large length scales}, that is, the (``macroscopic'') domain $U$ is very large relative to (``microscopic'') unit scale, which is the correlation length scale of the coefficients. 

\smallskip

The principle of homogenization asserts that, in the regime in which the ratio of these two length scales is large, a solution of~\eqref{e.pde} is, with probability approaching one, close in a strong norm (relative to its size in the same norm) to a solution of a deterministic equation of the form
\begin{equation}
\label{e.pde.homog}
-\nabla \cdot \left( D_p\overline{L} \left( \nabla \uhom \right) \right)
= 0 \quad \mbox{in} \ U.
\end{equation}
Dal Maso and Modica~\cite{DM1,DM2} were the first to prove such a result for the equation~\eqref{e.pde} in this (and actually a much more general) setting. They realized that the variational structure of the equation provides a natural subadditive quantity, which has a $\P$--almost sure limit by the subadditive ergodic theorem, and that this limit implies a general homogenization result for the equation. 

\smallskip

We are interested here in quantitative results, in particular the speed of convergence to the homogenization limit. There has been a lot of recent interest in building a quantitative theory of homogenization for linear, uniformly elliptic equations, and there is now an essentially complete and optimal theory available in this simplest of settings: see~\cite{AKMbook,AKM,GNO,GO6}. While there has been some success in extending this theory to degenerate linear equations (see~\cite{AD,D}), quantitative homogenization estimates for nonlinear equations are relatively sparse: the only previous results appeared in~\cite{AM,AS}, which quantified the subadditive argument of~\cite{DM1,DM2} to obtain an estimate for the homogenization error which is at most a power of the ratio of the length scales (see Theorem~\ref{t.AS.homogenization} below for the precise statement). The paper~\cite{AS} also introduced the concept of a \emph{large-scale regularity theory} for random elliptic operators and, in particular, proved a large-scale $C^{0,1}$--type estimate for solutions of~\eqref{e.pde} (see Theorem~\ref{t.AS.regularity} below). Subsequently, some variations and extensions of this result were developed in~\cite{GNO2,AM,FO,AKM,AKMbook} in the linear setting.

\smallskip

To develop a more precise quantitative theory for  nonlinear equations, extending the results known in the linear setting---such as sharp exponents for the scaling of the homogenization error and a characterization of the scaling limit of solutions---what is needed is finer estimates on solutions and, more importantly, on \emph{differences} of two solutions (which are typically very close to each other), on all length scales larger than a multiple of the microscopic scale. For linear equations, since the differences of solutions are also solutions, the large-scale regularity already gives exactly the sort of information which is required. For a nonlinear equation such as~\eqref{e.pde}, the difference of two solutions~$u$ and~$v$ is the solution of the linear equation 
\begin{equation*}
-\nabla \cdot \left( \a(x)\nabla (u-v) \right) = 0,
\end{equation*}
with coefficients $\a(x)$ given by 
\begin{equation*}
\a(x) = \int_0^1 D^2_p L(\nabla u + t \nabla (v-u) ,x)\,dt. 
\end{equation*}
If $v$ is a small perturbation of~$u$, then this equation is very close to the \emph{linearization} of~\eqref{e.pde} around~$u$, namely
\begin{equation}
\label{e.linearization}
-\nabla \cdot \left( D^2_pL(\nabla u(x),x) \nabla w(x) \right) = 0.
\end{equation}
It is therefore very natural to consider the large-scale behavior of solutions of linearized equations of the form~\eqref{e.linearization}, where $u$ is a solution of~\eqref{e.pde}. 

\smallskip

In this paper, we show that the linearized equation~\eqref{e.linearization} around an arbitrary solution~$u$ of~\eqref{e.pde} homogenizes, with an algebraic rate of convergence, and that the homogenized equation for this linearized equation is the linearization of the homogenized equation~\eqref{e.pde.homog}  around the corresponding homogenized solution $\uhom$ (with the same boundary conditions as $u$), namely
\begin{equation}
\label{e.linearization.homog}
-\nabla \cdot \left( D^2\overline{L}(\nabla \uhom) \nabla \whom \right) = 0.
\end{equation}
In other words, \emph{homogenization and linearization commute}. The precise statement can be found in Theorem~\ref{t.linearization} below. This result yields much finer information regarding the differences of solutions of the original nonlinear equation and we expect it to play a crucial role in the future development of a quantitative theory of stochastic homogenization for nonlinear equations.  

\smallskip

As a first consequence, we show that it provides sufficient information about the differences of solutions to improve their regularity. Recall that, since the difference of two solutions solves a linear equation, it satisfies a~$C^{0,\beta}$ estimate for a tiny exponent~$\beta(d,\Lambda)>0$ as a consequence of the De Giorgi-Nash estimate. 
On the other hand, differences of solutions of the homogenized equation~\eqref{e.pde.homog} possess much better regularity: they satisfy at least a $C^{1,\beta}$ estimate in our setting, by the Schauder estimates. However, the quantitative estimate on the homogenization of the linearized equation implies that differences of solutions of~\eqref{e.pde} can be well-approximated, on large scales, by differences of solutions of~\eqref{e.pde.homog}. This allows us to obtain a large scale $C^{0,1}$--type estimate for differences of solutions of~\eqref{e.pde} by ``borrowing'' the better regularity of the homogenized equation. On a technical level, this is achieved via a Campanato-type excess decay argument very similar to the one introduced in~\cite{AS}. See Theorem~\ref{t.regularity.differences} below for the  statement. 

\smallskip

We expect the large-scale~$C^{0,1}$ estimate for differences to play a fundamental role in the development of quantitative homogenization estimates which are optimal in the scaling of the error. This estimate implies, for instance, a bound on the difference of two correctors (with sharp stochastic integrability): see~\eqref{e.yesdiffcorrectors}, below. Essentially, this is the key nonlinear ingredient that makes it possible to develop a complete quantitative theory analogous to the linear setting, as we will show in a forthcoming article. In fact, such results have appeared in a very recent preprint of Fischer and Neukamm~\cite{FN} (which was posted more than a year after the first version of the present paper), who relied crucially on a version of the large-scale~$C^{0,1}$ estimate for differences.

\smallskip

Our third result, stated in Theorem~\ref{t.C11estimate}, is a large-scale $C^{1,1}$ estimate for solutions of the nonlinear heterogeneous equation. It characterizes the set~$\L_1$ of solutions with at-most linear growth (it coincides with the set of first-order correctors) and asserts that an arbitrary solution on a finite domain can be approximated, up to a quadratic error in the ratio of the scales, by elements of~$\L_1$. That is, a solution of the nonlinear solution can be approximated by an element of~$\L_1$ with the same precision that a harmonic function can be approximated by an affine function. This result can be compared to that of Moser and Struwe~\cite{MS} in the case of periodic case, who proved the qualitative characterization of~$\L_1$ but not the quantitative statement of~$C^{1,1}$ regularity. The proof of the latter, even in the periodic case, relies crucially on the analysis of the linearized equation and the approximation of differences developed here and can thus be considered the first ``truly nonlinear'' large-scale regularity result in homogenization. We remark that obtaining a~$C^{1,\delta}$--type statement, where the factor $\left( \frac rR \right)^2$ in~\eqref{e.C11} is replaced, for a small exponent~$\delta>0$, by $\left( \frac rR \right)^{1+\delta}$, is much less difficult to obtain. The optimal quadratic scaling is a much more subtle issue. 

\smallskip

The last main result we state concerns the regularity of~$\overline{L}$ itself. It is easy to see from the definition of~$\overline{L}$ that it satisfies the same upper and lower bounds of uniform convexity that is assumed for~$L$ in the variable~$p$, and therefore~$\overline{L}\in C^{1,1}$. It is natural to expect that~$\overline{L}$ is as smooth as~$L$ is in the variable~$p$. It turns out however that proving more smoothness for~$\overline{L}$ is subtle and intractably tied to the large-scale regularity theory for differences of solutions. As a consequence of the large-scale $C^{0,1}$--type estimate we are able to show that, for each $\gamma\in (0,1]$, and under an additional assumption on the spatial regularity of $L$ (see~\eqref{e.extra.assump}), we have roughly that 
\begin{equation*}
L(\cdot,x)\in C^{2,\gamma} \implies \overline{L}\in C^{2,\gamma}.
\end{equation*}
See Theorem~\ref{t.Lreg} for the precise statement.  

\smallskip

The case of (non-random) Lagrangians $L=L(p,x)$ which are periodic in $x$ is a special case of our assumptions and the large-scale regularity results in this paper, even in this much simpler situation and in their qualitative form, are new. 

\smallskip

While the equation~\eqref{e.pde} considered here is variational (in the sense that the coefficients are the gradient $D_pL$ of a convex Lagrangian~$L$), our arguments 
can be extended, with only minor, mostly notational changes, to the more general equation
\begin{equation*}
-\nabla \cdot\left( \a\left( \nabla u(x),x \right) \right) = 0
\end{equation*}
where $p\mapsto \a(p,x)$ is a uniformly monotone map (but not necessarily the gradient of a uniformly convex function) which belongs to $C^{2,\gamma}$.
This is because, contrary to a widely-held belief, \emph{every} such divergence-form, uniformly monotone operator has a variational structure (see~\cite[Section 2]{AM} and~\cite[Chapter 10]{AKMbook}). Indeed, we note that the quantitative homogenization results of~\cite{AS} were extended in~\cite{AM} to the setting of general uniformly monotone maps \emph{without changing the variational structure of the arguments}. 

\smallskip

The ``commutability of homogenization and linearization'' has been previously considered in the works of M\"uller and Neukamm~\cite{MN} and Gloria and Neukamm~\cite{GN}. The results in these papers are, however, not directly related to ours as the notion of linearization considered there is very different from ours. In particular, they linearize around the zero function in the direction of a fixed function which is smoothly varying on the macroscopic scale (and which is not necessarily a solution), rather than linearize around a solution oscillating on the microscopic scale. In particular, their results do not give information on the differences of solutions. While making a revision of this article, we became aware of the recent work of Neukamm and Sch\"affner~\cite{NS} who study the commutativity of homogenization and linearization in the periodic setting, under a perturbative assumption, for possibly nonconvex energy functionals arising in linear elasticity. Their results are qualitative in nature and do not provide information regarding the differences of solutions or their regularity, but are definitely in the spirit of Theorem~\ref{t.linearization}. 

\subsection{Statement of the main results}
\label{ss.assump}

The parameters~$d\in\N$ with $d\geq 2$,  $\gamma\in (0,1]$ and $\mathsf{M}_0,\Lambda \in [1,\infty)$ are fixed throughout the paper. For short we denote 
\begin{equation*}
\data:= (d,\gamma,\mathsf{M_0},\Lambda).
\end{equation*}
We assume the Lagrangians~$L$ satisfy the following conditions:
\begin{enumerate}
\item[(L1)] $L:\Rd \times \Rd \to \R$ \ is a Carath\'eodory function which is measurable in $x$ and~$C^{2,\gamma}$ in~$p$. It is assumed to satisfy the bound 
\begin{equation*}
\esssup_{x\in\Rd} \left( \left| D_pL(0,x) \right| + \left\| D^2_pL(\cdot,x) \right\|_{C^{0,\gamma}(\Rd)} \right) \leq \mathsf{M}_0. 
\end{equation*}

\smallskip

\item[(L2)] $L$ is uniformly convex in $p$: for every $p\in\Rd$ and Lebesgue-a.e.~$x\in\Rd$, 
\begin{equation*} \label{}
I_d \leq D^2_p L(p,x) \leq \Lambda I_d.
\end{equation*}
\end{enumerate} 
We define $\Omega$ to be the set of all such functions:
\begin{equation*} \label{}
\Omega := \left\{ L \,:\,  \mbox{$L$ satisfies~(L1) and (L2)} \right\}.
\end{equation*}
Note that $\Omega$ depends on the fixed parameters~$(d,\gamma,\Lambda,\mathsf{M}_0)$. We endow $\Omega$ with the following family of~$\sigma$--algebras: for each Borel $U \subseteq \Rd$, define
\begin{multline*} \label{}
\F(U):= \mbox{the $\sigma$--algebra generated by the family of random variables}\\
L \mapsto \int_{U} L(p,x) \phi(x)\, dx, \quad p\in \Rd, \ \phi \in C^\infty_c(\Rd). 
\end{multline*}
The largest of these is denoted by $\F:= \F(\Rd)$. We also denote by $\overline{\Omega}(\gamma,\mathsf{M}_0)$ the set of Lagrangians~$L$ which satisfy (L1) and (L2) and do not depend on the variable~$x$. 

\smallskip

We assume that the law of the ``canonical Lagrangian''~$L$ is a probability measure~$\P$ on $(\Omega,\F)$ satisfying the following two assumptions:

\begin{enumerate}

\item[(P1)] $\P$ has a unit range of dependence: for all Borel subsets $U,V\subseteq \Rd$ such that $\dist(U,V) \geq 1$,
\begin{equation*} \label{}
\mbox{$\F(U)$ and $\F(V)$ are $\P$--independent.}
\end{equation*}

\item[(P2)] $\P$ is stationary with respect to $\Zd$--translations: for every $z\in \Zd$ and $E\in \F$,
\begin{equation*} \label{}
\P \left[ E \right] = \P \left[ T_z E \right],
\end{equation*}
where the translation group $\{T_z\}_{z\in\Zd}$ acts on $\Omega$ by $(T_zL)(p,x) = L(p,x+z)$.
\end{enumerate}
The expectation with respect to~$\P$ is denoted by~$\E$. 

\smallskip

Since we are often concerned with showing that the fluctuations of our random variables are small, the following notation is convenient: for every $\sigma \in(0,\infty)$, $\theta>0$ and random variable~$X$ on $\Omega$, we write 
\begin{equation*}
X \leq \O_\sigma\left( \theta \right) 
\iff 
\E \left[ \exp\left( \left( \frac{X_+}{\theta} \right)^\sigma \right)  \right] \leq 2. 
\end{equation*}

The result of Dal Maso and Modica~\cite{DM1,DM2}, under more general assumptions than the ones here, implies that local minimizers of the energy functional for~$L$ converge, on large scales, $\P$--a.s.,~to those of the energy functional for~$\overline{L}$, for some deterministic and constant~$\overline{L}$. This qualitative homogenization result was quantified in~\cite{AS}, under the finite range of dependence assumption, a version of which we recall below in Theorem~\ref{t.AS.homogenization}. Our assumptions are still stronger than the ones in~\cite{AS} because we require that $L$ be $C^{2,\gamma}$ in the variable~$p$, uniformly in~$x$, which is necessary to study the linearized equations. Even without this assumption, it is fairly easy to show that the homogenized Lagrangian~$\overline{L}$ must inherit the uniform convexity condition~(L2) and is therefore~$C^{1,1}$. It is less obvious that~$\overline{L}$ is necessarily~$C^2$, even under the uniform $C^{2,\gamma}$ assumption. We show in Proposition~\ref{p.nu.C2beta} below that in fact~$\overline{L}\in C^{2,\beta}$ for an exponent~$\beta(\data)>0$ (which \emph{a~priori} may be smaller than~$\gamma$). In fact, $\overline{L} \in \overline{\Omega}(\beta,C)$ for a constant $C(\data)<\infty$ (which may be larger than~$\mathsf{M}_0$).

\smallskip

We next present the first main result of the paper, which is a quantitative statement concerning the commutability of linearization and homogenization. The statement should be compared with that of Theorem~\ref{t.AS.homogenization}, below. We remark that its statement, like those of Theorems~\ref{t.regularity.differences} and~\ref{t.C11estimate} below, is new even in the  case of deterministic equations with periodic coefficients.

\begin{theorem}[Quantitative homogenization of linearized equations]
\label{t.linearization}
\emph{}\\
Let $\sigma \in (0,d)$, $\delta\in \left(0,\tfrac 12\right]$, $\mathsf{M}\in [1,\infty)$ and $U\subseteq B_1$ be a bounded Lipschitz domain. Then there exist an exponent~$\alpha(U,\data)>0$, a constant~$C(\sigma,\delta,\mathsf{M},U,\data)<\infty$ and a random variable~$\X$ satisfying the bound
\begin{equation}
\label{e.size.X}
\X = \O_1\left(C \right)
\end{equation}
such that the following statement holds. 
For each~$\ep\in (0,1]$, pair~$u^\ep,\uhom \in W^{1,2+\delta}(U)$ satisfying
\begin{equation*}
\left\{ 
\begin{aligned}
& -\nabla \cdot \left( D_pL\left( \nabla u^\ep,\tfrac x\ep \right) \right) = 0  & \mbox{in} & \ U, \\
&  -\nabla \cdot \left( D_p\overline{L} \left( \nabla \uhom \right) \right)  = 0 & \mbox{in} & \ U, \\
& u^\ep - \uhom \in H^1_0(U), \\
& \left\| \nabla \uhom \right\|_{L^{2+\delta}(U)} \leq \mathsf{M},
\end{aligned}
\right.
\end{equation*}
function~$f\in W^{1,2+\delta}(U)$ and pair~$w^\ep, \whom \in H^1(U)$ satisfying the corresponding Dirichlet problems for the linearized equations
\begin{equation*}
\left\{ 
\begin{aligned}
& -\nabla \cdot \left( D_p^2L\left( \nabla u^\ep,\tfrac x\ep \right) \nabla w^\ep \right) = 0  & \mbox{in} & \ U, \\
&  -\nabla \cdot \left( D_p^2\overline{L} \left( \nabla \uhom \right) \nabla \whom \right)  = 0 & \mbox{in} & \ U, \\
& w^\ep, \whom \in f + H^1_0(U),
\end{aligned}
\right.
\end{equation*}
we have the estimate
\begin{multline}
\label{e.DPestimates}
\left\| \nabla w^\ep - \nabla \whom  \right\|_{H^{-1}(U)} 
+ \left\| D_p^2L\left( \nabla u^\ep,\tfrac \cdot\ep \right) \nabla w^\ep - D_p^2\overline{L} \left( \nabla \uhom \right) \nabla \whom   \right\|_{H^{-1}(U)}
\\
\leq
C \left\| \nabla f \right\|_{L^{2+\delta}(U)} \left( 
\ep^{\alpha(d-\sigma)}
+\X\ep^\sigma
\right).  
\end{multline}
\end{theorem}

Recall that $H^{-1}(U)$ is defined as the dual space to $H^1_0(U)$ and that a sequence of $L^2(U)$ functions converges weakly in $L^2(U)$ if and only if they converge strongly in $H^{-1}(U)$  (see Section~\ref{ss.notation}). Therefore the inequality~\eqref{e.DPestimates} should be regarded as a quantification of the weak convergence of the gradient~$\nabla w^\ep$ and flux~$D^2_pL(\nabla u^\ep,x)\nabla w^\ep$ to their homogenized limits. 
Of course, the left side of~\eqref{e.DPestimates} also controls the strong $L^2$ norm of the homogenization error, in view of the following functional inequality (see~\cite[Lemma 1.13]{AKMbook}): there exists~$C(U,d)<\infty$ such that, for every $v \in H^1_0(U)$,
\begin{equation*}
\left\| v \right\|_{L^2(U)} 
\leq 
C
\left\| \nabla v \right\|_{H^{-1}(U)}. 
\end{equation*}
This $L^2$ estimate for the homogenization error  can be upgraded to an estimate in~$L^\infty$ using the De Giorgi-Nash H\"older estimate and an interpolation argument (see the proof of~\cite[Corollary 4.2]{AS}). 

\smallskip

The estimate~\eqref{e.DPestimates} can be expressed in a more familiar way by using Chebyshev's inequality in combination with~\eqref{e.size.X} to obtain that, for each $\sigma < d$, there exist~$\alpha$ and $C<\infty$, as in the statement of the theorem, such that
\begin{equation}
\label{e.Cheb}
\P \left[ \left\| \nabla w^\ep - \nabla \whom  \right\|_{H^{-1}(U)}  > C \ep^{\alpha(d-\sigma)} \left\| \nabla f \right\|_{L^{2+\delta}(U)} \right] \leq C \exp \left( -C^{-1}\ep^{-\sigma} \right),
\end{equation}
with a similar bound holding of course for the fluxes. While the small exponent~$\alpha$ is not explicit and thus this estimate is evidently not optimal in terms of the scaling of the homogenization error, it is optimal in terms of stochastic integrability. Indeed, it is not possible to prove an estimate like~\eqref{e.Cheb} for any exponent $\sigma>d$. One reason for writing the estimate as we have in~\eqref{e.DPestimates}, with the explicit random variable~$\X$, is that it emphasizes its \emph{uniformity} in~$u$, $u^\ep$, $f$, which is important in view applications. In particular, we are able to linearize around an arbitrary solution and solve the linearized equation with an arbitrary Dirichlet condition. While we are interested here in quantitative statements, we remark that the proof of Theorem~\ref{t.linearization} can be modified to give a qualitative homogenization result for the linearized equation under more general, qualitative assumptions (e.g., $\P$ is only stationary and ergodic). 

\smallskip

One of the difficulties encountered in proving  Theorem~\ref{t.linearization} is due to the fact that the coefficients in the linearized equation are not stationary and do not have a finite range of dependence, since~$\nabla u^\ep$ has neither of these properties. It is therefore necessary to first establish that the $D^2_pL\left(\nabla u^\ep(x),\tfrac x\ep \right)$ can be approximated by a matrix-valued random field which is locally stationary and has a finite (mesoscopic) range of dependence. We then  show that the corresponding (local) homogenized matrix is close to $D^2\overline{L}(\nabla u(x))$ and adapt the classical two-scale expansion argument to obtain the theorem. The proof of Theorem~\ref{t.linearization} is given in Section~\ref{s.linearization}.

\smallskip

Our second main result is a large-scale~$C^{0,1}$-type estimate for \emph{differences} of solutions. This can be compared to~\cite[Theorem 1.2]{AS} which proved a similar bounds for solutions. 
Since the difference of two solutions is the solution of a linear equation, we therefore have a priori~$C^{0,\alpha}$ bounds for differences as a consequence of the De Giorgi-Nash estimate. This H\"older regularity with a small exponent is the best \emph{deterministic} bound we can expect to hold on large scales, due to the oscillatory nature of our Lagrangians. However, we show that this estimates can be improved to a $C^{0,1}$-type bound on scales larger than a \emph{random} scale~$\X$ which is finite $\P$--a.s. In fact, its stochastic moments with respect to~$\P$ are very strongly controlled. 

\smallskip

In the following statement and throughout the rest of the paper, for each~$U\subseteq\Rd$ with $|U|<\infty$, we denote
\begin{equation*}
\left\| f \right\|_{\underline{L}^q(U)} 
:= \left( \fint_U \left|f(x) \right|^q\,dx \right)^{\frac1q} 
= \left| U \right|^{-\frac1q} \left\| f \right\|_{L^q(U)}. 
\end{equation*}

\begin{theorem}
[{Large-scale $C^{0,1}$ estimate for differences of solutions}]
\label{t.regularity.differences}
\emph{}\\
Fix $\sigma \in (0,d)$ and $\mathsf{M}\in [1,\infty)$. There exist~$C(\sigma,\mathsf{M},\data) <\infty$ and a random variable~$\X$ satisfying
\begin{equation*}
\X \leq \O_\sigma\left( C \right)
\end{equation*}
such that the following holds. For every $R\geq 2\X$ and $u,v\in H^1(B_R)$ satisfying 
\begin{equation*}
\left\{
\begin{aligned}
& -\nabla \cdot \left( D_pL(\nabla u,x) \right) = 0 & \mbox{in} & \ B_R, \\
& -\nabla \cdot \left( D_pL(\nabla v,x) \right) = 0 & \mbox{in} & \ B_R, \\
& \left\| \nabla u \right\|_{\underline{L}^2(B_R)}, \, 
\left\| \nabla v \right\|_{\underline{L}^2(B_R)} \leq \mathsf{M},
\end{aligned}
\right.
\end{equation*}
and every $ r \in \left[ \X , \tfrac 12 R \right]$, we have the estimate
\begin{equation*}
\left\| \nabla (u-v) \right\|_{\underline{L}^2(B_r)} 
\leq 
 \frac CR \left\| u-v \right\|_{\underline{L}^2(B_R)}. 
\end{equation*}
\end{theorem}

An immediate but very important corollary of Theorem~\ref{t.regularity.differences} is a gradient estimate on the difference of two first-order correctors. Recall that, for each $\xi\in\Rd$, the first-order correctors $\phi_\xi$ is characterized uniquely (up to an additive constant) as the solution of
\begin{equation*}
\left\{ 
\begin{aligned}
& -\nabla \cdot \left( D_pL(\xi + \nabla \phi_\xi(x),x) \right) = 0 \quad \mbox{in}  \ \Rd, \\
& \nabla \phi_\xi \ \  \mbox{is $\Zd$--stationary} 
\quad \mbox{and} \quad
\E\left[ \fint_{[0,1]^d} \nabla \phi_\xi\right] = 0. 
\end{aligned} 
\right. 
\end{equation*}
For each $\xi_1,\xi_2\in\Rd$, the ergodic theorem gives that, $\P$--a.s., we have
\begin{align*}
\limsup_{R\to \infty} 
\frac1{R^2} \fint_{B_R} \left| \left( \xi_1\cdot x + \phi_{\xi_1}(x) \right) - \left( \xi_2\cdot x + \phi_{\xi_2}(x)\right)\right|^2 \,dx
&
= 
\fint_{B_1} \left| (\xi_1-\xi_2) \cdot x \right|^2
\\ &  
\leq \left| \xi_1 - \xi_2 \right|^2.
\end{align*}
Applying Theorem~\ref{t.regularity.differences} to $u(x):= \xi_1\cdot x+\phi_{\xi_1}(x)$ and $v(x):=\xi_2\cdot x+\phi_{\xi_2}(x)$ yields
\begin{equation*}
\left\| \nabla \phi_{\xi_1} - \nabla \phi_{\xi_2} \right\|_{\underline{L}^2(B_\X)} 
\leq 
\left| \xi_1-\xi_2 \right| 
+ 
\left\| u - v \right\|_{\underline{L}^2(B_\X)} 
\leq 
C \inf_{R \geq \X} \frac1R \left\| u-v \right\|_{\underline{L}^2(B_R)}
\leq 
C \left| \xi_1-\xi_2 \right|.
\end{equation*}
Giving up the volume factor, we then get 
\begin{equation}
\left\| \nabla \phi_{\xi_1} - \nabla \phi_{\xi_2} \right\|_{\underline{L}^2([0,1]^d)} 
\leq C \X^{\frac d2}  \left\| \nabla \phi_{\xi_1} - \nabla \phi_{\xi_2} \right\|_{\underline{L}^2(B_\X)}. 
\end{equation}
This is a very strong estimate. Note that if the coefficients are assumed to be H\"older continuous on the unit scale, then one can upgrade this $L^2$-type bound to a pointwise bound with the same right side. For each $s\in (0,2)$, the minimal scale satisfies $\X = \O_{\frac{sd}{2}} (C)$, and we deduce in particular that
\begin{equation}
\label{e.yesdiffcorrectors}
\E \left[
\exp\left( \left\| \nabla \phi_{\xi_1} - \nabla \phi_{\xi_2} \right\|_{\underline{L}^2([0,1]^d)}^{s}  \right)
\right]
\leq 
\E \left[ \exp\left( \X^{\frac{ds}{2}} \right)\right]
\leq C. 
\end{equation}
This kind of estimate on the difference of first-order correctors is exactly what is needed to adapt the optimal homogenization estimates proved for linear equations to the nonlinear setting. Note that the (almost) Gaussian-type  stochastic integrability of the estimate~\eqref{e.yesdiffcorrectors} is (almost) optimal. 

\smallskip

The proof of Theorem~\ref{t.regularity.differences}, which is presented in Section~\ref{s.regularity.differences}, follows a similar idea to the one of~\cite[Theorem 1.2]{AS}. We first obtain an algebraic error estimate for differences of solutions---showing that they can be well-approximated by differences of solutions to the homogenized equation---by interpolating the homogenization error estimates for linearized equations (Theorem~\ref{t.linearization} above) with the ones for the original nonlinear equation (\cite[Theorem 1.2]{AS}, restated below in Theorem~\ref{t.AS.homogenization}). Since the difference of solutions to the homogenized equation satisfies $C^{1,\alpha}$ estimate, this allows us to transfer the higher regularity to the heterogeneous difference via the excess decay argument introduced in~\cite{AS}. The latter is a quantitative version of an idea originating in the work of Avellaneda and Lin~\cite{AL1,AL2} in the periodic case. We remark that, by a similar argument, Theorem~\ref{t.linearization} also yields a large-scale $C^{0,1}$--type estimate for the linearized equation. Since this result is very close in spirit to Theorem~\ref{t.regularity.differences}, we postpone the statement to Section~\ref{s.regularity.differences}.  As with previous large-scale regularity estimates proved in~\cite{AS,AM,GNO2,AKMbook}, the stochastic integrability of the minimal scale~$\X$ is optimal in the sense that~$\X \leq \O_\sigma(C)$ is false, in general, for any exponent $\sigma > d$. See~\cite[Section 3.6]{AKMbook} for details.  

\smallskip

The next main result that we state is the large-scale $C^{1,1}$ estimate, which can be compared to the case $k=1$ of~\cite[Theorem 3.8]{AKMbook}. 
Parts (i) and~(ii) consist of a first-order Liouville theorem, characterizing the set of solutions with at-most linear growth. This is then made quantitative in part (iii) in the form of the large-scale~$C^{1,1}$-type estimate. This generalizes the estimate in the  linear setting, which can be found in~\cite{AKMbook,GNO2}. 

\smallskip

To state the theorem, we must first introduce some additional notation. Given a domain $U\subseteq \Rd$, we denote the set of solutions in $U$ by
\begin{equation*}
\mathcal{L}(U):=
\left\{ 
u\in H^1_{\mathrm{loc}}(U) \,:\,
-\nabla \cdot D_pL(\nabla u,x)  = 0 \ \mbox{in} \ U
\right\}. 
\end{equation*}
We also define
$\mathcal{L}_1$ to be the set of global solutions of the nonlinear equation which exhibit at most linear growth at infinity:
\begin{equation*}
\mathcal{L}_1 :=
\left\{ 
u\in \mathcal{L}(\Rd) \,:\,
\limsup_{r\to \infty} r^{-1} \left\| u \right\|_{\underline{L}^2(B_r)}
< \infty
\right\}. 
\end{equation*}
Note that $\L_1$ is a random object, as it depends on~$L$. 
For each $p\in\Rd$, we let~$\ell_p$ denote the affine function~$\ell_p(x):=p\cdot x$. Since the difference of two elements of $\mathcal{L}(\Rd)$ which exhibits strictly sublinear growth at infinity must be constant, by Theorem~\ref{t.regularity.differences}, the following theorem gives a complete classification of~$\mathcal{L}_1$. 

\begin{theorem}
[{Large-scale $C^{1,1}$-type estimate}]
\label{t.C11estimate} 
Fix $\sigma \in \left(0,d\right)$ and $\mathsf{M} \in [1,\infty)$. There exist~$\delta(\sigma,d,\Lambda)\in \left( 0, \frac12 \right]$, $C(\mathsf{M},\sigma,\data)<\infty$ and a random variable $\X_\sigma$ which satisfies the estimate
\begin{equation}
\label{e.X}
\X_\sigma \leq \O_{\sigma}\left(C\right)
\end{equation}
such that the following statements are valid. 
\begin{enumerate}
\item[{$\mathrm{(i)}$}] 
For every $u \in \mathcal{L}_1$ satisfying $\limsup_{r \to \infty} \frac1r \left\|  u - (u)_{B_r} \right\|_{\underline{L}^2(B_r)}   \leq \mathsf{M}$,
there exist an affine function $\ell$ such that, for every $R\geq \X_\sigma$,
\begin{equation*} 
\label{e.liouvillec0}
\left\| u - \ell \right\|_{\underline{L}^2(B_R)} \leq C R^{1-\delta} .
\end{equation*} 

\item[{$\mathrm{(ii)}$}]
For every $p\in B_\mathsf{M}$, there exists $u\in \mathcal{L}_1$ satisfying, for every $R\geq \X_\sigma$,
\begin{equation*} 
\label{e.liouvillec1}
\left\| u - \ell_p \right\|_{\underline{L}^2(B_R)} \leq C R^{1-\delta} .
\end{equation*}

\smallskip

\item[{$\mathrm{(iii)}$}]
For every $R\geq \X_s$ and $u\in \mathcal{L}(B_R)$ satisfying 
$\frac1R \left\| u - (u)_{B_R} \right\|_{\underline{L}^2 \left( B_{R} \right)}  \leq \mathsf{M}$,  there exists $\phi \in \mathcal{L}_1$ such that, for every $r \in \left[ \X_\sigma,  R \right]$, 
\begin{equation}
\label{e.C11}
\left\| u - \phi \right\|_{\underline{L}^2(B_r)} \leq C \left( \frac r R \right)^{2} \inf_{\psi\in\L_1} 
\left\| u - \psi \right\|_{\underline{L}^2(B_R)}.
\end{equation}
\end{enumerate}
\end{theorem}

In the case of deterministic, periodic coefficient fields, the Liouville theorem (parts~(i) and~(ii) of Theorem~\ref{t.C11estimate}) was obtained 
by Avellaneda and Lin~\cite{AL4} for linear equations and subsequently generalized by Moser and Struwe~\cite{MS} to the nonlinear setting. Obtaining a $C^{1,\alpha}$--type version of statement~(iii), for a tiny $\alpha(d,\Lambda)>0$, is not difficult, and in fact we believe it can be proved even in the random case with a qualitative argument. Obtaining the full $C^{1,1}$ regularity is more difficult (even in the periodic setting), requires a quantitative argument and crucially relies on Theorem~\ref{t.regularity.differences}.

\smallskip

We turn to the last main result of the paper concerning the improved regularity of~$\overline{L}$. As we show in Section~\ref{ss.detlin}, under the assumption (L1) that $L(\cdot,x) \in C^{2,\gamma}$, one can use deterministic regularity estimates (either of the De Giorgi-Nash or Meyers estimates will do) to obtain a tiny bit better regularity for the homogenized Lagrangian, namely $\overline{L}\in C^{2,\beta}$ for a tiny exponent $0< \beta \ll \gamma$. (Note that this observation is necessary even to ensure that the statement of Theorem~\ref{t.linearization} is coherent.) However, once we have proved Theorem~\ref{t.regularity.differences}, we can upgrade the H\"older exponent of $D^2\overline{L}$ all the way up to the exponent~$\gamma$, confirming that $\overline{L}$ is as regular as $L(\cdot,x)$, at least up to $C^{2,\gamma}$. This result requires an additional assumption on the spatial regularity of $L$, which is stated in Section~\ref{s.Lreg}.

\begin{theorem}
\label{t.Lreg} 
In addition to the standing assumptions, suppose that~\eqref{e.extra.assump} holds for some exponent $\gamma \in (0,1]$. Then~$\overline{L} \in C^{2,\gamma}_{\mathrm{loc}}(\Rd)$ and, for every $\mathsf{M} \in [1,\infty)$, 
\begin{equation}
\label{e.Lreg}
\left[ D^2\overline{L} \right]_{C^{0,\gamma}(B_{\mathsf{M}})}
\leq 
C(\mathsf{M},\gamma,\data)<\infty.
\end{equation}
\end{theorem}

The proof of Theorem~\ref{t.Lreg} appears in Section~\ref{s.Lreg}. 

\smallskip

We expect the higher regularity for $\overline{L}$ (i.e., $C^{k,\beta}$ for $k\geq 3$), under appropriate additional regularity assumptions on~$L$, to be a natural consequence of a large-scale regularity theory for higher-order linearized equations. This will be elaborated in a future paper.

\subsection{Outline of the paper}
In the next section we give some notation, recall some previous results and show that~$\overline{L} \in C^{2,\beta}$ for a tiny~$\beta>0$. The proofs of Theorems~\ref{t.linearization},~\ref{t.regularity.differences},~\ref{t.C11estimate} and~\ref{t.Lreg} are given in Sections~\ref{s.linearization},~\ref{s.regularity.differences},~\ref{s.C11estimate} and~\ref{s.Lreg}, respectively. In Appendix~\ref{ap.ASestimates} we recall some homogenization estimates from~\cite{AS} which are needed in Section~\ref{s.linearization}.

\section{Preliminaries}
\label{s.prelim}

In this section we introduce some notation and state some previous quantitative homogenization results which are used throughout the paper. We also prove some preliminary deterministic estimates regarding the approximation of differences of solutions by the solutions of linearized equations. As a consequence, we show that~$\overline{L}$ belongs to~$C^{2,\beta}$ for a small exponent~$\beta>0$.

\subsection{Notation}
\label{ss.notation}
If $E \subseteq \Rd$, the Lebesgue measure of~$E$ is denoted by~$|E|$. If $U\subseteq\Rd$ is a domain with $|U|<\infty$ we define the normalized domain  
\begin{equation*}
U_0:= \left| U \right|^{-\frac1d} U. 
\end{equation*}
Various of the parameters in our statements depend on~$U$, but in most cases this dependence is only on~$U_0$ and therefore will be invariant under changes of scale. We denote the cube of side length $3^n$ centered at the origin by
\begin{equation*}
\cu_n := \left( -\frac12 3^n ,\frac12 3^n \right)^d. 
\end{equation*}
If $U\subseteq\Rd$ is a domain, we define the norm 
\begin{equation*}
\left\| u \right\|_{H^1(U)}:= \left( \left\| u \right\|_{L^2(U)}^2 + \left\| \nabla u \right\|_{L^2(U)}^2 \right)^{\frac12}
\end{equation*}
and define $H^1(U)$ and $H^1_0(U)$, respectively, to be the completion of $C^\infty(\overline{U})$ and $C^\infty_c(U)$, respectively, with respect to the norm~$\left\| \cdot \right\|_{H^1(U)}$. We also define the space $H^{-1}(U)$ to be the completion of $C^\infty(\overline{U})$ with respect to the norm 
\begin{equation*}
\left\| u \right\|_{H^{-1}(U)}
:= 
\sup\left\{ 
\left| \int_U uv \right| \,:\, v\in H^1_0(U), \ \left\| v \right\|_{H^1(U)} \leq 1
\right\}. 
\end{equation*}
If $|U|<\infty$, we define 
\begin{equation*}
\fint_U f  := \frac1{|U|} \int_U f.
\end{equation*}
We also sometimes write~$\left( f \right)_U := \fint_U f$. It is often convenient to work with the following scale-invariant norms, defined for domains~$U$ with~$|U|<\infty$:
\begin{equation*}
\left\| u \right\|_{\underline{L}^p(U)}
:= \left( \fint_U \left|u\right|^p \right)^{\frac1p} = \left| U \right|^{-\frac1p} \left\| u \right\|_{L^p(U)},
\end{equation*}
\begin{equation*}
\left\| u \right\|_{\underline{H}^1(U)}
:=
\left( |U|^{-\frac2d} \left\| u \right\|_{\underline{L}^2(U)}^2 + \left\| \nabla u \right\|_{\underline{L}^2(U)}^2 \right)^{\frac12},
\end{equation*}
\begin{equation*}
\left\| u \right\|_{\underline{H}^{-1}(U)}
:= 
\sup\left\{ 
\left| \fint_U uv \right| \,:\, v\in H^1_0(U), \ \left\| v \right\|_{\underline{H}^1(U)} \leq 1
\right\}.
\end{equation*}
It is useful to note these ``underlined'' norms have the following scaling properties:  if $u_r(x):= ru( x/r)$ for $r>0$, then
\begin{equation}
\label{e.underlinedscalings}
\left\{
\begin{gathered}
\left\| u_r \right\|_{\underline{L}^p(rU)}
=
r \left\| u \right\|_{\underline{L}^p(U)}, 
\ \
\left\| u_r \right\|_{\underline{H}^1(rU)}
=
\left\| u \right\|_{\underline{H}^1(U)}, 
\ \
\left\| u_r \right\|_{\underline{H}^{-1}(rU)}
=
\left\| u \right\|_{\underline{H}^{-1}(U)}, 
\\
\left\| \nabla u_r \right\|_{\underline{L}^p(rU)}
=
\left\| \nabla u \right\|_{\underline{L}^p(U)},
\ \ 
\left\| \nabla u \right\|_{\underline{H}^{-1}(U)}
=
\frac1r \left\| \nabla u_r \right\|_{\underline{H}^{-1}(rU)}.
\end{gathered}
\right.
\end{equation}
We also require a notion of $H^{-1}$ norm defined by testing against any $H^1$ function, not simply $H^1_0$ functions. It is defined (in its normalized version) by:
\begin{equation}
\left\| u \right\|_{\Hminusul(U)}
:=
\sup\left\{ 
\left| \fint_U uv \right| \,:\, v\in H^1(U), \ \left\| v \right\|_{\underline{H}^1(U)} \leq 1
\right\}. 
\end{equation}
Directly from its definition we find that the $\Hminusul(U)$ norm has the following useful subadditivity property: if $U$ is the disjoint union of $\{ V_i\}_{i\in\{1,\ldots,N\}}$ (up to a set of measure zero), then
\begin{equation}
\label{e.Hminusul.subadd}
\left\| u \right\|_{\Hminusul(U)}
\leq 
\sum_{i=1}^N
\frac{|V_i|}{|U|}
\left\| u \right\|_{\Hminusul(V_i)}
\end{equation}
It is clear that 
\begin{equation}
\label{e.H1L2dumbdumb}
\left\| u \right\|_{\underline{H}^{-1}(U)} 
\leq 
\left\| u \right\|_{\Hminusul(U)}
\leq 
|U|^{\frac 1d} \left\| u \right\|_{\underline{L}^2(U)}.
\end{equation}
We will also need that the $\left\| \cdot\right\|_{\underline{H}^{-1}(U)}$ obeys the following product rule in every bounded Lipschitz domain~$U$, which can be checked directly from the definition of~$\left\| \cdot\right\|_{\underline{H}^{-1}}$ and the Poincar\'e inequality: there exists $C(U_0,d)<\infty$ such that, for every $f\in W^{1,\infty}(U)$ and $g\in H^{-1}(U)$,
\begin{equation}
\label{e.computeHm1}
\left\| fg \right\|_{\Hminusul(U)} 
\leq 
C\left( \left\| f \right\|_{L^\infty} + \left| U \right|^{\frac1d} \left\| \nabla f \right\|_{L^\infty(U)} \right)
\left\| g \right\|_{\Hminusul(U)},
\end{equation}
and the inequality holds also if $\underline{H}^{-1}(U)$ replaces $\Hminusul(U)$.

\smallskip

As mentioned above, if $X$ is a random variable and $\sigma,\theta\in (0,\infty)$, then we use 
\begin{equation*} \label{}
X \leq \O_\sigma(\theta)
\end{equation*}
as shorthand notation for the statement that 
\begin{equation}
\label{e.chebyforward}
\E \left[ \exp \left( \left( \frac{X_+}{\theta} \right)^\sigma \right) \right] \leq 2. 
\end{equation}
It roughly means that ``$X_+$ is at most of order $\theta$ with stretched exponential tails with exponent $\sigma$.'' Indeed, by Chebyshev's inequality, 
\begin{equation} 
\label{e.Cheby}
X \leq \O_\sigma(\theta) \implies \forall \lambda>0, \ \P \left[ X > \lambda \theta \right] \leq 2\exp\left( -\lambda^\sigma \right). 
\end{equation}
The converse of this statement is almost true: for every $\theta\geq 0$, 
\begin{equation}
\label{e.chebyconverse}
\forall \lambda \geq 0, 
\quad 
\P\left[ X \geq \lambda \theta \right] \leq \exp \Ll( - \lambda^\sigma \Rr)  
\implies  
X \leq \O_\sigma \left(2^\frac 1 \sigma \, \theta\right).
\end{equation}
This can be obtained by integration. We also use the notation
\begin{equation*} \label{}
X = \O_\sigma(\theta) 
\iff 
X \leq \O_\sigma(\theta) \ \mbox{and} \ -X \leq \O_\sigma(\theta). 
\end{equation*}
We also write $X\leq Y + \O_\sigma(\theta)$ to mean that $X-Y \leq \O_\sigma(\theta)$ as well as $X = Y+\O_\sigma(\theta)$ to mean that $X-Y = \O_\sigma(\theta)$. If $\sigma\in [1,\infty)$, then Jensen's inequality gives us a triangle inequality for $\O_\sigma(\cdot)$ in the following sense: for any measure space $(E,\mathcal{S},\mu)$, measurable function $K : E \to (0,\infty)$ and jointly measurable family $\left\{ X(z) \right\}_{z \in E}$ of nonnegative random variables, we have 
\begin{equation} 
\label{e.Osums}
\forall z\in E, \ X(z) \leq \O_\sigma(K(z))
\implies
\int_{E} X\,d\mu \leq \O_\sigma\left( \int_E K \,d\mu \right). 
\end{equation}
If $\sigma\in (0,1]$, then the statement is true after adding a prefactor constant $C(\sigma)\in [1,\infty)$ to the right side. 
We refer to~\cite[Appendix A]{AKMbook} for proofs of these facts.

\subsection{Previous quantitative homogenization results}

As far as we are aware, the papers~\cite{AS,AM} contain the only previous quantitative stochastic homogenization results for nonlinear equations. In this subsection, we recall several of the main results of~\cite{AS} which are needed in this paper, particular in the next section. 

\smallskip

The first we present is essentially the same as~\cite[Theorem 1.1]{AS}, although its statement is slightly differently than the latter and can be found in~\cite[Chapter 11]{AKMbook}. Note that the proof given in~\cite{AKMbook} follows the same high-level outline of the one in~\cite{AS}, but is much more efficient.

\begin{theorem}[{Quantitative homogenization~\cite[Theorem 11.10]{AKMbook}}]
\label{t.AS.homogenization}
\emph{}\\
Fix $\sigma \in (0,d)$, $\delta\in \left(0,\tfrac 12\right]$, $\mathsf{M}\in [1,\infty)$ and a Lipschitz domain $U\subseteq B_1$. There exist~{$\alpha(U,\data)>0$}, $C(\sigma,U,\delta,\data)<\infty$ and a random variable~$\X(\sigma,\delta,\mathsf{M},U,\data)$ satisfying the bound
\begin{equation}
\label{e.size.X.AS}
\X = \O_1\left(C \right)
\end{equation}
such that the following holds. 
Fix $\ep\in (0,1]$, a pair~$u,u^\ep \in W^{1,2+\delta}(U)$ satisfying
\begin{equation} \label{e.AS.homogenization.eqs}
\left\{ 
\begin{aligned}
& -\nabla \cdot \left( D_pL\left( \nabla u^\ep,\tfrac x\ep \right) \right) = 0  & \mbox{in} & \ U, \\
&  -\nabla \cdot \left( D_p\overline{L} \left( \nabla u \right) \right)  = 0 & \mbox{in} & \ U, \\
& u^\ep - u \in H^1_0(U), \\
& \left\| \nabla u \right\|_{L^{2+\delta}(U)} \leq \mathsf{M}.
\end{aligned}
\right.
\end{equation}
Then we have the estimate
\begin{equation}
\label{e.DPestimates.AS}
\left\| \nabla u^\ep - \nabla u  \right\|_{H^{-1}(U)} 
+ \left\| D_p L\left( \nabla u^\ep,\tfrac \cdot\ep \right)  - D_p\overline{L} \left( \nabla u \right)    \right\|_{H^{-1}(U)}
\leq
C \mathsf{M} \left( 
\ep^{\alpha(d-\sigma)}
+\X\ep^\sigma
\right).  
\end{equation}
\end{theorem}

The proof of Theorem~\ref{t.AS.homogenization} is based on an analysis of the subadditive quantity, introduced previously in~\cite{DM1}, defined by
\begin{equation*}
\nu(U,\xi) 
:=
\inf_{u\in \xi\cdot x + H^1_0(U)} 
\fint_U L\left( \nabla u(x),x \right)\,dx.
\end{equation*}
The effective Lagrangian is defined through the limit
\begin{equation*}
\overline{L}(\xi):= \lim_{n\to \infty} \E \left[ \nu(\cu_n,\xi) \right],
\end{equation*}
and this limit exists since the sequence $n\mapsto \E\left[ \nu(\cu_n,\xi) \right]$ is nonincreasing by the subadditivity of $\nu(\cdot,\xi)$. It was shown by an iterative argument in~\cite[Theorem 3.1]{AS} that for every $\sigma\in (0,d)$, there exists $\alpha (d,\Lambda) \in \left(0,\tfrac12 \right]$ and $C(\sigma,d,\Lambda)<\infty$ such that, for every $n\in\N$,
\begin{equation}
\label{e.subadd.conv}
\left| \nu(\cu_n,\xi) - \overline{L}(\xi) \right| 
\leq 
C3^{-n\alpha(d-\sigma)} + \O_1\left( C3^{-n\sigma} \right).
\end{equation}
This estimate then implies Theorem~\ref{t.AS.homogenization} by a quantitative two-scale expansion argument, as demonstrated in~\cite{AS}. 

\smallskip

It is sometimes useful to state Theorem~\ref{t.AS.homogenization} in a slightly different way, by indicating a \emph{random scale} above which homogenization holds with a deterministic estimate, rather that giving an estimate with a random right-hand side as in~\eqref{e.DPestimates.AS}. We present such a statement in the following corollary, which is an immediate consequence of the previous theorem.

\begin{corollary}
 \label{c.minimalscale}
Let $\sigma \in (0,d)$, $\delta\in \left(0,\tfrac 12\right]$ and $\mathsf{M}\in [1,\infty)$. There exist~{$\alpha(\delta,\data)\in \left(0,\tfrac12\right]$}, $C(\sigma,\delta,\mathsf{M},\data)<\infty$ and a random variable~$\X_\sigma$, satisfying the bound
\begin{equation}
\label{e.size.X_sigma.AS}
\X_\sigma = \O_\sigma\left(C \right)
\end{equation}
such that the following statement holds. For every~$r\in [\X_\sigma,\infty)$ and $f \in W^{1,2+\delta}(B_r)$ satisfying the bound $\left\| \nabla f \right\|_{\underline{L}^{2+\delta}(B_r)} \leq \mathsf{M}$ and every pair $u,\overline{u}\in H^1(B_r)$ satisfying 
\begin{equation} \label{e.AS.homogenization.eqs2}
\left\{ 
\begin{aligned}
& -\nabla \cdot \left( D_pL\left( \nabla u,x \right) \right) = 0  & \mbox{in} & \ B_r, \\
&  -\nabla \cdot \left( D_p\overline{L} \left( \nabla \overline{u} \right) \right)  = 0 & \mbox{in} & \ B_r, \\
& u , \overline{u} \in f +  H^1_0(B_r),
\end{aligned}
\right.
\end{equation}
we have the estimate
\begin{equation*} 
\frac1r \left\| u -  \overline{u}  \right\|_{\underline{L}^2(B_r)} \leq C r^{-\alpha(d-\sigma)}. 
\end{equation*}
\end{corollary}
\begin{proof}
The passage from Theorem~\ref{t.AS.homogenization} to the statement of the corollary is essentially identical to the argument of~\cite[Proposition 3.2]{AKMbook} or the first paragraph of the proof of~\cite[Theorem 1.2]{AS}. 
\end{proof}

The final result from~\cite{AS} we need is the following large-scale $C^{0,1}$--type estimate (see~\cite[Theorem 1.2]{AS}). We denote by $\mathcal{P}_1$ the linear space of affine functions. 

\begin{theorem}
[{Large-scale $C^{0,1}$-type estimate}]
\label{t.AS.regularity}
Fix $\sigma \in (0,d)$ and $\mathsf{M}\in [1,\infty)$. There exist constants~$C(\sigma,\mathsf{M},\data) <\infty$, $\alpha(d,\Lambda) \in \left(0,\tfrac12\right]$ and a random variable $\X(\sigma,\mathsf{M},\data)$
satisfying
$
\X \leq \O_\sigma\left( C \right)
$
such that the following holds. For every $R\geq 2\X$ and $u\in H^1(B_R)$ satisfying 
\begin{equation*}
\left\{
\begin{aligned}
& -\nabla \cdot \left( D_pL(\nabla u,x) \right) = 0 & \mbox{in} & \ B_R, \\
&  \frac{1}{R} \left\| u - \left( u \right)_{B_R} \right\|_{\underline{L}^2(B_R)}
\leq \mathsf{M},
\end{aligned}
\right.
\end{equation*}
and every $r\in \left[ \X , \tfrac 12 R \right]$, we have the estimates
\begin{equation*}
\left\| \nabla u \right\|_{\underline{L}^2(B_r)}
\leq 
 \frac{C}{R} \left\| u - \left( u \right)_{B_R} \right\|_{\underline{L}^2(B_R)}
\end{equation*}
and
\begin{equation*}
\inf_{\ell \in \mathcal{P}_1} \left\| u - \ell \right\|_{\underline{L}^2(B_r)} 
\leq 
C \left( \frac{r}{R} \right)^{1+\alpha} 
\inf_{\ell \in \mathcal{P}_1} \left\| u - \ell \right\|_{\underline{L}^2(B_R)} 
+
C r^{1-\alpha(d-\sigma)} .
\end{equation*}
\end{theorem}

\subsection{Deterministic linearization estimates}
\label{ss.detlin}

In this subsection we prove some deterministic estimates which measure the error in approximating the difference of two solutions to the nonlinear equation by the linearized equation. As a consequence, we obtain the~$C^2$ regularity of the homogenized Lagrangian~$\overline{L}$. 

\begin{lemma}[Approximation of differences by linearization]
\label{l.deterministic.linearization}
Fix $\delta >0$, a bounded Lipschitz domain~$U$ and $f\in W^{1,2+\delta}(U)$. There exist constants~$\beta(\delta,U_0,\data) \in \left(0,\tfrac12\right]$ and  $C(\delta,U_0,\data)<\infty$ such that the following statement is valid. 
If $u,v\in H^1(U)$ solve
\begin{equation*}
\left\{ 
\begin{aligned}
& -\nabla \cdot \left( D_pL(\nabla u,x) \right) 
= -\nabla \cdot \left( D_pL(\nabla v,x) \right) &  \mbox{in}  & \ U, \\
& u - v = f & \mbox{on} & \ \partial U
\end{aligned}
\right.  
\end{equation*}
and $w\in H^1(U)$ is the solution of the linearized problem
\begin{equation*}
\left\{
\begin{aligned}
& -\nabla \cdot \left( D_p^2 L(\nabla u,x) \nabla w \right)  = 0  & \mbox{in} & \ U, \\
& w = f & \mbox{on} & \ \partial U,
\end{aligned}
\right.
\end{equation*}
then
\begin{equation*}
\left\| \nabla u - \nabla v - \nabla w \right\|_{\underline{L}^2(U)} 
\leq 
C \left\| \nabla f \right\|_{\underline{L}^{2+\delta}(U)}^{1+\beta}. 
\end{equation*}
\end{lemma}
\begin{proof}
We first observe that the difference~$\tilde{w}:= u - v$ satisfies the linear equation
\begin{equation*}
-\nabla \cdot\left( \tilde{\a} \nabla \tilde{w} \right) = 0\quad \mbox{in} \ U, 
\end{equation*}
where the coefficients~$\tilde{\a}$ are given by 
\begin{equation*}
\tilde{\a}(x) := 
\int_0^1 
D^2_pL 
\left(t\nabla u(x) + (1-t) \nabla v(x), x \right)\,dt.
\end{equation*}
In particular, by the Meyers estimate--and shrinking $\delta$, if necessary, so that it is at most the Meyer exponent, which we note depend on $(\delta,U_0,\data)$--we have that 
\begin{equation}
\label{e.stupidmeyers}
\left\| \nabla \tilde{w} \right\|_{L^{2+\delta}(U)} 
\leq 
C \left\| \nabla f \right\|_{L^{2+\delta}(U)}. 
\end{equation}
We compare this to~$w$ by defining~$z:=w-\tilde{w}$ and observing that~$z\in H^1_0(U)$ satisfies
\begin{equation*}
-\nabla \cdot\left( D^2_pL(\nabla u,x) \nabla z \right) = 
-\nabla \cdot \left( \left( \tilde{\a} - D^2_pL(\nabla u,x) \right) \nabla \tilde{w} \right) 
\quad \mbox{in} \ U. 
\end{equation*}
Thus the energy estimate gives us
\begin{equation}
\label{e.linearbasic}
\left\| \nabla u - \nabla v - \nabla w \right\|_{L^{2}(U)}
=
\left\| \nabla z \right\|_{L^{2}(U)}
\leq 
C \left\| \left( \tilde{\a} - D^2_pL(\nabla u,\cdot) \right) \nabla \tilde{w}  \right\|_{L^{2}(U)}. 
\end{equation}
To estimate the term on the right side, 
observe that 
\begin{align*}
\left| \tilde{\a}(x) - D^2_pL(\nabla u(x),x) \right| 
&
\leq 
\left[ D^2_pL(\cdot,x) \right]_{C^{0,\gamma}(\Rd)}
\left| \nabla u(x) - \nabla v(x)\right|^\gamma
\\ & 
\leq 
C \left| \nabla u(x) - \nabla v(x)\right|^\gamma.
\end{align*}
Therefore we have the~$L^1$ bound 
\begin{equation*}
\left\| \tilde{\a} - D^2_pL(\nabla u,\cdot) \right\|_{\underline{L}^1(U)} 
\leq 
C \left\| \nabla u - \nabla v \right\|_{\underline{L}^2(U)}^{\gamma}.
\end{equation*}
We also have the trivial~$L^\infty$ bound
\begin{equation*}
\left\| \tilde{\a}  - D^2_pL(\nabla u,\cdot) \right\|_{L^\infty(U)} 
\leq
\left\| \tilde{\a}  \right\|_{L^\infty(U)} + \left\| D^2_pL(\nabla u,\cdot) \right\|_{L^\infty(U)} 
\leq C. 
\end{equation*}
By interpolation, we therefore get, for every $q\in [1,\infty)$, 
\begin{equation*}
\left\| \tilde{\a} - D^2_pL(\nabla u,\cdot) \right\|_{\underline{L}^q(U)}
\leq 
C \left\| \nabla u  - \nabla v  \right\|_{\underline{L}^2(U)}^{\frac{\gamma}{q}}
= 
C \left\| \nabla \tilde{w} \right\|_{\underline{L}^2(U)}^{\frac{\gamma}{q}}. 
\end{equation*}
Therefore, by the H\"older inequality, 
\begin{align*}
\left\| \left( \tilde{\a} - D^2_pL(\nabla u,\cdot) \right) \nabla \tilde w  \right\|_{\underline{L}^2(U)}
& \leq 
\left\| \tilde{\a} - D^2_pL(\nabla u,\cdot) \right\|_{\underline{L}^{\frac{4+2\delta}{\delta}}(U)} 
\left\| \nabla \tilde{w} \right\|_{\underline{L}^{2+\delta}(U)}
\\ & 
\leq 
C\left\| \nabla \tilde{w} \right\|_{\underline{L}^{2+\delta}(U)}^{1+\gamma\delta/(4+2\delta)}. 
\end{align*}
Combining this with~\eqref{e.stupidmeyers} and~\eqref{e.linearbasic} completes the proof. 
\end{proof}

A consequence of the previous lemma is the~$C^2$ regularity of the homogenized Lagrangian~$\overline{L}$. 

\begin{proposition}
\label{p.nu.C2beta}
Let $U\subseteq\Rd$ be a bounded Lipschitz domain. 
There exist an exponent~$\beta(U_0,\data) \in \left(0,\tfrac12\right]$ and a constant~$C(U_0,\data)<\infty$ such that the mapping $\xi \mapsto \nu(U,\xi)$ belongs to $C^{2,\beta}(\Rd)$ and 
\begin{equation}
\label{e.nu.C2beta}
\left\| D^2_\xi \nu(U,\cdot) \right\|_{C^{0,\beta}(\Rd)} \leq C. 
\end{equation}
Moreover, there exist~$\beta(\data) \in \left(0,\tfrac12\right]$ and $C(\data)<\infty$ such that $\overline{L}\in C^{2,\beta}(\Rd)$ and  
\begin{equation}
\label{e.Lbar.C2beta}
\left[ D^2 \overline{L} \right]_{C^{0,\beta}(\Rd)} \leq C. 
\end{equation}
\end{proposition}
\begin{proof}
Fix a bounded Lipschiz domain $U\subseteq\Rd$ and $\xi\in\Rd$. Let $v(\cdot,U,\xi)$ denote the solution of the Dirichlet problem (recall here that $\ell_\xi(x):=\xi\cdot x$)
\begin{equation*}
\left\{ 
\begin{aligned}
& - \nabla \cdot \left( D_pL(\nabla v(\cdot,U,\xi),x \right) = 0 & \mbox{in} & \ U, \\
& v(\cdot,U,\xi) = \ell_\xi & \mbox{on} & \ \partial U.
\end{aligned}
\right.
\end{equation*}
Define
\begin{equation*}
\tilde{\a}_\xi(x):= D^2_pL(\nabla v(x,U,\xi),x)
\end{equation*}
and, for each $e\in\Rd$, let $\tilde{w}_e(\cdot,U,\xi)$ solve the linearized problem 
\begin{equation*}
\left\{ 
\begin{aligned}
& - \nabla \cdot \left( \tilde{\a}_\xi (x)\nabla \tilde{w}_e(\cdot,U,\xi) \right) = 0 & \mbox{in} & \ U, \\
& \tilde{w}_e(\cdot,U,\xi) = \ell_e & \mbox{on} & \ \partial U.
\end{aligned}
\right.
\end{equation*}
For~$C(U_0,d,\Lambda)<\infty$, we have 
\begin{equation}
\label{e.easyboundwtilde}
\left\| \nabla \tilde{w}_e(\cdot,U,\xi) \right\|_{\underline{L}^2(U)} \leq C. 
\end{equation}
According to Lemma~\ref{l.deterministic.linearization}, 
there exist~$\beta(U_0,\data)>0$ and~$C(U_0,\data)<\infty$ such that, for every $\xi, \xi'\in\Rd$, 
\begin{equation}
\label{e.vC1alpha.est}
\left\| \nabla v(\cdot,U,\xi') - \nabla v(\cdot,U,\xi) - \sum_{i=1}^d (\xi_i'-\xi_i) \nabla \tilde{w}_{e_i}(\cdot,U,\xi) \right\|_{\underline{L}^2(U)} \leq C \left| \xi-\xi' \right|^{1+\beta}. 
\end{equation}
In particular, $\xi\mapsto \nabla v(\cdot,U,\xi)$ is~$C^{1,\beta}$ mapping from $\Rd$ into $L^2(U)$ and 
\begin{equation*}
\nabla \tilde{w}_{e_i}(\cdot,U,\xi)
= \partial_{\xi_i} \nabla v(\cdot,U,\xi).
\end{equation*}
In fact, in view of the formula
\begin{equation*}
D_\xi \nu(U,\xi) = \fint_U D_p L\left( \nabla v(x,U,\xi),x \right)\,dx
\end{equation*}
and the regularity assumption on $L$, 
the estimate~\eqref{e.vC1alpha.est} implies~\eqref{e.nu.C2beta} with $\beta\wedge \gamma$ in place of~$\beta$. Indeed, differentiating the previous display yields 
\begin{equation}
\label{e.D2nuUiden}
D_\xi^2 \nu(U,\xi) = \fint_U D_p^2 L\left( \nabla v(x,U,\xi),x \right) D_\xi  \nabla  v(x,U,\xi) \,dx.
\end{equation}
It is clear from~\eqref{e.easyboundwtilde} and this expression that~$\left| D^2_\xi\nu(U,\cdot)\right|$ is bounded on $\Rd$ and, for each $\xi,\xi'\in\Rd$, 
\begin{align*}
\lefteqn{
\left| D_\xi^2 \nu(U,\xi)  - D_\xi^2 \nu(U,\xi')  \right| 
} \quad 
\\ & 
\leq 
\fint_U \left|  
D_p^2 L\left( \nabla v(x,U,\xi),x \right) 
- 
D_p^2 L\left( \nabla v(x,U,\xi'),x \right)
\right| \left|  D_\xi \nabla  v(x,U,\xi) \right|\,dx 
\\ & \quad 
+ \fint_U \left|  
D_p^2 L\left( \nabla v(x,U,\xi'),x \right)  \right| \left|  D_\xi \nabla  v(x,U,\xi) -  D_\xi \nabla  v(x,U,\xi') \right|\,dx
\\ & 
\leq C \mathsf{M}_0 \fint_U \left| \nabla v(x,U,\xi) - \nabla v(x,U,\xi') \right|^\gamma  \left|  D_\xi \nabla  v(x,U,\xi) \right|\,dx 
\\ & \quad 
+ C  \fint_{U}  \left|  D_\xi \nabla  v(x,U,\xi) -  D_\xi \nabla  v(x,U,\xi') \right|\,dx
\\ & 
\leq C\left( \left| \xi - \xi' \right|^\gamma+ \left| \xi - \xi' \right|^\beta \right). 
\end{align*}
This proves the desired bound on $\left[ D^2\nu(U,\cdot) \right]_{C^{0,\beta}(\Rd)}$. Since $\nu(U,\cdot)$ is bounded in $C^{1,1}(\Rd)$ by a constant $C(\Lambda)<\infty$ (see for instance~\cite[Lemma 11.2]{AKMbook}), we obtain~\eqref{e.nu.C2beta}.

\smallskip

The last statement concerning the~$C^{2,\beta}$ regularity of~$\overline{L}$ follows immediately from~\eqref{e.nu.C2beta} and the pointwise, deterministic limit $\overline{L}(\xi) = \lim_{n\to \infty} \E\left[ \nu(\cu_n,\xi)\right]$.
\end{proof}


\smallskip

We next combine Proposition~\ref{p.nu.C2beta} with~\eqref{e.subadd.conv} to obtain the following estimate on the convergence of $D^2\nu(\cu_n,\cdot)$ to $D^2\overline{L}$. 

\begin{lemma}
\label{l.nuconvtoLinC2}
Let $\sigma \in (0,d)$ and~$\mathsf{M}\in [1,\infty)$. 
There exist~$\alpha(d,\Lambda,\gamma)\in \left(0,\tfrac12\right]$ and constant~$C(\sigma,\mathsf{M},\data)<\infty$ such that, for every $n\in\N$,
\begin{equation}
\label{e.nuconvC2}
\sup_{\left| \xi \right| \leq \mathsf{M}} 
\left| D^2\overline{L}(\xi) - D^2_p \nu(\cu_n,\xi) \right| \leq C3^{-n\alpha(d-\sigma)} + \O_1\left( C3^{-n\sigma} \right).
\end{equation}
\end{lemma}
\begin{proof}
The argument is based on the elementary fact (by the Arzela-Ascoli theorem, for instance) that if $\{ f_k\}_{k\in\N}$ is a sequence of functions uniformly bounded in $C^{2,\beta}(B_1)$ and converges pointwise to zero, then $D^2f_k \to 0$ uniformly in~$B_1$. This can be seen from a more quantitative perspective as a consequence of the following interpolation inequality: for each~$\beta\in (0,1]$, there exists a constant~$C(\beta,d)<\infty$ such that 
\begin{equation*}
\left\| D^2 f \right\|_{L^\infty(B_1)} 
\leq 
C \left\| f \right\|_{L^1(B_1)}^{\frac{\beta}{2+\beta}} 
\left\| f \right\|_{C^{2,\beta}(B_1)}^{\frac{2}{2+\beta}} .
\end{equation*}
Applying this to the difference of $\nu(\cu_n,\cdot) - \overline{L}$ in the ball $B_{\mathsf{M}}$ yields, in view of~\eqref{e.subadd.conv},~\eqref{e.nu.C2beta} and~\eqref{e.Lbar.C2beta}, an exponent~$\beta(\data)\in \left(0,\tfrac12 \right]$ and~$C(\sigma',\mathsf{M},\data)<\infty$ such that 
\begin{align*}
\lefteqn{
\sup_{\left| \xi \right| \leq \mathsf{M}} 
\left| D^2\overline{L}(\xi) - D^2_p \nu(\cu_n,\xi) \right|
} \quad & \\
& \leq
C \left( 
\fint_{B_{\mathsf{M}} }
\left|\overline{L}(\xi) -  \nu(\cu_n,\xi) \right|\,d\xi
\right)^{\frac{\beta}{2+\beta}}
\left( \left\| \overline{L} \right\|_{C^{2,\beta}(\Rd)} + \left\| \nu(\cu_n,\cdot) \right\|_{C^{2,\beta}(\Rd)} \right)^{\frac{2}{2+\beta}}
\\ & 
\leq 
C \left(C 3^{-n\alpha(d-\sigma')} + \O_1\left(C3^{-n\sigma'} \right) \right)^{\frac{\beta}{2+\beta}} , 
\end{align*}
where $\alpha(d,\Lambda) \in (0,1)$ is as in~\eqref{e.subadd.conv}, $\sigma'\in (\sigma,d)$ will be chosen below and depend only on $(\sigma,d)$, and we can justify the derivation of the last line with the aid of~\eqref{e.Osums}. Finally, we use the elementary inequality  
\begin{equation*}
(a+b)^\ep \leq a^\ep + \ep a^{\ep-1} b \quad \forall a,b\in (0,\infty), \ \ep \in (0,1] 
\end{equation*}
to obtain
\begin{equation*}
\sup_{\left| \xi \right| \leq \mathsf{M}} 
\left| D^2\overline{L}(\xi) - D^2_p \nu(\cu_m,\xi) \right|
\leq 
C 3^{-n\alpha\beta(d-\sigma')/(2+\beta)} + \O_1\left(C3^{n\alpha(d-\sigma') -n\sigma'} \right) 
\end{equation*}
Taking~$\sigma' = \frac{d +\sigma}{2}$ and relabeling~$\frac{\alpha \beta}{2(2+\beta)}$ as~$\alpha$, we obtain~\eqref{e.nuconvC2} since $\alpha \in (0,1]$. 
\end{proof}

\section{Quantitative homogenization of the linearized equation}
\label{s.linearization}

This section is devoted to the proof of the first main result, Theorem~\ref{t.linearization}. As mentioned in the introduction, the main difficulty is that the linearized equation around an arbitrary solution does not possess nice statistical properties since the coefficients depend on the solution. The first step in the proof therefore is to approximate the solution and thus the linear coefficients by gluing together local solutions defined on a mesoscopic scale. Theorem~\ref{t.AS.homogenization} ensures that the error resulting from this approximation is sufficiently small, and the resulting equation is locally stationarity and has a finite range of dependence property. We then homogenize it using the techniques from~\cite{AS,AKMbook}. 

\subsection{Setup}
It is convenient to rescale the statement of Theorem~\ref{t.linearization}. We select~$\sigma\in (0,d)$, a reference Lipschitz domain~$U_0\subseteq B_1$, take $r\geq1$, $\delta\in \left(0,\tfrac12\right]$ and set $U:= rU_0$. We also fix $u, \uhom\in W^{1,2+\delta}(U)$ satisfying 
\begin{equation}
\left\{
\begin{aligned}
& - \nabla \cdot \left( D_p L( \nabla u,x) \right) = 0 = -\nabla \cdot \left(D_p \Lhom \left(\nabla \uhom \right) \right) & \mbox{in} & \ U, \\
& u - \uhom \in H^1_0(U),\\
& \left\| \nabla \uhom \right\|_{\underline{L}^{2+\delta}(U)} \leq \mathsf{M}. 
\end{aligned}
\right.
\end{equation}
We then select another function~$f\in W^{1,2+\delta}(U)$ and denote by $w,\whom \in H^1(U)$ the solutions of 
\begin{equation}
\label{e.linearization.wts}
\left\{ 
\begin{aligned}
& -\nabla \cdot \left( D_p^2L\left( \nabla u,x \right) \nabla w \right) = 0 = -\nabla \cdot \left( D_p^2\overline{L} \left( \nabla \uhom \right) \nabla \whom \right) & \mbox{in} & \ U, \\
& w,\whom \in f + H^1_0(U).
\end{aligned}
\right.
\end{equation}
The goal is to prove the following estimate, for an exponent~$\alpha(\delta,U_0,\data)>0$, a constant~$C(\sigma,\delta,\mathsf{M},U_0,\data)<\infty$ and random variable~$\X$ satisfying the bound~\eqref{e.size.X}:
\begin{multline}
\label{e.DPestimates.rescale}
\frac1r \left\| \nabla w - \nabla \whom  \right\|_{\underline{H}^{-1}(U)} 
+ \frac1r\left\| D_p^2L\left( \nabla u,\cdot \right) \nabla w - D_p^2\overline{L} \left( \nabla \uhom \right) \nabla \whom   \right\|_{\underline{H}^{-1}(U)}
\\
\leq
\left\| \nabla f \right\|_{\underline{L}^{2+\delta}(U)} \left( 
C r^{-\alpha}
+\X r^{-\sigma}
\right).  
\end{multline}
This estimate is equivalent to Theorem~\ref{t.linearization} by a rescaling (cf.~\eqref{e.underlinedscalings}).  

\smallskip

We may suppose without loss of generality that 
\begin{equation}
\label{e.rgeqC}
r \geq C(\sigma,\delta,\mathsf{M},U_0,\data)
\end{equation}
for any particular constant~$C(\sigma,\delta,\mathsf{M},U_0,\data)<\infty$ of our choosing.
Indeed, for $r\leq C$, we obtain the estimate~\eqref{e.DPestimates.rescale}, after suitably enlargening the constant on the right side, from the fact that the quantities on the left side are bounded by~$C\left\| \nabla f \right\|_{\underline{L}^2(U)}$ with~$C(U_0,d,\Lambda)<\infty$. 

\smallskip

Throughout the rest of the section, and unless otherwise stated to the contrary,~$C$ denotes a large constant belonging to the interval~$[1,\infty)$ which depends only on the parameters~$(\sigma,\delta,\mathsf{M},U_0,\data)$ which may vary from line to line (or even between different occurrences in the same line). Similarly, unless otherwise indicated,~$\alpha$ and~$\beta$ denote small exponents belonging to $\left(0,\tfrac12\right]$ which depend only on~$(\delta,\data)$ (actually they will depend only on~$(\delta,d,\Lambda,\gamma)$) which may vary in each occurrence. Finally,~$\X$ denotes a random variable satisfying the bound~$\X = \O_1(C)$ which is also allowed to vary from line to line.

\subsection{Definition of mesoscopic scales}
To begin the proof of~\eqref{e.linearization.wts}, we take $n\in\N$ such that $r \in \left( 3^{n-1},3^n\right]$. In addition, we will work with three mesoscopic scales represented by~$k,l,m\in\N$ with
\begin{equation}
\label{e.mesofriends}
k < l < m < n
\end{equation}
Among these,~$l$ and~$m$ will be a very large mesoscopic scales close to the macroscopic scale with $m$ much closer than $l$ ($n-m \ll n-l \ll n$), and~$k$ a very small mesoscopic scale close to the microscopic scale ($k \ll n$). We select~$k$ to be the largest integer and~$l$ and~$m$ the smallest integers satisfying
\begin{equation}
\label{e.mesoscales}
k \leq \smallpara(d-\sigma) n, \quad 
n-l \leq \smallpara(d-\sigma) n, \quad 
\frac12 \leq \frac{n-m}{\smallpara^2 (d-\sigma) n} \leq 1,
\end{equation}
where $\smallpara(\sigma,\delta,d,\Lambda,\gamma),\in \left( 0,\tfrac{1}{100} \right]$ is minimum of various exponents appearing below (each of which depends on the appropriate parameters). We will select $\smallpara$ near the conclusion of the proof. We can ensure that the condition~\eqref{e.mesofriends} is enforced by choosing $C$ large enough in~\eqref{e.rgeqC}. This choice of~$C$ will depend of course on the constant $\smallpara$ in~\eqref{e.mesoscales}.

\smallskip

For each $x\in\Rd$ and $j\in\N$, we define $[x]_j$ to be the closest element of $3^j\Zd$ to $x$ (in the case this closest point is not unique, we make any choice which preserves measurability, lexicographical ordering for instance).
We also define open sets $U^{\circ}_i$ with $i\in\{1,2,3,4\}$ and $V$ such that 
\begin{equation*}
U^{\circ}_4 \subseteq U^{\circ}_3 \subseteq U^{\circ}_2 \subseteq U^{\circ}_1 \subseteq U \subseteq V
\end{equation*}
by
\begin{equation*}
U^{\circ}_i := \left\{ x\in U \,:\, [x]_m+\cu_{m+2i} \subseteq U \right\}
\quad  \mbox{and}  \quad
V:= \bigcup_{x\in U} \left( [x]_m+\cu_{m+2} \right). 
\end{equation*}
By taking the~$C$ in~\eqref{e.rgeqC} sufficiently large, depending on~$U_0$, we can ensure that $U^{\circ}_4$ is nonempty. We have that, for each $i\in\{1,2,3\}$, 
\begin{equation*}
3^{m} \leq \dist\left(U^{\circ}_{i+1} , \partial U^{\circ}_i\right), \
\dist\left(U^{\circ}_1, \partial U \right), \
\dist\left(\partial V, U \right) 
\leq C3^{m}
\end{equation*}
and there exists $C(U_0,d)<\infty$ such that 
\begin{equation}
\label{e.blayer.measure}
\left| V \setminus U^{\circ}_4 \right| \leq C 3^{-(n-m)} \left| U \right|. 
\end{equation}

\subsection{Deterministic regularity estimates}
We next record some deterministic regularity estimates which will be used many times in the forthcoming argument.

\smallskip

By the global Meyers estimates in Lipschitz domains (see~\cite[Appendix B]{AM} for instance), under the assumption that the exponent~$\delta$ is sufficiently small, depending only on~$(U_0,d,\Lambda)$, there exists a constant~$C(U_0,d,\Lambda)<\infty$ such that 
\begin{equation}
\label{e.meyers.u}
\left\| \nabla u \right\|_{\underline{L}^{2+\delta}(U)} 
\leq C \mathsf{M} \leq C
\end{equation}
and
\begin{equation}
\label{e.meyers.wwhom}
\left\| \nabla w \right\|_{\underline{L}^{2+\delta}(U)}
+ 
\left\| \nabla \whom \right\|_{\underline{L}^{2+\delta}(U)}
\leq 
C \left\| \nabla f \right\|_{\underline{L}^{2+\delta}(U)}.
\end{equation}
Without loss of generality (by shrinking~$\delta$, if necessary) we may assume that these bounds hold. 
Next, by subtracting a constant from $u$ and $\uhom$ as well as from $f$, $w$ and $\whom$ and applying the Sobolev extension theorem, we may suppose as well that these functions are globally defined,  belong to $W^{1,2+\delta}(\Rd)$ and satisfy 
\begin{equation}
\label{e.extend.uuhom}
|U|^{-\frac1{2+\delta}} \left( 
 \left\| \nabla u \right\|_{L^{2+\delta}(\Rd)} 
 +
 \left\| \nabla \uhom \right\|_{L^{2+\delta}(\Rd)} 
 \right)
 \leq C \mathsf{M} \leq C
\end{equation}
and
\begin{equation}
\label{e.extend.wwhom}
 \left\| \nabla w \right\|_{L^{2+\delta}(\Rd)} 
 +
 \left\| \nabla \whom \right\|_{L^{2+\delta}(\Rd)} 
 \leq C \left\| \nabla f \right\|_{L^{2+\delta}(\Rd)}
 \leq C \left\| \nabla f \right\|_{L^{2+\delta}(U)}.
\end{equation}

\smallskip

We next recall some pointwise bounds for the solutions~$\uhom$ and~$\whom$ of the homogenized equations. 
By the De Giorgi-Nash estimate, 
there exist an exponent~$\beta(d,\Lambda)>0$ and constant~$C(U_0,d,\Lambda)<\infty$ such that  
\begin{equation}
\label{e.uhom.C1beta}
\left\| \nabla \uhom \right\|_{L^\infty\left(U^{\circ}_1\right)}
+
3^{\beta m} \left[ \nabla \uhom \right]_{C^{0,\beta} (U^{\circ}_1)} 
\leq 
C3^{d(n-m)/2} \mathsf{M}
\leq 
C3^{n d \smallpara^2 (d-\sigma) }. 
\end{equation}
To see this, apply the H\"older estimate for $\nabla \uhom$ in each ball of radius $3^{m+1}$ centered at a point~$x\in U^{\circ}_1$ to obtain, for a constant~$C(U_0,d,\Lambda)<\infty$,  
\begin{align*}
\left\| \nabla \uhom \right\|_{L^\infty\left(B_{3^m}(x)\right)}
+ 3^{\beta m}  \left[ \nabla \uhom \right]_{C^{0,\beta}(B_{3^m}(x))}
&
\leq 
C\left\| \nabla \uhom \right\|_{\underline{L}^2\left(B_{3^{m+1}}(x)\right)} 
\\ & 
\leq C \left( \frac{|U|}{\left|B_{3^{m+1}}(x)\right|} \right)^{\frac12} \left\| \nabla \uhom \right\|_{\underline{L}^2\left(U \right)}
\\ & 
\leq C 3^{d(n-m)/2} \mathsf{M}
\leq C3^{ n d \smallpara^2 (d-\sigma) }. 
\end{align*}
A covering of $U^{\circ}_1$ by such balls yields~\eqref{e.uhom.C1beta}. 

\smallskip

Pointwise bounds for the solution~$\whom$ of the linearized equation follow next from the Schauder estimates. These depend on the H\"older seminorm of the linearized coefficients. Combining Proposition~\ref{p.nu.C2beta} and~\eqref{e.uhom.C1beta} yields~$\beta(d,\Lambda,\gamma)>0$ and~$C(\data)<\infty$ such that 
\begin{equation}
\label{e.whom.C1beta.coeff}
3^{\beta m}
\left[ 
 D^2\overline{L}\left( \nabla \uhom \right) \right]_{C^{0,\beta}(U^{\circ}_1)}
 \leq 
C3^{ n d \smallpara^2 (d-\sigma) }.
\end{equation}
The Schauder estimates yield,
for~$C(\mathsf{M},U_0,\data)<\infty$ and~$\beta(d,\Lambda,\gamma)\in\left(0,\tfrac12\right]$, the existence of~$\alpha(d,\Lambda,\gamma)\in\left(0,\tfrac12\right]$ such that 
\begin{align}
\label{e.whom.C1beta}
\lefteqn{
\left\| \nabla \whom \right\|_{L^\infty\left(U^{\circ}_2 \right)}
+
3^{\beta m} \left[ \nabla \whom \right]_{C^{0,\beta} (U^{\circ}_2)} 
} \qquad & 
\\ \notag &
\leq 
C \left( 1 + \left( 3^{\beta m} \left[  D^2\overline{L}\left( \nabla \uhom \right) \right]_{C^{0,\beta}(U^{\circ}_1)} \right)^{\frac {d}{2\beta}} \right)
\sup_{x\in U^{\circ}_2} \left\| \nabla \whom \right\|_{\underline{L}^2(B_{3^m}(x))}
\\ \notag &
\leq 
C3^{ n  (\smallpara^2 /\alpha) (d-\sigma)}\left\| \nabla f \right\|_{\underline{L}^2(U)}.
\end{align}
Note that Schauder estimates with explicit dependence on the H\"older seminorm of the coefficients can be obtained from the usual statement (which can be found  for instance in~\cite[Theorem 3.13]{HL}) and a straightforward scaling argument. The exponents of~$3$ on the right side of each of the estimates~\eqref{e.uhom.C1beta},~\eqref{e.whom.C1beta.coeff} and~\eqref{e.whom.C1beta} above is of the form~$C\smallpara^2(d-\sigma)$ for a constant $C(d,\Lambda,\gamma)<\infty$. Eventually we will choose the parameter~$\smallpara$ very small so that these factors grow as a very small power $3^n$ (and can thus be absorbed by other factors which are negative powers of $3^n$).

\subsection{Approximation by a locally stationary equation}

The first step in the proof of Theorem~\ref{t.linearization} is to show that~$w$ may be approximated by the solution~$\tilde{w}$ of a linear equation with locally stationary coefficients. In order to construct this equation and obtain an estimate on the difference between~$w$ and~$\tilde{w}$, we need to recall certain quantitative homogenization estimates from~\cite{AS}. 

\smallskip

To give ourselves a little room, we put $\sigma':= \frac12(\sigma+d)$ and $\sigma'' := \frac12(\sigma'+d)$. By~\cite[Theorem 1.1]{AS}, there exist $\alpha_0(d,\Lambda)\in\left(0,\tfrac12\right]$, $C(\sigma,\delta,\mathsf{M},U_0,\data)<\infty$ and a random variable~$\X$ satisfying 
\begin{equation*}
\X \leq \O_1(C)
\end{equation*}
such that 
\begin{equation}
\label{e.uuhom.ee}
\frac1r \left\| u -  \uhom \right\|_{\underline{L}^2(U)}^2
\leq
C r^{-\alpha_0(d-\sigma)} + \X r^{-\sigma''}. 
\end{equation}

\smallskip

We will compare~$u$ to solutions of the Dirichlet problem with affine boundary data in mesoscopic cubes. Recall that we denote by~$\ell_\xi$ the affine function~$\ell_\xi(x):=\xi\cdot x$. For each $\xi \in\Rd$ and Lipschitz domain $U\subseteq \Rd$, we denote by $v(\cdot,U,\xi)$ the minimizer in the definition of~$\nu(U,\xi)$ which is the unique solution of the Dirichlet problem 
\begin{equation}
\label{e.vz}
\left\{
\begin{aligned}
& - \nabla \cdot \left( D_pL(\nabla v(\cdot,U,\xi),x) \right) = 0 & \mbox{in} & \ U, \\
& v(\cdot,U,\xi) = \ell_\xi & \mbox{on} & \ \partial U. 
\end{aligned}
\right.
\end{equation}
We will use the following estimates for the solutions of~\eqref{e.vz}: for every $\mathsf{K}\in [1,\infty)$ and $q \in [1,\infty)$, 
there exist~$\alpha_1(d,\Lambda)>0$ and $C(\mathsf{K},q,d,\Lambda)<\infty$ and a random variable~$\X$ satisfying $\X\leq \O_1(C)$ such that, for every $M,N\in\N$ with $M\leq N$,
\begin{multline}
\label{e.ellztildevz}
\sup_{\xi \in B_{\mathsf{K}3^{Mq}}} 
(1+|\xi|)^{-2} 3^{-d(N-M)} 
\sum_{z \in 3^M\Zd\cap \cu_N}
3^{-2M}
\left\| 
v(\cdot,z+\cu_M,\xi) - \ell_\xi
\right\|_{\underline{L}^2 \left(z+\cu_{M} \right)}^2
\\
\leq C3^{-M\alpha_1(d-\sigma)} + \X 3^{-\sigma''M}
\end{multline}
and
\begin{multline}
\label{e.vztildevz}
\sup_{\xi \in B_{\mathsf{K}3^{Mq}}} 
(1+|\xi|)^{-2} 3^{-d(N-M)} 
\sum_{z\in 3^M\Zd\cap \cu_N}
\left\| 
\nabla v(\cdot,\cu_N,\xi) - \nabla v(\cdot,z+\cu_M,\xi)
\right\|_{\underline{L}^2 \left(z+\cu_M \right)}^2
\\
\leq C3^{-M\alpha_1(d-\sigma)} + \X 3^{-\sigma''N}.
\end{multline}
The first inequality~\eqref{e.ellztildevz} is a consequence of~\cite[Corollary 3.5]{AS}. 
The second inequality~\eqref{e.vztildevz} was essentially proved in~\cite{AS}, but was not stated in exactly this form, so we give a proof of it in Appendix~\ref{ap.ASestimates}. We will apply these inequalities with $N \geq n$ and with $\mathsf{K}$ and $q$ chosen, depending only on $(d,\Lambda)$, so that, by~\eqref{e.uhom.C1beta}, the supremums over $\xi\in B_{\mathsf{K}3^{qN}}$ can be replaced by supremums over $|\xi| \leq \left\| \nabla \uhom\right\|_{L^\infty(U^{\circ})}$. 


\smallskip

To introduce the approximating equation, we define,
for each $\xi\in\Rd$, the linear coefficient field
\begin{equation}
\label{e.axi}
\a_\xi(x):= D^2_pL\left(\nabla v \left(x,z'+\cu_k,\xi \right),x \right), 
\quad z' \in 3^k\Zd, \ 
x \in z'+ \cu_k.
\end{equation}
Observe from the assumption (L2) that $\a_\xi$ is uniformly elliptic,
\begin{equation*}
I_d \leq \a_\xi \leq \Lambda I_d \quad \mbox{in} \ \Rd.
\end{equation*}
It is clear that $\a_\xi$ is $3^k\Zd$--stationary and has a range of dependence of at most $2\left( 1+ \sqrt{d} \cdot 3^k\right)$. In particular, the theory of quantitative homogenization for linear equations with finite range dependence (see~\cite{AKMbook}) applies to the linear equation with coefficients~$\a_\xi$. We denote the homogenized coefficients corresponding to~$\a_\xi$ by $\ahom_\xi$. 

\smallskip

Finally, we define the linear coefficient field~$\a(\cdot)$ by gluing together the $\a_\xi$'s in mesoscopic boxes with the parameter $\xi$ given by a local averaged slope of $\uhom$:
\begin{equation}
\label{e.a}
\a(x) := \a_{\left( \nabla \uhom\right)_{z+\cu_{l}}} (x),\quad z\in 3^{l}\Zd, \ x\in z+\cu_l. 
\end{equation}
This coefficient field~$\a(\cdot)$ is \emph{locally stationary} in the sense that in each mesoscopic cube $z+\cu_l$ with $z\in 3^l\Zd$ it is the restriction of the $3^k\Zd$--stationary field~$\a_\xi$ with parameter $\xi = \left( \nabla \uhom \right)_{z+\cu_{l}}$. 

\smallskip

We next prove that $\a(x)$ is close to $D^2_pL(\nabla u(x),x)$ and deduce therefore, by an argument similar to the proof of Lemma~\ref{l.deterministic.linearization}, that~$w$ is close to the solution~$\tilde{w}\in H^1(U)$ of the Dirichlet problem 
\begin{equation}
\label{e.wtilde}
\left\{ 
\begin{aligned}
& -\nabla \cdot \left( \a \nabla \tilde{w} \right) = 0 & \mbox{in} & \ U, \\
& \tilde{w} \in f + H^1_0(U).
\end{aligned}
\right.
\end{equation}

\begin{lemma}
\label{l.wtowtilde} 
There exist~$\alpha(\delta,\data)\in \left(0,\tfrac12\right]$, $C(\sigma,\delta,\mathsf{M},U_0,\data)<\infty$ and a random variable $\X$ satisfying $\X \leq \O_1(C)$ such  that if $\smallpara \leq \alpha$, then 
\begin{equation}
\label{e.wtowtilde}
\left\| \nabla w - \nabla \tilde{w} \right\|_{\underline{L}^2(U)}
\leq 
C \left\| \nabla f \right\|_{\underline{L}^{2+\delta}(U)}
\left(  3^{-n \alpha \smallpara^2 (d-\sigma)^2}+ \X 3^{-n\sigma} \right).
\end{equation}
\end{lemma}
\begin{proof}
We have that $\a(x) = D^2_pL(F(x),x)$, where we define the vector field $F$ by 
\begin{equation*}
F(x) :=
\nabla v \left(x,z'+\cu_k, \left( \nabla \uhom \right)_{z+\cu_l} \right),
\quad z\in 3^l\Zd, \ z' \in 3^k\Zd, \ 
x \in z'+ \cu_k.
\end{equation*}
We also define a vector field~$H$ similar to $F$ but using a larger mesoscopic scale:
\begin{equation*}
H(x):= \nabla v\left(\cdot,z+\cu_{l+1},\left( \nabla \uhom\right)_{z+\cu_l}\right) , \quad z\in 3^l\Zd, \ x\in z+\cu_l. 
\end{equation*}
Throughout the argument $\alpha$, $C$ and $\X$ will be as in the statement of the lemma, but may change in each occurrence. 

\smallskip

\emph{Step 1.} We show that, for fixed $\sigma \in (0,d)$, there exist constants~$\alpha(\delta,\data)>0$ and $C(\sigma,\data)<\infty$, and a random variable $\X = \O_1(C)$   such that if $\smallpara \leq \alpha$, then
\begin{equation}
\label{e.nablautoH}
\left\| \nabla u - H \right\|_{\underline{L}^2(U^{\circ}_1)}^2 
\leq
 C 3^{-n\alpha(d-\sigma)} + \X 3^{-n\sigma'}.
\end{equation}
By the Caccioppoli and triangle inequalities, for every $z\in 3^l\Zd \cap U^{\circ}_1$,
\begin{align*}
\lefteqn{
\left\| \nabla u - \nabla v\left(\cdot,z+\cu_{l+1},\left( \nabla \uhom\right)_{z+\cu_l}\right) \right\|_{\underline{L}^2(z+\cu_l)}^2
} \ \ & 
\\ &
\leq 
C3^{-2l} \left\|  
u - v\left(\cdot,z+\cu_{l+1},\left( \nabla \uhom\right)_{z+\cu_l}\right) 
\right\|_{\underline{L}^2(z+\cu_{l+1})}^2
\\ &
\leq 
C3^{-2l} 
\left( 
\left\|  u -  \uhom 
\right\|_{\underline{L}^2(z+\cu_{l+1})}^2 
+
\left\|   
\uhom - v\left(\cdot,z+\cu_{l+1},\left( \nabla \uhom\right)_{z+\cu_l}\right)  
\right\|_{\underline{L}^2(z+\cu_{l+1})}^2 
\right).
\end{align*}
Set $\sigma'' := \frac{d+\sigma'}{2}$. 
By~\eqref{e.uuhom.ee} and~\eqref{e.mesoscales}, we find, for $\smallpara$ sufficiently small,  
\begin{align*}
\frac{\left| \cu_{l} \right|}{\left| U^{\circ}_1\right|} 
\sum_{z\in 3^l\Zd \cap U^{\circ}_1}
3^{-2l}\left\|  
u -  \uhom 
\right\|_{\underline{L}^2(z+\cu_{l+1})}^2
&
\leq
C 3^{-2l} \left\| u -  \uhom \right\|_{\underline{L}^2(U)}^2
\\ &
\leq C 3^{2(n-l)} \left( 3^{-\alpha_0(d-\sigma) n} + \X 3^{-n\sigma''} \right). 
\\ &
\leq C \left( 3^{-\frac{\alpha_0}{4}(d-\sigma) n} + \X 3^{-n\sigma'} \right).
\end{align*}
Next, for each $z\in 3^l\Zd \cap U^{\circ}_1$, 
define the affine function
\begin{equation*}
\ell_z(x):= \uhom(z) + \left( \nabla \uhom\right)_{z+\cu_l} \cdot (x-z)
\end{equation*}
and compute, with the aid of~\eqref{e.uhom.C1beta},
\begin{equation*}
\left\| \ell_z - \uhom \right\|_{\underline{L}^2(z+\cu_{l+1})} 
\leq 
C3^{l(1+\beta)}\left[ \nabla \uhom \right]_{C^{0,\beta}(U^{\circ}_1)}
\leq 
C3^{l-\beta(m-l)+ d(n-m)/2}
\mathsf{M}.
\end{equation*}
Thus by the triangle inequality, 
\begin{align*}
\lefteqn{
\left\|  
\uhom 
- v\left(\cdot,z+\cu_{l+1},\left( \nabla \uhom\right)_{z+\cu_l}\right) 
\right\|_{\underline{L}^2(z+\cu_{l+1})}^2 
} \quad & 
\\ & 
\leq
2\left\|   v\left(\cdot,z+\cu_{l+1},\left( \nabla \uhom\right)_{z+\cu_l}\right) - \ell_z \right\|_{\underline{L}^2(z+\cu_{l+1})}^2 
+
2\left\| \ell_z - \uhom \right\|_{\underline{L}^2(z+\cu_{l+1})}^2 
\\ & 
\leq 
2\left\|   v\left(\cdot,z+\cu_{l+1},\left( \nabla \uhom\right)_{z+\cu_l}\right) - \ell_z \right\|_{\underline{L}^2(z+\cu_{l+1})}^2 
+
C3^{2l-2\beta(m-l)+ d(n-m)}.
\end{align*}
Summing over $z\in 3^l\Zd \cap U^{\circ}_1$, and applying~\eqref{e.ellztildevz} and~\eqref{e.mesoscales}, yields
\begin{multline*}
\frac{1}{\left|3^l\Zd \cap U^{\circ}_1\right|} 
\sum_{z\in 3^l\Zd \cap U^{\circ}_1}
3^{-2l}\left\|  
\uhom 
- v\left(\cdot,z+\cu_{l+1},\left( \nabla \uhom\right)_{z+\cu_l}\right) 
\right\|_{\underline{L}^2(z+\cu_{l+1})}^2
\\
\leq C \left( 3^{-\alpha_1(d-\sigma'') l} + 3^{-2\beta(m-l)+ d(n-m)} \right) + \X 3^{- l \sigma'' } 
\\
\leq C \left( 3^{-n\alpha_1(1-\smallpara)(d-\sigma'') } + 3^{- n \smallpara \left(2\beta - 2d\smallpara\right)} \right) + \X 3^{- n \sigma'' (1-2\smallpara)}.
\end{multline*}
Combining the above inequalities and taking $\smallpara$ sufficiently small yields~\eqref{e.nablautoH}.  

\smallskip

\emph{Step 2.} We show that there exists $\alpha(\delta,d,\Lambda)>0$ such that  
\begin{equation}
\label{e.nablautoH.layer}
\frac{1}{|U|} \left\| \nabla u - H \right\|_{{L}^2(U\setminus U^{\circ}_1)}^2
\leq 
C 3^{-n\alpha\smallpara^2(d-\sigma)}. 
\end{equation}
By~\eqref{e.meyers.u} and~\eqref{e.blayer.measure}, 
\begin{equation*}
\frac{1}{|U|} \left\| \nabla u \right\|_{{L}^2(U\setminus U^{\circ}_1)}^2
\leq 
C \left( 
\frac{\left| U \setminus U^{\circ}_1 \right|}{\left| U \right|} 
\right)^{\frac{\delta}{2+\delta}}
\left\| \nabla u \right\|_{\underline{L}^{2+\delta}(U)}^2
\leq
C\mathsf{M}^2 3^{-\delta (n-m)/(2+\delta)}.
\end{equation*}
Similarly, we have 
\begin{equation*}
\frac{1}{|U|} \left\| \nabla \uhom \right\|_{{L}^2(V\setminus U^{\circ}_1)}^2
\leq C\mathsf{M}^2 3^{-\delta (n-m)/(2+\delta)}.
\end{equation*}
Likewise,
\begin{align*}
\frac{1}{|U|} \left\| H \right\|_{L^2(U\setminus U^{\circ}_1)}^2 
&
\leq 
\frac{|\cu_{l}|}{|U|} \sum_{z\in 3^l\Zd \cap U \setminus U^{\circ}_1}
\left\| \nabla v\left(\cdot,z+\cu_{l+1},\left( \nabla \uhom\right)_{z+\cu_l} \right) 
\right\|_{\underline{L}^2(z+\cu_l)}^2 
\\ & 
\leq C\frac{|\cu_{l}|}{|U|}  \sum_{z\in 3^l\Zd \cap U \setminus U^{\circ}_1}
\left\| \nabla v\left(\cdot,z+\cu_{l+1},\left( \nabla \uhom\right)_{z+\cu_l} \right) 
\right\|_{\underline{L}^2(z+\cu_{l+1})}^2
\\ & 
\leq
C \frac{|\cu_{l}|}{|U|}  \sum_{z\in 3^l\Zd \cap U \setminus U^{\circ}_1}
\left( 1 + \left| \left( \nabla \uhom\right)_{z+\cu_l} \right|^2 \right)
\\ & 
\leq
C\frac{1}{|U|} \left(| V\setminus U^{\circ}_1| +   \left\| \nabla \uhom \right\|_{L^2(V\setminus U^{\circ}_1)}^2 \right).
\end{align*}
Putting these together gives~\eqref{e.nablautoH.layer}.

\smallskip

\emph{Step 3.} We argue that, for fixed $\sigma \in (0,d)$, there exist constants~$\alpha(\delta,d,\Lambda,\gamma)>0$ and $C<\infty$, and a random variable $\X = \O_1(C)$  such that if $\smallpara \leq \alpha$, then 
\begin{equation}
\label{e.nablautoF}
\left\| \nabla u - F \right\|_{\underline{L}^2(U_1^\circ)}^2
\leq 
 C 3^{-n\alpha\smallpara^2(d-\sigma)} + \X 3^{-n\sigma'}.
\end{equation}
By the previous two steps, it suffices to show that 
\begin{equation*}
\left\| F - H \right\|_{\underline{L}^2(U^1_\circ)}^2
\leq 
C 3^{-k \alpha(d-\sigma) } + \X 3^{-n\sigma'} \leq  C 3^{- n \alpha \smallpara (d-\sigma)^2} + \X 3^{-n\sigma'}.
\end{equation*}
This is an immediate consequence of~\eqref{e.vztildevz} and~\eqref{e.uhom.C1beta}. 

\smallskip

\emph{Step 4.} We show next that, there exist~$\alpha(\delta,d,\Lambda,\gamma)>0$, $C<\infty$ and $\X$ such that if $\smallpara \leq \alpha$, then, for every $q\in [2,\infty)$, 
\begin{equation}
\label{e.Ftonablau.D2p}
\left\| D^2_pL(\nabla u,\cdot) - \a \right\|_{\underline{L}^q(U)}
\leq 
C3^{- n \alpha \smallpara^2 (d-\sigma)^2/q} + \X 3^{-\sigma n}. 
\end{equation}
For $q=2$, we have, by~\eqref{e.nablautoF}, 
\begin{align*}
\left\| D^2_pL(\nabla u,\cdot) - D^2_pL(F,\cdot) \right\|_{\underline{L}^2(U_1^\circ)}
& 
\leq 
\sup_{x\in\Rd} \left[ D^2_pL(\cdot,x) \right]_{C^{0,\gamma}(\Rd)}
\left\| \nabla u- F \right\|_{\underline{L}^2(U_1^\circ)}^\gamma
\\ & 
\leq 
\mathsf{M}_0 
\left(   C 3^{-n\alpha\smallpara^2(d-\sigma)^2} + \X 3^{-n\sigma'} \right)^\gamma,
\end{align*}
and, since $D^2_pL(p,\cdot)$ is bounded, we have, in view of~\eqref{e.blayer.measure},
\begin{equation*}
|U|^{-\frac12} \left\| D^2_pL(\nabla u,\cdot) - D^2_pL(F,\cdot) \right\|_{{L}^2(U\setminus U_1^\circ)}
\leq 
C \left| U \right|^{-\frac12} \left| U\setminus U_1^\circ \right|^{\frac12}
\leq 
C3^{-n\alpha\smallpara^2(d-\sigma)}.
\end{equation*}
Putting these together, we get that 
\begin{align*}
\left\| D^2_pL(\nabla u,\cdot) - D^2_pL(F,\cdot) \right\|_{\underline{L}^2(U)}
& 
\leq 
\mathsf{M}_0 
\left(   C 3^{-n\alpha\smallpara^2(d-\sigma)^2} + \X 3^{-n\sigma'} \right)^\gamma.
\end{align*}
Using again that~$D^2_pL(p,\cdot)$ is bounded, by interpolation we deduce
\begin{equation*}
\left\| D^2_pL(\nabla u,\cdot) - D^2_pL(F,\cdot) \right\|_{\underline{L}^q(U)}
\leq C \left(  3^{-n\alpha\smallpara^2(d-\sigma)^2} + \X 3^{-n\sigma'}  \right)^{\gamma/q}.
\end{equation*}
Using the inequality
\begin{equation} \label{e.epinterp}
\forall \ep\in (0,1], \ a,b\in (0,\infty),\quad
(a+b)^\ep \leq a^\ep + \ep a^{\ep-1} b, 
\end{equation}
we get 
\begin{equation*}
\left\| D^2_pL(\nabla u,\cdot) - D^2_pL(F,\cdot) \right\|_{\underline{L}^q(U)}
\leq C3^{- n \alpha \smallpara^2(d-\sigma)^2 \gamma/q} + \X 3^{-n \sigma }.
\end{equation*}
Taking $\smallpara$ sufficiently small, and recalling that $\a= D^2_pL(F,\cdot)$ gives the claim~\eqref{e.Ftonablau.D2p}.

\smallskip

\emph{Step 5.} The conclusion. By subtracting the equations for~$w$ and~$\tilde{w}$, we find that the difference~$w-\tilde{w}$ satisfies
\begin{equation*}
-\nabla \cdot \left( \a \nabla \left(w-\tilde{w} \right) \right) 
= 
\nabla \cdot \left(  \left( D^2_pL(\nabla u,x) - \a \right) \nabla w \right) 
\quad \mbox{in} \ U. 
\end{equation*}
Testing this equation with $w-\tilde{w}$, which we notice belongs to~$H^1_0(U)$, we obtain 
\begin{align*}
\left\| \nabla \left(w-\tilde{w} \right) \right\|_{\underline{L}^2(U)} 
\leq 
C \left\| \left( D^2_pL(\nabla u,x) - \a \right) \nabla w \right\|_{\underline{L}^2(U)}. 
\end{align*}
Finally, by the H\"older inequality, the Meyers estimate and~\eqref{e.Ftonablau.D2p} with $q=\frac{4+2\delta}{\delta}$,
\begin{align}
\label{e.switchflux}
\left\| \left( D^2_pL(\nabla u,x) - \a \right) \nabla w \right\|_{\underline{L}^2(U)}
&
\leq 
\left\| D^2_pL(\nabla u,x) - \a \right\|_{\underline{L}^{\frac{4+2\delta}{\delta}}(U)} 
\left\| \nabla w \right\|_{\underline{L}^{2+\delta}(U)}
\\ \notag & 
\leq 
C \left\| D^2_pL(\nabla u,x) - \a \right\|_{\underline{L}^{\frac{4+2\delta}{\delta}}(U)}
\left\| \nabla f \right\|_{\underline{L}^{2+\delta}(U)}
\\ \notag & 
\leq 
C\left\| \nabla f \right\|_{\underline{L}^{2+\delta}(U)} 
\left( 3^{- n \alpha \smallpara^2 (d-\sigma)^2} + \X 3^{- n \sigma} \right).
\end{align}
This completes the proof of the lemma. 
\end{proof}

\subsection{Homogenization estimates for the locally stationary equation}

In view of Lemma~\ref{l.wtowtilde}, we are motivated to obtain homogenization estimates for the locally stationary problem~\eqref{e.wtilde}. This is accomplished in the next subsection, and here we prepare for the analysis by recording some  estimates for the stationary fields~$\a_\xi(\cdot)$. 

\smallskip

We next check that the homogenized coefficients~$\ahom_\xi$ corresponding to~$\a_\xi$ are close to what we expect, namely~$D^2_p\Lhom\left(  \xi \right)$. This is a crucial point in the proof of Theorem~\ref{t.linearization}, for it is here that we really see that homogenization and linearization must commute. 

\smallskip

For each $\eta\in\Rd$ (recall that $\ell_\eta(x):=\eta\cdot x$), we have that
\begin{equation}
\label{e.ahomzchar}
\frac12 \eta \cdot \ahom_\xi \eta 
= 
\lim_{j\to \infty} 
\E \left[ \inf_{w \in \ell_\eta + H^1_0(\cu_j)}  \fint_{\cu_j} \frac 12 \nabla w \cdot \a_\xi \nabla w  \right].
\end{equation}
Let us denote by~$w_\xi(\cdot,\cu_j,\eta)$ the minimizer of the optimization problem inside the expectation on the right side of the previous display. 

\begin{lemma}
\label{l.localahom}
Let~$\mathsf{K},Q\in [1,\infty)$.
There exist an exponent~$\alpha(\data)\in\left(0,\tfrac12\right]$ and constant~$C(Q,\mathsf{K},\data)<\infty$ such that
\begin{equation}
\label{e.localahom}
\sup_{\left|\xi\right| \leq \mathsf{K} 3^{qk}} 
\frac{\left| 
\ahom_\xi - D^2_p\Lhom \left( \xi \right) 
\right| }{1+|\xi|}
\leq 
C 3^{-k\alpha}. 
\end{equation}
\end{lemma}
\begin{proof}
Fix $\left|\xi\right| \leq \mathsf{K} 3^{Qk}$.

\smallskip

\emph{Step 1.} For each $j\in\N$ with $j \geq k$, we approximate~$\a_\xi$ in $\cu_j$ by the coefficients
\begin{equation*}
\tilde{\a}_{\xi,j}(x):= D^2_pL\left( \nabla v \left( x, \cu_j, \xi \right),x \right).
\end{equation*}
The claim is that, for $q \in [2,\infty)$, 
\begin{equation}
\label{e.axitotildeaxi}
\left\| \a_\xi - \tilde{\a}_{\xi,j} \right\|_{\underline{L}^q(\cu_j)} 
\leq 
C(1+|\xi|) \left( 3^{-k\alpha_1} + \X 3^{-n} \right)^{\frac {2\gamma}{q}} .
\end{equation}
To prove~\eqref{e.axitotildeaxi}, we first observe that, for every $z\in 3^k\Zd\cap \cu_j$ and $x\in z+\cu_k$,
\begin{align*}
\left| \a_\xi(x) - \tilde{\a}_{\xi,j}(x) \right|
\leq 
\left[ D^2L \right]_{C^{0,\gamma}(\Rd)}
\left| \nabla v(x,z+\cu_k,\xi) - \nabla v(x,\cu_j,\xi) \right|^\gamma,
\end{align*}
and therefore, by~\eqref{e.vztildevz}, 
\begin{align*}
\left\| \a_\xi - \tilde{\a}_{\xi,j} \right\|_{\underline{L}^2(\cu_j)} 
&
\leq 
C 3^{-d(j-k)} \sum_{z\in 3^k\cap \cu_j} 
\left\| \nabla v(\cdot,z+\cu_k,\xi) - \nabla v(\cdot,\cu_j,\xi) \right\|_{\underline{L}^2(z+\cu_k)}^\gamma
\\ & 
\leq \left( 1+ |\xi| \right) \left( C3^{-k\alpha_1} + \X 3^{-k} \right)^{\frac\gamma2}.
\end{align*}
Using also that 
\begin{equation*}
\left\| \a_\xi - \tilde{\a}_{\xi,j} \right\|_{L^\infty(\cu_j)} 
\leq 
\left\| \a_\xi  \right\|_{L^\infty(\cu_j)} 
+
\left\|  \tilde{\a}_{\xi,j} \right\|_{L^\infty(\cu_j)} 
\leq C, 
\end{equation*}
we find that, for every $q\in [2,\infty)$, 
\begin{align*}
\left\| \a_\xi - \tilde{\a}_{\xi,j} \right\|_{\underline{L}^q(\cu_j)}
&
\leq
 \left\| \a_\xi - \tilde{\a}_{\xi,j} \right\|_{L^\infty(\cu_j)}^{1-\frac 2q}  \left\| \a_\xi - \tilde{\a}_{\xi,j} \right\|_{\underline{L}^2(\cu_j)}^{\frac 2q} 
\leq 
(1+|\xi|)
\left( C3^{-k\alpha_1 } + \X 3^{- k } \right)^{\frac {2\gamma}{q}} .
\end{align*}
This completes the proof of~\eqref{e.axitotildeaxi}. 

\smallskip

\emph{Step 2.} For each $j\in \N$ with $j\geq k$, let us denote by $\tilde{w}_{\xi,j}(\cdot,\eta)$ the minimizer of the variational problem 
\begin{equation*}
\inf_{w\in \ell_\eta+H^1_0(\cu_j)} 
\fint_{\cu_j} \frac 12 \nabla w \cdot \tilde{\a}_{\xi,j} \nabla w.
\end{equation*}
In this step we show that, for some $\alpha(d,\Lambda)>0$,
\begin{equation}
\label{e.VtoVtildej}
\E \left[ 
\left\| \nabla w_\xi(\cdot,\cu_j,\eta) - \nabla \tilde{w}_{\xi,j}(\cdot,\eta) \right\|_{\underline{L}^2(\cu_j)}^2
\right]
\leq 
C \left|\eta\right|^2 3^{-k\alpha}.
\end{equation}
Fix $j\in\N$ and $\eta \in B_1$. The difference $W_\xi := w_\xi(\cdot,\cu_j,\eta)  -  \tilde{w}_{\xi,j}(\cdot,\eta)$ belongs to $H^1_0(\cu_j)$ and satisfies the equation
\begin{equation*}
-\nabla \cdot \a_\xi \nabla W_\xi = \nabla \cdot \left( \a_\xi -  \a_{\xi,j}\right) \nabla \tilde{w}_{\xi,j}
\quad \mbox{in} \ \cu_j. 
\end{equation*}
Testing this equation with $W_\xi$ yields
\begin{align*}
\left\| \nabla W_\xi \right\|_{\underline{L}^2(\cu_j)}^2
& 
\leq C \fint_{\cu_j} \nabla W_\xi\cdot \a_\xi \nabla W_{\xi} 
=  C \fint_{\cu_j} \nabla W_\xi\cdot \left( \a_{\xi,j} - \a_\xi  \right) \nabla \tilde{w}_{\xi,j}.
\end{align*}
From this, the H\"older inequality, the Meyers estimate and~\eqref{e.axitotildeaxi}, we obtain
\begin{align*}
\left\| \nabla W_\xi \right\|_{\underline{L}^2(\cu_j)}
& 
\leq C \left\|  \left( \a_{\xi,j} - \a_\xi  \right) \nabla \tilde{w}_{\xi,j} \right\|_{\underline{L}^2(\cu_j)}
\\ & 
\leq C \left\|  \left( \a_{\xi,j} - \a_\xi  \right)\right\|_{\underline{L}^{\frac{4+2\delta}{\delta}}(\cu_j)}
\left\| \nabla \tilde{w}_{\xi,j} \right\|_{\underline{L}^{2+\delta}(\cu_j)}
\\ & 
\leq   C\left( 3^{-k\alpha_1 } + \X 3^{- k} \right)^{\frac {2\delta \gamma}{2+\delta}}   \left| \eta \right|.
\end{align*}
Squaring and taking expectations yields~\eqref{e.VtoVtildej} via~\eqref{e.epinterp}. 

\smallskip 

\emph{Step 3.} The conclusion. 
As was shown in the proof of Proposition~\ref{p.nu.C2beta}, we have  
\begin{equation*}
\tilde{w}_{\xi,j}(\cdot,\cu_j,\eta) = \eta \cdot D_\xi v \left( \cdot,\cu_j,\xi \right)
\end{equation*}
and, moreover, by~\eqref{e.D2nuUiden} and a testing of the equation for $\tilde{w}_{\xi,j}$ by $\tilde{w}_{\xi,j} - \ell_\eta\in H^1_0(\cu_j)$,   
\begin{equation*}
\fint_{\cu_j} \frac 12 \nabla \tilde{w}_{\xi,j}(\cdot,\cu_j,\eta) \cdot \tilde{\a}_{\xi,j} \nabla \tilde{w}_{\xi,j}(\cdot,\cu_j,\eta)
= \frac12 \eta\cdot D^2_p \nu\left(\cu_j , \xi \right) \eta. 
\end{equation*}
By Lemma~\ref{l.nuconvtoLinC2}, we have that 
\begin{equation*}
\E \left[\left| D^2\overline{L}(\xi) - D^2_p \nu(\cu_j,\xi) \right|  \right] \leq C 3^{-j\alpha} (1+\left| \xi \right|^2) \longrightarrow 0 \quad \mbox{as} \ j\to \infty. 
\end{equation*}
Combining~\eqref{e.ahomzchar} and~\eqref{e.VtoVtildej}, and using as well as the global Meyers estimate which gives, for some $\delta(\data)>0$, the deterministic bound 
\begin{align*}
&
\left\| \nabla w_\xi (\cdot,\cu_j,\eta) \right\|_{\underline{L}^{2+\delta}(\cu_j)} 
+
\left\| \nabla \tilde{w}_{\xi,j} (\cdot,\cu_j,\eta) \right\|_{\underline{L}^{2+\delta}(\cu_j)} 
\\ & \qquad 
\leq 
C\left\| \nabla w_\xi (\cdot,\cu_j,\eta) \right\|_{\underline{L}^{2}(\cu_j)} 
+
C\left\| \nabla \tilde{w}_{\xi,j} (\cdot,\cu_j,\eta) \right\|_{\underline{L}^{2}(\cu_j)}  \leq C|\eta|,
\end{align*}
we deduce that 
\begin{equation*}
\limsup_{j\to \infty}
\left| \frac12\eta\cdot \ahom_\xi\eta - \E \left[ \fint_{\cu_j} \frac 12 \nabla \tilde{w}_{\xi,j}(\cdot,\cu_j,\eta) \cdot \tilde{\a}_{\xi,j} \nabla \tilde{w}_{\xi,j}(\cdot,\cu_j,\eta)  \right] \right|
\leq C \left| \eta \right|^2 3^{-k\alpha}. 
\end{equation*}
Combining the previous displays now yields the lemma. 
\end{proof}

We next state some estimates for the first-order correctors for the (globally stationary) coefficients~$\a_\xi$. For each $z\in 3^{l-1}\Zd$ and $e\in \overline{B}_1$, we take~$\phi_{e,z}\in H^1(z+\cu_l)$ to be the solution of 
\begin{equation}
\label{e.phiez}
\left\{
\begin{aligned}
& -\nabla \cdot \a_{\nabla \uhom(z)} \left(e+\nabla \phi_{e,z} \right) = 0 
& \mbox{in} & \ z + \cu_l, \\
& \phi_{e,z} = 0 
& \mbox{on} & \ \partial (z + \cu_l). 
\end{aligned}
\right.
\end{equation}
As previously mentioned, the coefficient field $\a_\xi$ is $3^k\Zd$--stationary and has range of dependence at most $C3^k$. Therefore, after rescaling the equation by dilation factor of $C3^k$ and applying~\cite[Theorem 2.17]{AKMbook}, we obtain the following estimate: there exist $\alpha(d,\Lambda)>0$ and $C(\sigma,d,\Lambda) < \infty$ and, for every $z\in 3^{l-1}\Zd$ and $e\in \overline{B}_1$, a random variable $\X_z \leq \O_1(C)$ satisfying 
\begin{multline}
\label{e.ests.phiez}
3^{-l}  \left\| \nabla \phi_{e,z}  \right\|_{\Hminusul(z+\cu_l)} 
+
3^{-l}  \left\| 
\a_{\nabla \uhom(z)} \left( e+ \nabla \phi_{e,z} \right) 
- \ahom_{\nabla \uhom(z)} e
\right\|_{\Hminusul(z+\cu_l)} 
\\
\leq 
C3^{-\beta (d-\sigma'') (l-k)} + \X_z 3^{-\sigma''(l-k)} \leq C3^{-n \alpha (d-\sigma) } + \X_z 3^{-n \sigma'}. 
\end{multline}
Observe that here we also used the following consequence of~\eqref{e.mesoscales}:
\begin{equation*} 
l- k \geq n - n (2\smallpara  + \smallpara^2 ) (d-\sigma) 
\end{equation*}
These also imply that 
\begin{equation}
\label{e.phiez.alsoL2}
3^{-l} \left\| \phi_{e,z}  \right\|_{\underline{L}^{2}(z+\cu_l)}
\leq 
C3^{-n \alpha (d-\sigma) } + \X_z 3^{-n \sigma'}. 
\end{equation}
By testing~\eqref{e.phiez} with $\phi_{e,z}$ itself, we also have the deterministic estimate
\begin{equation}
\label{e.phiez.unifbound}
\left\| \nabla \phi_{e,z} \right\|_{\underline{L}^{2}(z+\cu_l)} 
\leq 
C. 
\end{equation}
The Meyers estimate applied to $x\mapsto e\cdot x + \phi_{e,z}(x)$ then yields
\begin{equation}
\label{e.phiez.meyers}
\left\| \nabla \phi_{e,z} \right\|_{\underline{L}^{2+\delta}(z+\cu_l)} 
\leq 
C 
\left\| \nabla \phi_{e,z} \right\|_{\underline{L}^{2}(z+\cu_l)} 
\leq 
C. 
\end{equation}

\smallskip

We also need to record some estimates for the dependence in~$\xi$ of $\a_\xi$. The claim is that there exists $\beta(d,\Lambda,\gamma)>0$ such that, for every $q\in [2,\infty)$, $\xi,\xi'\in\Rd$ and $z\in 3^k\Zd$,  
\begin{equation}
\label{e.axicont}
\left\| \a_\xi - \a_{\xi'} \right\|_{\underline{L}^q(z+\cu_k)}
\leq 
C \left| \xi - \xi' \right|^{\beta/q}. 
\end{equation}
This estimate then implies that, for some $\beta(d,\Lambda,\gamma)>0$,
\begin{equation}
\label{e.ahomxicont}
\left| \ahom_{\xi} - \ahom_{\xi'} \right| 
\leq 
C \left| \xi - \xi' \right|^{\beta}.
\end{equation}
To prove~\eqref{e.axicont}, we first consider the case $q=2$ and, recalling the definition~\eqref{e.axi} of $\a_\xi$ and using assumption (L1), we find 
\begin{align*}
\lefteqn{
\left\| \a_\xi - \a_{\xi'} \right\|_{\underline{L}^2(z+\cu_k)}
} \quad & 
\\ & 
\leq 
\left\| D^2_pL\left( \nabla v(\cdot,z+\cu_k,\xi) ,\cdot \right) 
- 
D^2_pL\left( \nabla v(\cdot,z+\cu_k,\xi') ,\cdot \right)
\right\|_{\underline{L}^2(z+\cu_k)}
\\ & 
\leq C\sup_{x\in\Rd} \left[ D^2_pL(\cdot,x) \right]_{C^{0,\beta}(\Rd)} 
\left\| 
\nabla v(\cdot,z+\cu_k,\xi) 
- \nabla v(\cdot,z+\cu_k,\xi') 
\right\|_{\underline{L}^2(z+\cu_k)}^\beta
\\ & 
\leq C \left| \xi - \xi' \right|^\beta. 
\end{align*}
We then obtain~\eqref{e.axicont} for general $q\in [2,\infty)$ from the case $q=2$ and the fact that, by assumption~(L2), 
\begin{equation*}
\left\| \a_\xi - \a_{\xi'} \right\|_{L^\infty(z+\cu_k)}
\leq
\left\| \a_\xi \right\|_{L^\infty(z+\cu_k)}
+
\left\| \a_{\xi'} \right\|_{L^\infty(z+\cu_k)}
\leq C\Lambda \leq C. 
\end{equation*}

\subsection{Homogenization with locally stationary coefficients}

By Lemma~\ref{l.wtowtilde} and the triangle inequality, to prove~\eqref{e.DPestimates.rescale} it suffices to prove the estimate with~$\tilde{w}$ in place of~$w$. The advantage is that~$\tilde{w}$ satisfies an equation with locally stationary coefficients which, by Lemma~\ref{l.localahom}, have local homogenized coefficients that are close to~$D^2\overline{L}(\nabla \uhom)$. To complete the proof, we therefore need to establish a quantitative homogenization result for linear equations with coefficients which are locally stationary. This is accomplished by a fairly straightforward (although technical) adaptation of the argument for (globally) stationary coefficients which can be found, for instance, in the proof of~\cite[Theorem 1.17]{AKMbook}. 

\smallskip

The result is presented in the following proposition. Recall that the coefficient field~$\a(x)$ is defined above in~\eqref{e.a} and~$\tilde{w}$ is the solution of~\eqref{e.wtilde}.

\begin{proposition}
\label{p.homog.locallystationary}
There exist~$\alpha(\delta,\data)\in \left(0,\tfrac12\right]$ and $C(\sigma,\delta,\mathsf{M},U_0,\data)<\infty$  such that 
\begin{multline}
\label{e.homog.LS}
\frac1r 
\left\| \nabla \tilde{w} - \nabla \whom \right\|_{\underline{H}^{-1}(U)} 
+
\frac1r\left\| \a \nabla \tilde{w} - D^2\overline{L}(\nabla \uhom)  \nabla \whom \right\|_{\underline{H}^{-1}(U)} 
\\
\leq 
C \left\| \nabla f \right\|_{\underline{L}^{2+\delta}(U)} 
\left(
 3^{-n\alpha \smallpara^2 (d-\sigma)}  + \X 3^{-n\sigma'}\right).
\end{multline}
\end{proposition}

The idea of the proof of Proposition~\ref{p.homog.locallystationary} is to compare the solution~$\tilde{w}$ of~\eqref{e.wtilde} to a two-scale expansion around~$\whom$. 
We construct, for each~$e\in\Rd$, a function~$\phi_e$ which will play the role of the first-order corrector (with slope~$e$) in the two-scale expansion. Since the linear equation for $\tilde{w}$ has a locally stationary structure, we build~$\phi_e$ by gluing (finite-volume) correctors $\phi_{e,z}$ for the stationary problems on the mesoscopic scale, defined in~\eqref{e.phiez}. We fix a smooth function $\chi$ such that its $3^{l-1}\Zd$ translates form a partition of unity. Precisely, we require
\begin{equation}
\label{e.chi}
0 \leq \chi \leq 1, \quad
\chi \equiv 0 \ \mbox{in} \ \Rd \setminus \cu_l,  \quad
\left\| \nabla \chi \right\|_{L^\infty(\Rd)} \leq C 3^{-l}, \quad
\sum_{z\in 3^{l-1}\Zd} \chi(\cdot-z) \equiv 1. 
\end{equation}
We can construct such a function~$\chi$ by mollifying the indicator function of~$\cu_{l-1}$, for instance. We then define 
\begin{equation*}
\phi_e(x) 
= 
\sum_{z\in 3^{l-1}\Zd} \chi(x-z) \phi_{e,z}(x). 
\end{equation*}
Finally, we define the competitor to~$\tilde{w}$ by 
\begin{equation}
\label{e.T}
T := \left( 1-\zeta \right) \whom  + \zeta \left(  \whom \ast \psi \right) + \zeta  \sum_{j=1}^d \left( \partial_{x_j} \whom \ast \psi \right) \phi_{e_j},
\end{equation}
where $\zeta\in C^\infty(\Rd)$ is a cutoff function chosen to satisfy, for a constant $C(U_0,d)<\infty$, 
\begin{equation}
\label{e.cutoffzeta}
0\leq \zeta \leq 1, \quad
\zeta \equiv 1 \ \mbox{on} \ U^{\circ}_4, \quad
\zeta \equiv 0 \quad \mbox{in} \ \Rd\setminus U^{\circ}_3, \quad 
\left\| \nabla^2 \zeta \right\|_{L^\infty(\Rd)} \leq C3^{-2m},
\end{equation}
and $\psi$ is a mollifier on length scale $3^l$, that is, $\psi\in C^\infty(\Rd)$ satisfies 
\begin{equation} 
\label{e.mollpsi}
\psi \geq 0, \quad
\int_{\Rd} \psi(x)\,dx =1, \quad
\psi \equiv 0 \quad \mbox{on} \ \Rd \setminus B_{3^l}, \quad 
\left\| \nabla^2 \psi \right\|_{L^\infty(\Rd)} \leq C 3^{-(d+2)l}. 
\end{equation}

\smallskip

The purpose of the cutoff function~$\zeta$ in the definition of~$T$ is the enforcement of the boundary condition: notice that~$T \in f+H^1_0(U)$. It also cuts off the second term in the definition of~$T$ in the boundary layer where we do not have good estimates on~$\phi_e$ and $\nabla \whom$. The reason for convolving ~$\nabla \whom$ with the mollifier~$\psi$ is that we wish to differentiate $T$ but only possess $C^{0,\beta}$ estimates on~$\nabla \whom$ (cf.~\eqref{e.whom.C1beta}). Mollifying it on a mesoscopic scale does not change $\nabla \whom$ very much, since it is a macroscopic function, but gives us some control on its gradient. Indeed, by~\eqref{e.whom.C1beta},~\eqref{e.mollpsi} and~\eqref{e.mesoscales}, we have, for small enough $\smallpara$,  
\begin{align} 
\label{e.nablawhomconv0}
3^{-l} \left\| \whom \ast \psi -  \whom \right\|_{L^\infty(U^{\circ}_3)} & 
\leq 
C  3^{n (\smallpara^2 /\alpha) (d-\sigma)}\left\| \nabla f \right\|_{\underline{L}^2(U)},
\end{align}
\begin{align} 
\label{e.nablawhomconv1}
\left\| \nabla \whom \ast \psi - \nabla \whom \right\|_{L^\infty(U^{\circ}_3)}
&
\leq 
C 3^{\beta l} \left[  \nabla \whom \right]_{C^{0,\beta}(U^{\circ}_2)} 
\\ \notag & 
\leq 
C 3^{-\beta (m-l) }
3^{n (\smallpara^2 /\alpha)(d-\sigma)} \left\| \nabla f \right\|_{\underline{L}^2(U)}
\\ \notag & 
\leq 3^{- n \alpha \smallpara (d-\sigma)}  \left\| \nabla f \right\|_{\underline{L}^2(U)},
\end{align}
and, for $j \in \N$, 
\begin{align} 
\label{e.nablawhomconv2}
\left\| \nabla^j \left( \nabla \whom \ast \psi \right) \right\|_{L^\infty(U^{\circ}_3)} 
& 
\leq 
\left\| \nabla \whom \right\|_{L^\infty(U^{\circ}_2)} \left\| \nabla^j \psi \right\|_{L^1(\Rd)}
\\ \notag & 
\leq 
C_j 3^{-lj}
3^{(\smallpara^2 /\alpha) (d-\sigma)}\left\| \nabla f \right\|_{\underline{L}^2(U)}.
\end{align}

\smallskip

The proof of Proposition~\ref{p.homog.locallystationary} now breaks into two basic steps, which are presented in Lemmas~\ref{l.plugT} and~\ref{l.stripcorrectors}: first we show that $\nabla T$ is close to $\nabla \tilde{w}$ in a strong norm; second, we show that the gradient and flux of~$T$ are close to the gradient and homogenized flux of $\nabla \whom$ in weak norms. Each of these steps relies on the homogenization estimates for the functions~$\phi_{e,z}$ from the previous section: see~\eqref{e.ests.phiez}. 

\begin{lemma}
\label{l.plugT}
There exist~$\alpha(\delta,\data)\in \left(0,\tfrac12\right]$ and $C(\sigma,\delta,\mathsf{M},U_0,\data)<\infty$ such that 
\begin{multline}
\label{e.nablaTtotildew}
\left\| \nabla T - \nabla \tilde{w} \right\|_{\underline{L}^2(U)}  + \left\| 
\nabla \cdot 
\left( \a\nabla T \right) 
\right\|_{\underline{H}^{-1}(U)}
\\
\leq 
C \left\| \nabla f \right\|_{\underline{L}^{2+\delta}(U)} 
\left(
 3^{-n\alpha \smallpara^2 (d-\sigma)}  + \X 3^{-n\sigma'}\right).
\end{multline}
\end{lemma}

\begin{lemma}
\label{l.stripcorrectors}
There exist~$\alpha(\delta,\data)\in \left(0,\tfrac12\right]$ and $C(\sigma,\delta,\mathsf{M},U_0,\data)<\infty$ such that 
\begin{multline}
\label{e.stripcorrectors}
\frac1r \left( 
\left\| \nabla T - \nabla \whom \right\|_{\underline{H}^{-1}(U)} 
+
\left\| \a \nabla T - D^2\overline{L}(\nabla \uhom)  \nabla \whom \right\|_{\underline{H}^{-1}(U)} 
\right)
\\
\leq 
C \left\| \nabla f \right\|_{\underline{L}^{2+\delta}(U)} 
\left(
r^{-\alpha \smallpara^2 (d-\sigma)} + \X r^{-\sigma'} \right).
\end{multline}
\end{lemma}

Before giving the proofs of these lemmas, we compute the gradient of~$T$ and discard some terms which (with good choices of the mesoscale parameters) are small. 

\begin{lemma}
\label{l.gradientT}
There exist~$\alpha(\delta,\data)\in\left(0,\tfrac12 \right]$ and $C(\sigma,\delta,\mathsf{M},U_0,\data)<\infty$ and a random variable $\X$ satisfying $\X=\O_1(C)$ such that  if $\smallpara \leq \alpha$, 
\begin{multline}
\label{e.nablaTgarbage}
\left\| 
\nabla T 
- \zeta 
\sum_{z\in 3^{l-1}\Zd} \sum_{j=1}^d \chi(\cdot-z)  
(\partial_{x_j} \whom \ast \psi) \left( e_j + \nabla \phi_{e_j,z} \right) 
\right\|_{\underline{L}^2(U)} 
\\
\leq 
C \left\| \nabla f \right\|_{\underline{L}^{2+\delta}(U)} \left( 3^{-n\alpha \smallpara^2 (d-\sigma)} + \X 3^{-n \sigma'} \right). 
\end{multline}
\end{lemma}
\begin{proof}
To shorten the notation, we write $\sum_z$ and $\sum_j$ in place of $\sum_{z\in 3^{l-1}\Zd}$ and $\sum_{j=1}^d$, respectively, and $\sum_{z,j} = \sum_z\sum_j$. We also denote~$\chi_z := \chi(\cdot-z)$. 

\smallskip

A direct computation yields  
\begin{align}
\label{e.nablaT}
\lefteqn{
\nabla T    
- 
\zeta 
\sum_{z,j} \chi_z (\partial_{x_j} \whom \ast \psi) \left( e_j + \nabla \phi_{e_j,z} \right) 
} \quad & 
\\ \notag & 
= \left(1-\zeta \right) \nabla \whom  + \left(\whom \ast \psi - \whom \right) \nabla \zeta
\\ \notag & \quad
+  
\sum_{z,j} \left( \nabla \left(\zeta \chi_z \right)  \left( \partial_{x_j} \whom \ast \psi \right) + \zeta \chi_z \nabla \left( \partial_{x_j} \whom \ast \psi \right)  \right) \phi_{e_j,z}.
\end{align}
To prove the lemma, we will show that each of the three terms on the right side of~\eqref{e.nablaT} is bounded by the right side of~\eqref{e.nablaTgarbage}. 
The first term is small because it is confined to a boundary layer: we use the H\"older inequality, the Meyers estimate~\eqref{e.meyers.wwhom},~\eqref{e.cutoffzeta}, and~\eqref{e.mesoscales} to get 
\begin{align*}
\left\|  \left(1-\zeta \right) \nabla \whom \right\|_{\underline{L}^2(U)}
& 
\leq 
\left\| 1-\zeta \right\|_{\underline{L}^{\frac{4+2\delta}{\delta}}(U)} 
\left\| \nabla \whom  \right\|_{\underline{L}^{2+\delta}(U)}
\\ & 
\leq 
C \left( \frac{\left| U \setminus U_4^\circ \right|}{\left| U \right|} \right)^{\frac{\delta}{4+2\delta}}  \left\| \nabla f  \right\|_{\underline{L}^{2+\delta}(U)}
\\ & 
\leq C 3^{-(n-m)\delta/(4+2\delta)} \left\| \nabla f  \right\|_{\underline{L}^{2+\delta}(U)}.
\\ & 
\leq C 3^{-n\alpha \smallpara^2 (d-\sigma)} \left\| \nabla f  \right\|_{\underline{L}^{2+\delta}(U)}.
\end{align*}
To estimate the second term, we use that $\nabla \zeta$ is supported in $U_3^\circ\setminus U_4^\circ$, and obtain by~\eqref{e.cutoffzeta} and~\eqref{e.nablawhomconv0} that
\begin{multline*} 
\left\|\left(\whom \ast \psi - \whom \right) \nabla \zeta \right\|_{\underline{L}^2(U)} \leq 
C \left\|\nabla \zeta \right\|_{L^\infty(\R^d)} \left\| \whom \ast \psi - \whom \right\|_{\underline{L}^\infty(U_3^\circ)}
\\ \leq C 3^{-(m-l) }3^{n (\smallpara^2 /\alpha) (d-\sigma)} \left\| \nabla f \right\|_{\underline{L}^2(U)} \leq C 3^{- n \alpha \smallpara(d-\sigma) }  \left\| \nabla f \right\|_{\underline{L}^2(U)}.
\end{multline*}
Combining~\eqref{e.chi},~\eqref{e.cutoffzeta}, and~\eqref{e.nablawhomconv2}, we also get
\begin{multline*} 
\sup_{z,j} \left\| \nabla \left(\zeta \chi_z \right)  \left( \partial_{x_j} \whom \ast \psi \right) + \zeta \chi_z \nabla \left( \partial_{x_j} \whom \ast \psi \right)  \right\|_{L^\infty(\R^d)} \\ \leq C 3^{-l} 3^{n (\smallpara^2 /\alpha) (d-\sigma)}\left\| \nabla f \right\|_{\underline{L}^2(U)}.
\end{multline*}
Using this and~\eqref{e.phiez.alsoL2}, we estimate the final term as
\begin{align} \notag 
\lefteqn{\left\|  \sum_{z,j} \left( \nabla \left(\zeta \chi_z \right)  \left( \partial_{x_j} \whom \ast \psi \right) + \zeta \chi_z \nabla \left( \partial_{x_j} \whom \ast \psi \right)  \right) \phi_{e_j,z} \right\|_{\underline{L}^2(U)}} \quad &
\\ \notag & \leq C 3^{n (\smallpara^2 /\alpha) (d-\sigma)} \left\| \nabla f \right\|_{\underline{L}^2(U)} \frac{|\cu_{l+1} |}{|U|} \sum_{z,j} 3^{-l} 
\left\| \phi_{e_j,z} \right\|_{\underline{L}^2(z + \cu_{l+1})} 
\\ \notag & \leq C 3^{n (\smallpara^2 /\alpha) (d-\sigma)} \left\| \nabla f \right\|_{\underline{L}^2(U)} \left( 3^{-n \alpha (d-\sigma) } + \X_z 3^{-n \sigma'} \right).
\end{align}
Taking $\alpha$ smaller, if necessary, yields the lemma.  
\end{proof}

We now turn to the proof of Lemma~\ref{l.plugT}.

\begin{proof}[{Proof of Lemma~\ref{l.plugT}}]
The rough idea is to plug~$T$ into the equation for~$\tilde{w}$ and check that the error we make is small: the claim is that 
\begin{equation}
\label{e.Tplugineq}
\left\| 
\nabla \cdot 
\left( \a\nabla T \right) 
\right\|_{\underline{H}^{-1}(U)}
\leq 
C \left\| \nabla f \right\|_{\underline{L}^{2+\delta}(U)}   \left( 3^{-n\alpha \smallpara^2 (d-\sigma)} + \X 3^{-n \sigma'} \right).
\end{equation}
The proof of~\eqref{e.Tplugineq} occupies the first three steps below. In the fourth step we use~\eqref{e.Tplugineq} to obtain the lemma. As above we write $\sum_z$ and $\sum_j$ in place of $\sum_{z\in 3^{l-1}\Zd}$ and $\sum_{j=1}^d$, respectively, to shorten the notation. We also write $\chi_z := \chi(\cdot-z)$.

\smallskip

\emph{Step 1.} We organize the computation. 
According to~\eqref{e.nablaTgarbage},  
\begin{align} \label{e.splittingcomp0}
\lefteqn{
\left\| 
\nabla \cdot \a \left( \nabla T 
- \zeta \sum_{z,j} \chi_z 
\left( \partial_{x_j} \whom \ast\psi\right) \left(e_j+ \nabla \phi_{e_j,z} \right)
\right) 
\right\|_{\underline{H}^{-1}(U)} 
} \quad &
\\ \notag & 
\leq 
C \left\| 
\nabla T 
- \zeta 
\sum_{z,j} \chi_z  
(\partial_{x_j} \whom \ast \psi) \left( e_j + \nabla \phi_{e_j,z} \right) 
\right\|_{\underline{L}^2(U)} 
\\ \notag & 
\leq 
C \left\| \nabla f \right\|_{\underline{L}^{2+\delta}(U)}  \left( 3^{-n\alpha \smallpara^2 (d-\sigma)} + \X 3^{-n \sigma'} \right).
\end{align}
It therefore suffices to estimate the $H^{-1}$ norm of
\begin{align} 
\label{e.splittingcomp}
\lefteqn{\nabla \cdot  \left(   
\sum_{z,j}  \zeta  \chi_z (\partial_{x_j} \whom \ast \psi) \a \left( e_j + \nabla \phi_{e_j,z} \right)  \right)} \quad &
\\ \notag & = 
\sum_{z,j}  ( \partial_{x_j} \whom \ast \psi) \zeta \chi_z \nabla \cdot \a_z \left( e_j + \nabla \phi_{e_j,z} \right)
\\ \notag & \quad 
+ \sum_{z,j}  \nabla \left( (\partial_{x_j} \whom \ast \psi) \zeta \chi_z \right) \cdot   \a_z \left( e_j + \nabla \phi_{e_j,z} \right) 
\\ \notag & \quad 
+ \nabla \cdot \left(  \zeta  \sum_{z,j} \chi_z (\partial_{x_j} \whom \ast \psi) \left(\a - \a_z \right) \left( e_j + \nabla \phi_{e_j,z} \right) \right),
\end{align}
where, for each $z\in 3^{l-1}\Zd$, we abuse notation by defining
\begin{equation*}
\ahom_z := \ahom_{\left( \nabla \uhom \right)_{z+\cu_l}}. 
\end{equation*}
The first term on the right side of~\eqref{e.splittingcomp} is zero by the equation of $\phi_{e_j,z}$, and thus the following steps are devoted to estimates of the other two terms. 

\smallskip

\emph{Step 2.} We show that 
\begin{multline} \label{e.splitsecondterm000}
\left\| \sum_{z,j}  \nabla \left( (\partial_{x_j} \whom \ast \psi) \zeta \chi_z \right) \cdot   \a_z \left( e_j + \nabla \phi_{e_j,z} \right) \right\|_{\Hminusul(U)} 
\\ \leq  C \left\| \nabla f \right\|_{\underline{L}^{2+\delta}(U)}\left( 3^{-n\alpha \smallpara^2 (d-\sigma)} + \X 3^{-n \sigma'} \right).
\end{multline}
We decompose each summand on the left as 
\begin{align} 
\label{e.splitsecondterm001}
\lefteqn{
\nabla \left( (\partial_{x_j} \whom \ast \psi) \zeta \chi_z \right) \cdot   \a \left( e_j + \nabla \phi_{e_j,z} \right)  
} \quad & 
\\ \notag & 
=   \nabla \left( (\partial_{x_j} \whom \ast \psi) \zeta \chi_z \right) \cdot \ahom_z e_j  
\\ \notag & \quad 
+ \nabla \left( (\partial_{x_j} \whom \ast \psi) \zeta \chi_z \right) \cdot   \left( \a \left( e_j + \nabla \phi_{e_j,z} \right) - \ahom_z e_j  \right)   .
\end{align}
It is in the estimate of the first term that we use the equation for $\whom$. To see this, write
\begin{equation*} 
\sum_{z,j}  \nabla \left( (\partial_{x_j} \whom \ast \psi) \zeta \chi_z \right) = \sum_z \nabla \cdot \left(   \ahom_z \chi_z  (\nabla \whom \ast \psi) \zeta  \right) ,
\end{equation*}
and decompose the term on the right as 
\begin{align}
\label{e.moresplitting}
\sum_z \nabla \cdot \left(   \ahom_z \chi_z  (\nabla \whom \ast \psi) \zeta  \right) & =   
\nabla \cdot \left( \zeta D_p^2 \overline{L} \left( \nabla \uhom \right)  \nabla \whom  \right) 
\\ \notag & \quad
 +
\nabla \cdot \left( \zeta \sum_z \chi_z \ahom_z \left( \nabla \whom\ast \psi -  \nabla \whom \right) \right)  
\\ \notag & \quad
-
\nabla \cdot \left( \zeta \sum_z \chi_z \left( D_p^2 \overline{L} \left( \nabla \uhom \right) -  \ahom_z \right)  \nabla \whom  \right).
\end{align}
By the equation for $\whom$, we have
\begin{equation*}
\nabla \cdot \left( \zeta D_p^2 \overline{L} \left( \nabla \uhom \right)  \nabla \whom  \right)  
=
\nabla \cdot \left( (\zeta-1) D_p^2 \overline{L} \left( \nabla \uhom \right) \nabla \whom  \right) ,
\end{equation*}
and thus, by the H\"older inequality,~\eqref{e.meyers.wwhom} and~\eqref{e.cutoffzeta},
\begin{align*}
\left\| 
\nabla \cdot \left( \zeta D_p^2 \overline{L} \left( \nabla \uhom \right)  \nabla \whom  \right) 
\right\|_{\Hminusul(U)}
&
\leq
\left\| 
 (1-\zeta)   D_p^2 \overline{L} \left( \nabla \uhom \right)  \nabla \whom
\right\|_{\underline{L}^{2}(U)}
\\ & 
\leq 
\left\| 1-\zeta \right\|_{\underline{L}^{\frac{4+2\delta}{\delta}}(U)}
\left\| \nabla \whom \right\|_{\underline{L}^{2+\delta}(U)}
\\ & 
\leq 
C 3^{-n\alpha \smallpara^2 (d-\sigma)}  \left\| \nabla f \right\|_{\underline{L}^{2+\delta}(U)}.
\end{align*}
We estimate the $\Hminusul (U)$ norm of the other two terms on the right side of~\eqref{e.moresplitting} by showing that what is under the divergence sign is small in $L^2(U)$. We have, by~\eqref{e.nablawhomconv1},
\begin{align*}
\lefteqn{
\left\| 
\nabla \cdot \left( \zeta \sum_z \chi_z \ahom_z
\left( \nabla \whom\ast \psi -  \nabla \whom \right) \right)
\right\|_{\Hminusul (U)}  
} \qquad & 
\\ \notag & 
\leq
\left\| 
\zeta \sum_z \chi_z \ahom_z \left( \nabla \whom\ast \psi -  \nabla \whom \right)
\right\|_{\underline{L}^{2}(U)} 
\\ \notag & 
\leq
C \left\| 
 \nabla \whom\ast \psi -  \nabla \whom
\right\|_{\underline{L}^{2}(U^{\circ}_3)} 
\\ \notag & 
\leq 
C 3^{- n \alpha \smallpara (d-\sigma)}   \left\| \nabla f \right\|_{\underline{L}^2(U)}
\end{align*}
and, similarly, by the H\"older inequality,~\eqref{e.meyers.wwhom},~\eqref{e.uhom.C1beta} and~\eqref{e.ahomxicont}, 
\begin{align*}
\lefteqn{
\left\|
\nabla \cdot \left( \zeta \sum_z \chi_z \left( D_p^2 \overline{L} \left( \nabla \uhom \right) -  \ahom_z \right) \nabla  \whom\right)
\right\|_{\Hminusul (U)}^2
} \qquad & 
\\ & 
\leq 
\left\|
 \zeta \sum_z \chi_z \left( D_p^2 \overline{L} \left( \nabla \uhom \right) -  \ahom_z \right)   \nabla \whom
\right\|_{\underline{L}^{2}(U)}
\\ \notag & 
\leq 
\sup_{z \in 3^{l-1} \Z^d \cap U_3^\circ} \left\| D_p^2 \overline{L} \left( \nabla \uhom \right) - \ahom_z  
\right\|_{L^\infty(z + \cu_{l+1})}
\left\|
\nabla \whom
\right\|_{\underline{L}^{2}(U)}.
\end{align*}
To estimate the first factor on the right, we use  Lemma~\ref{l.localahom} and~\eqref{e.uhom.C1beta}, noticing also that for every $z \in 3^{l-1}\Zd\cap U_3^\circ$ we have $z+\cu_{l+1} \subseteq U^\circ_1$, to obtain
\begin{align} \notag 
\lefteqn{
\left\| D_p^2 \overline{L} \left( \nabla \uhom \right) - \ahom_z \right\|_{L^\infty(z + \cu_{l+1})}
} \qquad &
\\ \notag  &
\leq 
\left\| D_p^2 \overline{L} \left( \left( \nabla \uhom \right)_{z+\cu_l} \right) - \ahom_{\left( \nabla \uhom \right)_{z+\cu_l} } 
\right\|_{L^\infty(z + \cu_{l+1})}
\\ \notag & \quad 
+
\left\|  D_p^2 \overline{L} \left( \nabla \uhom \right) -  D_p^2 \overline{L} \left( \left( \nabla \uhom \right)_{z+\cu_l} \right) 
\right\|_{L^\infty(z + \cu_{l+1})}
\\ \notag  &
\leq C (1+\left\| \nabla \uhom \right\|_{L^\infty(U_1^\circ)}) 3^{-\alpha k}
+ 
C
\left( 
3^{\beta l} \left[
\nabla \uhom
\right]_{C^{0,\beta}(U^{\circ}_1)} 
\right)^{\beta} 
\\ \notag  &
\leq 
C ( 3^{-\alpha k} + 3^{-\beta(m-l)} ) 3^{n d \smallpara^2 (d-\sigma)}.
\end{align}
Therefore, collecting the above estimates, we obtain by~\eqref{e.moresplitting} and~\eqref{e.mesoscales} that
\begin{equation} \label{e.splitsecondterm002}
\left\| \sum_z \nabla \cdot \left(   \ahom_z \chi_z  (\nabla \whom \ast \psi) \zeta  \right) \right\|_{\Hminusul (U)} 
\leq C \left\| \nabla f \right\|_{\underline{L}^{2}(U)}  3^{-n\alpha \smallpara^2 (d-\sigma)} .
\end{equation}
To estimate the second term on the right side of~\eqref{e.splitsecondterm001}, 
we use~\eqref{e.ests.phiez} and~\eqref{e.computeHm1} to obtain, for each $z\in 3^{l-1}\Zd$,  
\begin{align*} 
& 
\left\| \nabla \left( (\partial_{x_j} \whom \ast \psi) \zeta \chi_z \right) \cdot   \left( \a_z \left( e_j + \nabla \phi_{e_j,z} \right) - \ahom_z e_j  \right)  \right\|_{\Hminusul \left(z+\cu_{l}\right)}
\\ & \qquad 
\leq C\left(  
\left\| \nabla \left( (\partial_{x_j} \whom \ast \psi) \zeta \chi_z \right)  \right\|_{L^{\infty}\left(z+\cu_{l}\right)}
+
3^{l} \left\| \nabla^2 \left( (\partial_{x_j} \whom \ast \psi) \zeta \chi_z \right)  \right\|_{L^{\infty}\left(z+\cu_{l}\right)}  
\right) 
\\ & \qquad \qquad  \times 
3^l \left( 3^{-n \alpha (d-\sigma) } + \X_z 3^{-n \sigma'} \right).
\end{align*}
Applying~\eqref{e.nablawhomconv2},~\eqref{e.chi}, and~\eqref{e.cutoffzeta} we see that 
\begin{multline*}
\left\| \nabla \left( (\partial_{x_j} \whom \ast \psi) \zeta \chi_z \right)  \right\|_{L^{\infty}\left(z+\cu_{l}\right)}
+
3^{l} \left\| \nabla^2 \left( (\partial_{x_j} \whom \ast \psi) \zeta \chi_z \right)  \right\|_{L^{\infty}\left(z+\cu_{l}\right)} 
\\
\leq 
C3^{-l+n(\smallpara \beta_2/\alpha) (d-\sigma)} \left\| \nabla f \right\|_{\underline{L}^2(U)}. 
\end{multline*}
We now obtain~\eqref{e.splitsecondterm000} by combining the two previous displays with~\eqref{e.splitsecondterm001},~\eqref{e.splitsecondterm002} and~\eqref{e.Hminusul.subadd}, the latter in view of the fact that
\begin{equation*} 
\frac{|\cu_{l}|}{|U|}\sum_{z \in 3^{l-1}\Zd\cap U_3^\circ}  \X_z = \O_{1}(C),
\end{equation*}
and then taking the parameter $\smallpara>0$ sufficiently small. 

\smallskip

\emph{Step 3.}
We estimate the $H^{-1}(U)$ norm of the divergence of the third term on the right side of~\eqref{e.splittingcomp}. The claim is that 
\begin{multline}
\label{e.splittingcomp.est1}
\left\| 
\nabla \cdot \left( 
\zeta   
\sum_{z,j} \chi_z
(\partial_{x_j} \whom \ast \psi) \left(\a - \a_z \right) \left( e_j + \nabla \phi_{e_j,z} \right)
\right)
\right\|_{\underline{H}^{-1}(U)} 
\\
\leq 
C \left\| \nabla f \right\|_{\underline{L}^2(U)} 
 3^{-n\alpha \smallpara^2 (d-\sigma)} . 
\end{multline}
For this it is enough to bound the $L^2$ norm of the term inside of the parenthesis. 
By the H\"older inequality,~\eqref{e.whom.C1beta} and~\eqref{e.phiez.meyers}, we have
\begin{align*}
\lefteqn{
\left\| 
\zeta   
\sum_{z,j} \chi_z
(\partial_{x_j} \whom \ast \psi) \left(\a - \a_z \right) \left( e_j + \nabla \phi_{e_j,z} \right)
\right\|_{\underline{L}^2(U)} 
} \qquad  &
\\ & 
\leq 
\left\| 
\zeta   
\sum_{z,j} \chi_z
(\partial_{x_j} \whom \ast \psi) \left(\a - \a_z \right) \left( e_j + \nabla \phi_{e_j,z} \right)
\right\|_{\underline{L}^2(U)} 
\\ & 
\leq
C\left\| \nabla \whom \right\|_{L^\infty(U^{\circ}_2)}
 \sum_z  \left\| \chi_z (\a - \a_z) \right\|_{\underline{L}^{\frac{4+2\delta}{\delta}} (U)} 
\left( 1 + \sum_{j}\left\| \nabla \phi_{e_j,z} \right\|_{\underline{L}^{2+\delta} \left(z+\cu_l\right)}
\right)
\\ & 
\leq 
C\left\| \nabla f \right\|_{\underline{L}^2(U)} 3^{n(\smallpara \beta_2/\alpha) (d-\sigma)}
 \sum_z  \left\| \chi_z (\a - \a_z) \right\|_{\underline{L}^{\frac{4+2\delta}{\delta}} (U)}  .
\end{align*}
Furthermore, to estimate the last sum, using~\eqref{e.uhom.C1beta},~\eqref{e.axicont}, and letting $[z]$ denote the nearest point of $3^l \Zd$ to $z\in 3^{l-1}\Zd$, we have, for $q \in [2,\infty)$,
\begin{align*}
\left\| \a - \sum_z \chi_z \a_z \right\|_{\underline{L}^q(U^{\circ}_1)} 
&
\leq 
\sum_{z\in 3^{l-1}\Zd\cap U^{\circ}_1}
\left\|  \a - \a_z \right\|_{\underline{L}^q(z+\cu_l)} 
\\ & 
\leq
\sum_{z\in 3^{l-1}\Zd\cap U^{\circ}_1}
\left\|  
\a_{\left( \nabla \uhom \right)_{[z]+\cu_l}} 
- \a_{\left( \nabla \uhom \right)_{z+\cu_l}}
\right\|_{\underline{L}^q(z+\cu_l)} 
\\ & 
\leq 
C\sum_{z\in 3^{l-1}\Zd\cap U^{\circ}_1}
\left| 
\left( \nabla \uhom \right)_{[z]+\cu_l} 
- \left( \nabla \uhom \right)_{z+\cu_l}
\right|^{\beta/q} 
\\ & 
\leq 
C \left( 3^{\beta l} \left[ \nabla \uhom \right]_{C^{0,\beta}(U^{\circ}_1)} \right)^{\beta/q}
\\ & 
\leq 
C \left( 3^{-\beta (m-l)} 3^{d(n-m)/2} \right)^{\beta/q}. 
\end{align*}
Putting the last two displays together, using~\eqref{e.mesoscales}, and taking $\alpha$ smaller, if necessary, we obtain~\eqref{e.splittingcomp.est1}.

\smallskip

\emph{Step 4.} The conclusion. We combine~\eqref{e.splittingcomp0} and~\eqref{e.splittingcomp} with the estimates~\eqref{e.splitsecondterm000} and~\eqref{e.splittingcomp.est1}, recalling that the first term on the right in~\eqref{e.splittingcomp} is zero, to deduce~\eqref{e.Tplugineq}.  
Since $T - \tilde{w} \in H^1_0(U)$, the equation for $\tilde{w}$ gives 
\begin{equation*} \label{}
\fint_{U} \nabla \left( T - \tilde{w} \right) \cdot \a \nabla \tilde{w} \,dx = 0. 
\end{equation*}
The estimate~\eqref{e.Tplugineq} yields 
\begin{align*} \label{}
\left| \fint_{U} \nabla \left( T - \tilde{w} \right) \cdot \a \nabla T \,dx \right| 
&
\leq 
C \left\| T-\tilde{w} \right\|_{\underline{H}^1(U)} 
\left\| \nabla \cdot \a \nabla T \right\|_{\underline{H}^{-1}(U)} 
\\ &
\leq 
C \left\| \nabla T- \nabla \tilde{w} \right\|_{\underline{L}^2(U)}
\left\| \nabla \cdot \a \nabla T \right\|_{\underline{H}^{-1}(U)}.
\end{align*}
Combining the previous two displays yields
\begin{align*} \label{}
\left\|  \nabla  T -\nabla \tilde{w}  \right\|_{\underline{L}^2(U)}^2
& 
\leq
\fint_{U} \nabla \left( T - \tilde{w} \right) \cdot \a  \nabla \left( T - \tilde{w} \right) \,dx 
\\ & 
\leq 
C \left\| \nabla T- \nabla\tilde{w} \right\|_{\underline{L}^2(U)}
\left\| \nabla \cdot \a \nabla T \right\|_{\underline{H}^{-1}(U)}
\end{align*}
and thus 
\begin{equation*}
\left\|  \nabla  T -\nabla \tilde{w}  \right\|_{\underline{L}^2(U)}
\leq C\left\| \nabla \cdot \a \nabla T \right\|_{\underline{H}^{-1}(U)}.
\end{equation*}
Combining this with~\eqref{e.Tplugineq} completes the proof of~\eqref{e.nablaTtotildew} and the lemma. 
\end{proof}

\smallskip

\begin{proof}[{Proof of Lemma~\ref{l.stripcorrectors}}]
We estimate the terms on the left side of~\eqref{e.stripcorrectors} by using~\eqref{e.nablaTgarbage} to reduce the desired inequalities to bounds on the functions~$\phi_{e,z}$ which are then consequences of quantitative homogenization estimates for linear equations. For convenience we denote
\begin{equation*}
\mathcal{E}':= 
 \left\| \nabla f \right\|_{\underline{L}^{2+\delta}(U)} 
\left(3^{-n\alpha \smallpara^2 (d-\sigma)} + \X 3^{-n \sigma'}  \right).
\end{equation*}

\smallskip

\emph{Step 1.} We show that 
\begin{equation}
\label{e.nablaTtowhom}
\frac1r \left\| \nabla T - \nabla \whom \right\|_{\underline{H}^{-1}(U)} 
\leq 
C\mathcal{E}'. 
\end{equation}
Using~\eqref{e.ests.phiez},~\eqref{e.nablawhomconv2} and~\eqref{e.nablaTgarbage}, we see that 
\begin{align*}
\lefteqn{
\frac1r 
\left\| 
\nabla T - \zeta \left( \nabla \whom \ast\psi \right)
\right\|_{\underline{H}^{-1}(U)}
} \qquad & 
\\ & 
\leq \frac1r \left\| \zeta \sum_{j=1}^d 
\partial_{x_j} \left(  \whom \ast\psi \right) \nabla \phi_{e_j} 
\right\|_{\underline{H}^{-1}(U)} 
+ C \mathcal{E}'
\\ & 
\leq 
\sum_{j=1}^d 
\left\| \nabla \left( \zeta \left(  \nabla\whom \ast\psi \right) \right) \right\|_{L^\infty(U)} 
\left\|  \nabla \phi_{e_j}\right\|_{\underline{H}^{-1}(U)}  
+ C \mathcal{E}'
\leq 
C \mathcal{E}'. 
\end{align*}
Next, we have, by~\eqref{e.nablawhomconv1}, 
\begin{align*}
\frac1r 
\left\| 
\zeta \left( \nabla \whom \ast\psi - \nabla\whom \right)
\right\|_{\underline{H}^{-1}(U)}
&
\leq
C \left\| 
\zeta \left( \nabla \whom \ast\psi - \nabla\whom \right)
\right\|_{\underline{L}^{2}(U)}
\\ & 
\leq 
C\left\| 
 \nabla \whom \ast\psi - \nabla\whom 
\right\|_{\underline{L}^{\infty}(U^{\circ}_3)}
\leq 
C \mathcal{E}'. 
\end{align*}
Finally, 
\begin{align*}
\frac1r 
\left\| 
(1- \zeta) \nabla \whom 
\right\|_{\underline{H}^{-1}(U)}
& 
\leq 
\left\| 
(1- \zeta) \nabla \whom 
\right\|_{\underline{L}^{2}(U)}
\\ & 
\leq 
\left\| 1-\zeta \right\|_{\underline{L}^{\frac{4+2\delta}{\delta}}(U)} \left\| \nabla \whom \right\|_{\underline{L}^{2+\delta}(U)}
\\ & 
\leq 
C 3^{-n\alpha \smallpara^2 (d-\sigma)} \left\| \nabla f \right\|_{\underline{L}^{2+\delta}(U)}
\leq 
C \mathcal{E}'. 
\end{align*}
The triangle inequality and the previous three displays yield~\eqref{e.nablaTtowhom}. 

\smallskip

\emph{Step 2.} To prepare for the estimate for the fluxes, we show that 
\begin{equation}
\label{e.gethomogcoeff}
\left\| D^2\overline{L}(\nabla \uhom) - \sum_{z\in 3^{l-1}\Zd} \chi_z \ahom_z \right\|_{L^\infty(U^\circ_3)}
\leq 
C3^{-n\alpha  \smallpara(d-\sigma)},
\end{equation}
where we abuse notation by defining, for each $z\in 3^{l-1}\Zd$, 
\begin{equation*}
\ahom_z:= \ahom_{\left( \nabla \uhom \right)_{z+\cu_l}}. 
\end{equation*}
By~\eqref{e.Lbar.C2beta}, Lemma~\ref{l.localahom} and~\eqref{e.uhom.C1beta}, we have that, for every $z\in 3^{l-1}\Zd\cap U^{\circ}_2$ and $x\in z+\cu_l$, 
\begin{align*}
\left| 
D^2\overline{L}(\nabla \uhom(x))
-
\ahom_z \right| 
&
\leq 
C 3^{-n\alpha \smallpara(d-\sigma)} + \left| 
D^2\overline{L}(\nabla \uhom(x))
-
D^2\overline{L}(\nabla \uhom(z)) \right| 
\\ & 
\leq
C 3^{-n\alpha \smallpara(d-\sigma)}
+ 
\left[ D^2\overline{L} \right]_{C^{0,\beta}(\Rd)} 
\left| \nabla \uhom(x) - \nabla \uhom(z) \right|^\beta
\\ & 
\leq 
C3^{-n\alpha \smallpara(d-\sigma)} 
+ 
C\left[ D^2\overline{L} \right]_{C^{0,\beta}(\Rd)} 
\left( 3^{\beta l} \left[ \nabla \uhom \right] \right)^\beta
\\ & 
\leq
C 3^{-n\alpha \smallpara(d-\sigma)} 
+ 
C \left( 
3^{-\beta(m-l)} 3^{d(n-m)/2}
\right)^{\beta}.
\end{align*}
After redefining $\alpha$, by~\eqref{e.mesoscales} this becomes
\begin{equation*}
\sup_{z\in 3^{l-1}\Zd\cap U^{\circ}_2}
\sup_{x\in z+\cu_l}
\left| 
D^2\overline{L}(\nabla \uhom(x))
-
\ahom_z \right| 
\leq 
C 3^{-n\alpha \smallpara(d-\sigma)} 
+ 
C 3^{-n\alpha \smallpara(d-\sigma)} .
\end{equation*}
This yields~\eqref{e.gethomogcoeff} after summing over $z$ by the triangle inequality, recalling also that $\smallpara \leq \smallpara$. 

\smallskip

\emph{Step 3.} We show that 
\begin{equation}
\label{e.nablaTtowhomflux}
\frac1r \left\| \a \nabla T - D^2\overline{L}\left(\nabla \uhom \right) \nabla \whom \right\|_{\underline{H}^{-1}(U)} 
\leq 
C \mathcal{E}'. 
\end{equation}
Let $\tilde{\a}:= \sum_{z\in 3^{l-1}\Zd} \chi_z \ahom_z$. 
By~\eqref{e.ests.phiez},~\eqref{e.nablawhomconv2} and~\eqref{e.nablaTgarbage}, 
\begin{align*}
\lefteqn{
\frac1r 
\left\| 
\a \nabla T - \zeta \tilde{\a} \left( \nabla \whom \ast\psi \right)
\right\|_{\underline{H}^{-1}(U)}
} \qquad & 
\\ & 
\leq \frac1r \left\| \zeta \sum_{j=1}^d
\partial_{x_j} \left(  \whom \ast\psi \right) \left( \a\left( e_j +  \nabla \phi_{e_j} \right) - \tilde{\a} e_j \right) 
\right\|_{\underline{H}^{-1}(U)} 
+ C \mathcal{E}'
\\ & 
\leq 
\sum_{j=1}^d 
\left\| \nabla \left( \zeta \left(  \nabla\whom \ast\psi \right) \right) \right\|_{L^\infty(U)}
\left\| 
\a\left( e_j +  \nabla \phi_{e_j} \right) - \tilde{\a} e_j
\right\|_{\underline{H}^{-1}(U)}  
+ C \mathcal{E}'
\leq C \mathcal{E}'. 
\end{align*}
Next we use~\eqref{e.whom.C1beta} and~\eqref{e.gethomogcoeff} to see that 
\begin{align*}
\lefteqn{
\frac1r 
\left\| 
\zeta \left( D^2\overline{L}(\nabla \uhom) - \tilde{\a} \right) \left( \nabla \whom \ast\psi \right)
\right\|_{\underline{H}^{-1}(U)}
} \quad & 
\\ & 
\leq 
C
\left\| 
\zeta \left( D^2\overline{L}(\nabla \uhom) - \tilde{\a} \right) \left( \nabla \whom \ast\psi \right)
\right\|_{\underline{L}^{2}(U)}
\\ & 
\leq 
C
\left\| D^2\overline{L}(\nabla \uhom) - \tilde{\a} \right\|_{L^\infty(U^{\circ}_3)} 
\left\| \nabla \whom \right\|_{\underline{L}^2(U^\circ_2)}
\\ & 
\leq 
C\left\| \nabla f \right\|_{\underline{L}^2(U)}  3^{-n \alpha  \smallpara(d-\sigma) \left( 1 -   \smallpara/\alpha^2 \right)}
\leq
C \mathcal{E}',
\end{align*}
for sufficiently small~$\smallpara$. 
Finally, as in Step~1, we have
\begin{align*}
\frac1r 
\left\| 
\zeta D^2\overline{L}(\nabla\uhom) \left( \nabla \whom \ast\psi - \nabla\whom \right)
\right\|_{\underline{H}^{-1}(U)}
&
\leq 
\frac1r 
\left\| 
\zeta \left( \nabla \whom \ast\psi - \nabla\whom \right)
\right\|_{\underline{H}^{-1}(U)}
\\ &
\leq C\mathcal{E}'
\end{align*}
and
\begin{equation*}
\frac1r 
\left\| 
(1- \zeta) D^2\overline{L}(\nabla\uhom) \nabla \whom 
\right\|_{\underline{H}^{-1}(U)}
\leq 
\frac1r 
\left\| 
(1- \zeta) \nabla \whom 
\right\|_{\underline{H}^{-1}(U)}
\leq 
C \mathcal{E}'.
\end{equation*}
Combining the three previous displays yields~\eqref{e.nablaTtowhomflux}. 
\end{proof}

\begin{proof}[{Proof of Proposition~\ref{p.homog.locallystationary}}]
The statement is an immediate consequence of the triangle inequality and Lemmas~\ref{l.plugT} and~\ref{l.stripcorrectors}.
\end{proof}

\subsection{The conclusion}

We conclude the proof of Theorem~\ref{t.linearization} by  summarizing how the previous lemmas fit together to give the theorem. 

\begin{proof}[{Proof of Theorem~\ref{t.linearization}}]
As discussed above, it suffices to prove~\eqref{e.DPestimates.rescale}. For the  estimate of the gradient term, we have
\begin{align*}
\frac1r \left\| \nabla w - \nabla \whom  \right\|_{\underline{H}^{-1}(U)} 
&
\leq 
\left\| \nabla w - \nabla \tilde{w}  \right\|_{\underline{L}^{2}(U)} 
+ 
\frac1r \left\| \nabla \tilde{w} - \nabla \whom  \right\|_{\underline{H}^{-1}(U)} 
\\ & 
\leq 
C \left\| \nabla f \right\|_{\underline{L}^{2+\delta}(U)} \left(  3^{-n \alpha \smallpara^2 (d-\sigma)^2}+ \X 3^{-n\sigma} \right).
\end{align*}
Here we used the triangle inequality,~\eqref{e.H1L2dumbdumb}, Lemma~\ref{l.wtowtilde} and Proposition~\ref{p.homog.locallystationary}. For the fluxes, the triangle inequality and~\eqref{e.H1L2dumbdumb} yield
\begin{align*}
\lefteqn{
 \frac1r\left\| D_p^2L\left( \nabla u,\cdot \right) \nabla w - D_p^2\overline{L} \left( \nabla \uhom \right) \nabla \whom   \right\|_{\underline{H}^{-1}(U)}
 } \qquad & 
\\ & 
\leq
\left\| \left( D^2_pL(\nabla u,x) - \a \right) \nabla w \right\|_{\underline{L}^2(U)}
+
\left\| \a \right\|_{L^\infty(U)} 
\left\| \nabla w - \nabla \tilde{w}  \right\|_{\underline{L}^{2}(U)} 
\\ & \qquad
+
\frac1r\left\| \a \nabla \tilde{w} - D^2\overline{L}(\nabla \uhom)  \nabla \whom \right\|_{\underline{H}^{-1}(U)} .
\end{align*}
The three terms on the right are controlled by~\eqref{e.switchflux}, Lemma~\ref{l.wtowtilde} and Proposition~\ref{p.homog.locallystationary}, respectively. 
We have shown that the left side of~\eqref{e.DPestimates.rescale} is bounded by
\begin{equation*}
C \left\| \nabla f \right\|_{\underline{L}^{2+\delta}(U)} \left(  3^{-n \alpha \smallpara^2 (d-\sigma)^2}+ \X 3^{-n\sigma} \right)
\end{equation*}
for~$\smallpara$ specified in the statements of~\eqref{e.switchflux}, Lemma~\ref{l.wtowtilde} and Proposition~\ref{p.homog.locallystationary}. Recalling that  $r \geq 3^{n-1}$, we obtain~\eqref{e.DPestimates.rescale} and hence the theorem. 
\end{proof}

\subsection{A corollary}

In view of its application in the next section, we finish this section by restating the result of Theorem~\ref{t.linearization} in ``minimal scale form'' in the case that the domain~$U$ is a ball.

\begin{corollary}
\label{c.linearization}
Let $\sigma \in (0,d)$, $\delta\in \left(0,\tfrac 12\right]$ and $\mathsf{M}\in [1,\infty)$. There exist~{$\alpha(\delta,\data)\in \left(0,\tfrac12\right]$}, $C(\sigma,\delta,\mathsf{M},\data)<\infty$ and a random variable~$\X_\sigma$, satisfying the bound
\begin{equation}
\label{e.size.X_sigma.AS.2}
\X_\sigma = \O_\sigma\left(C \right)
\end{equation}
such that the following statement holds. For every~$r\in [\X_\sigma,\infty)$ and pair $u,\overline{u}\in W^{1,2+\delta}(B_r)$ satisfying 
\begin{equation*}
\left\{ 
\begin{aligned}
& -\nabla \cdot \left( D_pL\left( \nabla u,x \right) \right) = 0  & \mbox{in} & \ B_r, \\
&  -\nabla \cdot \left( D_p\overline{L} \left( \nabla \overline{u} \right) \right)  = 0 & \mbox{in} & \ B_r, \\
& u - \overline{u} \in H^1_0(B_r), \\
& 
\left\| \nabla u \right\|_{L^{2+\delta}(B_r)}
+ \left\| \nabla \overline{u} \right\|_{L^{2+\delta}(B_r)} \leq \mathsf{M},
\end{aligned}
\right.
\end{equation*}
function~$f\in W^{1,2+\delta}(B_r)$ and pair~$w, \overline{w} \in H^1(B_r)$ satisfying 
\begin{equation*}
\left\{ 
\begin{aligned}
& -\nabla \cdot \left( D_p^2L\left( \nabla u,x \right) \nabla w \right) = 0  & \mbox{in} & \ B_r, \\
&  -\nabla \cdot \left( D_p^2\overline{L} \left( \nabla \overline{u} \right) \nabla \overline{w} \right)  = 0 & \mbox{in} & \ B_r, \\
& w^\ep, \overline{w} \in f + H^1_0(B_r),
\end{aligned}
\right.
\end{equation*}
we have the estimate
\begin{equation}
\label{e.DPestimates.minscale}
\frac1r\left\|  w -  \overline{w}  \right\|_{\underline{L}^2(B_r)} 
\leq
C \left\| \nabla f \right\|_{\underline{L}^{2+\delta}(B_r)} r^{-\alpha(d-\sigma)}.
\end{equation}
\end{corollary}

\section{The large-scale \texorpdfstring{$C^{0,1}$}{C0,1}--type estimate for differences}

\label{s.regularity.differences}

In this section we prove our second main result, Theorem~\ref{t.regularity.differences}, on the large-scale $C^{0,1}$ estimate for differences. The argument follows the one introduced in~\cite{AS} for proving the large-scale $C^{0,1}$ estimate for solutions. Since differences of solutions of the homogenized equation satisfy a $C^{1,\beta}$ estimate (this is by the Schauder estimates, since $\overline{L}\in C^{2,\beta}$), we should expect to be able to transfer this higher regularity to differences of solutions of the heterogeneous equation by a excess decay iteration (as in~\cite[Lemma 5.1]{AS}). What is needed to apply this idea is an approximation result, essentially a quantitative homogenization estimate for differences, which states that differences of solutions of the heterogeneous equation are close to those of the homogenized equation. This is accomplished by interpolating between the homogenization error estimates for the linearized equation (Theorem~\ref{t.linearization}) and the nonlinear equation (Theorem~\ref{t.AS.homogenization}).

\begin{proposition}[Error estimate for differences]
\label{p.differences.error}
Fix $\mathsf{M}\in [1,\infty)$ and $\sigma \in (0,d)$. There exist~$\alpha(\data)\in \left(0,\tfrac12\right]$, ~$C(\sigma,\mathsf{M},\data)<\infty$ and a random variable $\X$ satisfying
\begin{equation*}
\X \leq \O_\sigma (C)
\end{equation*}
such that the following statement holds. For every $R\geq \X$ and pair~$u,v\in H^1(B_R)$ of solutions of the equations 
\begin{equation*}
-\nabla \cdot D_pL(\nabla u,x) = 0 
\quad \mbox{and} \quad 
-\nabla \cdot D_pL(\nabla v,x) = 0 
\quad \mbox{in} \ B_R
\end{equation*}
which satisfy
\begin{equation}  \label{e.u&v-bnd}
\frac1R \left\| u - (u)_{B_R} \right\|_{\underline{L}^2(B_R)} \leq\mathsf{M}
\quad \mbox{and} \quad 
\frac1R \left\| v - (v)_{B_R} \right\|_{\underline{L}^2(B_R)} \leq \mathsf{M},
\end{equation}
the solutions $\overline{u},\overline{v} \in H^1(B_{R/2})$ of the Dirichlet problems for the homogenized equation
\begin{equation}
\label{e.diffee}
\left\{ \begin{aligned}
& -\nabla \cdot D_p\overline{L} ( \nabla \overline{u} ) = 0 = \nabla \cdot D_p\overline{L} ( \nabla \overline{v} )
& \mbox{in} & \ B_{R/2}, \\
& \overline{u} = u, \ \overline{v} = v & \mbox{on} & \ \partial B_{R/2},
\end{aligned} \right.
\end{equation}
satisfy the estimate
\begin{equation} \label{e.differences.error}
\left\| u - v - (\overline{u} - \overline{v}) \right\|_{\underline{L}^{2}(B_{R/2})}
\leq 
C R^{-\alpha(d-\sigma)} \left\| u-v - (u-v)_{B_R} \right\|_{\underline{L}^2(B_R)}. 
\end{equation}
\end{proposition}
\begin{proof}
Let $\sigma \in (0,d)$ and choose $\X_\sigma$ to be the maximum of the random variables in the statements of Corollaries~\ref{c.minimalscale} and~\ref{c.linearization}.
Fix $R \geq \X_\sigma$ and $u,v\in H^1(B_R)$ satisfying~\eqref{e.diffee}. We split the argument in two cases: (i) the oscillation of~$u-v$ is much smaller compared those of~$u$ and~$v$, in which case we apply the homogenization result for the linearized equation, or (ii) it is not, and we apply the homogenization result for the original nonlinear equation and conclude by the triangle inequality. 

\smallskip

We fix a small parameter $\ep_0\in (0,1]$ which will be chosen in Step~3 of the proof. 

\smallskip

\emph{Step 1.} 
We consider the case that
\begin{equation} 
\label{e.u&v-case1}
\left\| u-v - (u-v)_{B_R} \right\|_{\underline{L}^2(B_R)} 
\geq 
\ep_0  \left(\left\| u - (u)_{B_R} \right\|_{\underline{L}^2(B_R)} + \left\| v - (v)_{B_R} \right\|_{\underline{L}^2(B_R)}\right)
\end{equation}
and prove that there exist~$C(\sigma,\data) < \infty$ and~$\alpha_1(d,\Lambda) \in (0,1)$ such that
\begin{equation}  \label{e.u&v-case1-res}
\left\| u - v - (\bar u - \bar v) \right\|_{\underline{L}^2(B_{R/2})} \leq C \ep_0^{-1}  R^{-\alpha_1(d-\sigma)} \left\| u-v - (u-v)_{B_R} \right\|_{\underline{L}^2(B_R)}.
\end{equation}
We take $\overline{u}, \overline{v} \in  H^1(B_{R/2})$ to be the solutions of 
\begin{equation*}
\left\{
\begin{aligned}
& -\nabla \cdot \left( D_p\overline{L} \left( \nabla \overline{u} \right)  \right) = 0 
& \mbox{in} & \ B_{R/2}, \\
& \overline{u} = u & \mbox{on} & \ \partial B_{R/2},
\end{aligned}
\right.
\quad \mbox{and} \quad 
\left\{
\begin{aligned}
& -\nabla \cdot \left( D_p\overline{L} \left( \nabla \overline{v} \right)  \right) = 0
& \mbox{in} & \ B_{R/2}, \\
& \overline{v} = v & \mbox{on} & \ \partial B_{R/2}.
\end{aligned}
\right.
\end{equation*}
As $R \geq \X_\sigma$, Corollary~\ref{c.minimalscale} and the Meyers estimate implies that there exist~$\alpha_1(d,\Lambda) \in (0,1)$ and~$C(\sigma,\data) < \infty$ such that 
\begin{align*}
\lefteqn{
\left\| u - \bar u \right\|_{\underline{L}^2(B_{R/2})} + \left\| v - \bar v \right\|_{\underline{L}^2(B_{R/2})}  
} \qquad & 
\\ & 
\leq C R^{-\alpha_1(d-\sigma)}  \left(\left\| u - (u)_{B_R} \right\|_{\underline{L}^2(B_R)} + \left\| v - (v)_{B_R} \right\|_{\underline{L}^2(B_R)}\right).
\end{align*}
Therefore, the triangle inequality and the assumption~\eqref{e.u&v-case1} together yield~\eqref{e.u&v-case1-res}.

\smallskip

\emph{Step 2.} We consider the alternative case to   the one in Step 1, namely that 
\begin{equation} 
\label{e.u&v-case2}
\left\| u-v - (u-v)_{B_R} \right\|_{\underline{L}^2(B_R)} 
\leq 
\ep_0 \left(\left\| u - (u)_{B_R} \right\|_{\underline{L}^2(B_R)} + \left\| v - (v)_{B_R} \right\|_{\underline{L}^2(B_R)}\right).
\end{equation}
We show that there exist~$\alpha(\data),\beta(\data) \in \left(0,\tfrac12\right]$ and~$C(\sigma,\mathsf{M},\data)<\infty$  such that 
\begin{align}
\label{e.yesuvcase2}
\lefteqn{
\left\| u - v - (\bar u - \bar v) \right\|_{\underline{L}^2(B_{R/2})} 
} \qquad & 
\\ & \notag
\leq
C \left( \left(\ep_0 \mathsf{M} \right)^{\beta} +   \mathsf{M} R^{-\alpha(d-\sigma)} \right) \left\| u-v - (u-v)_{B_R} \right\|_{\underline{L}^2(B_R)}.
\end{align}
Using the bound~\eqref{e.u&v-bnd}, the assumption~\eqref{e.u&v-case2} implies
\begin{equation*}
\left\| u-v - (u-v)_{B_R} \right\|_{\underline{L}^2(B_R)} 
\leq 
\ep_0  R \mathsf{M}.
\end{equation*}
The difference~$u-v$ is a solution of the equation
\begin{equation*}
-\nabla \cdot(\tilde{\a}(x) \nabla (u-v)) = 0
\quad \mbox{in} \ 
B_R. 
\end{equation*}
where 
\begin{equation*}
\tilde{\a}(x) := 
\int_0^1 
D^2_pL 
\left(t\nabla u(x) + (1-t) \nabla v(x), x \right)\,dt.
\end{equation*}
Note that $I_d \leq \tilde{\a} \leq \Lambda I_d$. 
By the Meyers estimate and the Caccioppoli inequality, there exists $\delta(d,\Lambda)\in \left(0,\tfrac12\right]$ such that 
\begin{equation} 
\label{e.uv.case2nabla}
\left\| \nabla u - \nabla v  \right\|_{\underline{L}^{2+\delta}(B_{R/2})} \leq CR^{-1} \left\| u-v - (u-v)_{B_R} \right\|_{\underline{L}^2(B_R)} \leq C \ep_0 \mathsf{M}.
\end{equation}
Analogously, since $\overline{u}-\overline{v}$ solves  $-\nabla \cdot(\hat{\a}(x) \nabla (\overline{u}-\overline{v})) = 0$ in $B_{R/2}$ with
\begin{equation*}
\hat{\a}(x) := 
\int_0^1 
D^2\overline{L} 
\left(t\nabla \overline{u} (x) + (1-t) \nabla \overline{v} (x), x \right)\,dt,
\end{equation*}
we get by the global Meyers estimate that 
\begin{align}
\label{e.uv.case2nabla2}
\lefteqn{
\left\| \nabla \overline{u} - \nabla \overline{v}  \right\|_{\underline{L}^{2+\delta}(B_{R/2})} 
} \qquad &
\\ & \notag
\leq C \left\| \nabla u - \nabla v  \right\|_{\underline{L}^{2+\delta}(B_{R/2})} \leq CR^{-1} \left\| u-v - (u-v)_{B_R} \right\|_{\underline{L}^2(B_R)}.
\end{align} 
Let $w,\overline{w} \in H^1(B_{R/2})$ be the solutions of the Dirichlet problems
\begin{equation*}
\left\{
\begin{aligned}
& -\nabla \cdot \left( D_p^2 L(\nabla u,x) \nabla w \right)  = 0  & \mbox{in} & \ B_{R/2}, \\
& w = u-v & \mbox{on} & \ \partial B_{R/2}. 
\end{aligned}
\right.
\end{equation*}
and 
\begin{equation*}
\left\{
\begin{aligned}
& -\nabla \cdot \left( D_p^2 \overline{L}(\nabla \overline{u}) \nabla \overline{w} \right)  = 0  & \mbox{in} & \ B_{R/2}, \\
& \overline{w} = \overline{u}-\overline{v}  & \mbox{on} & \ \partial B_{R/2}. 
\end{aligned}
\right.
\end{equation*}
In view of~\eqref{e.uv.case2nabla}, the interior and Global Meyers estimates as well as the fact that~$u-v-(\overline{u}-\overline{v})\in H^1_0(B_{R/2})$, we may apply Corollary~\ref{c.linearization} to obtain~$\alpha(\data) \in (0,1)$ and~$C(\sigma,\mathsf{M},\data)<\infty$ such that 
\begin{align}  \label{e.uv.case2res3}
\left\|w - \overline{w} \right\|_{\underline{L}^{2}(B_{R/2})} & \leq C R^{1-\alpha(d-\sigma)} \left\|\nabla u - \nabla v \right\|_{\underline{L}^{2+\delta}(B_{R/2})} 
\\ \notag & \leq C  R^{-\alpha(d-\sigma)} \left\| u-v - (u-v)_{B_R} \right\|_{\underline{L}^2(B_R)}.
\end{align}
On the other hand, applying Lemma~\ref{l.deterministic.linearization} we find constants $\beta(\delta,\data) > 0$ and $C(\delta,\data)<\infty$ such that 
\begin{align*}
\left\| \nabla u - \nabla v - \nabla w \right\|_{L^2(B_{R/2})} 
&
\leq 
C \left\| \nabla u - \nabla v  \right\|_{\underline{L}^{2+\delta}(B_{R/2})}^{1+\beta}
\\ &
\leq C \left(\ep_0 \mathsf{M} \right)^{\beta} R^{-1}  \left\| u-v - (u-v)_{B_R} \right\|_{\underline{L}^2(B_R)}.
\end{align*}
By Poincar\'e's inequality we then obtain 
\begin{equation}  \label{e.uv.case2res1}
\left\| u -  v - w \right\|_{\underline{L}^2(B_{R/2})} \leq C \left(\ep_0 \mathsf{M} \right)^{\beta}   \left\| u-v - (u-v)_{B_R} \right\|_{\underline{L}^2(B_R)}.
\end{equation}
A similar argument, applying the proof of Lemma~\ref{l.deterministic.linearization} to $\overline{L}$ instead of $L$, together with Proposition~\ref{p.nu.C2beta} and~\eqref{e.uv.case2nabla2}, yields
\begin{equation}  \label{e.uv.case2res2}
\left\| \overline{u}-\overline{v} - \overline{w}  \right\|_{\underline{L}^2(B_{R/2})} \leq C \left(\ep_0 \mathsf{M} \right)^{\beta}   \left\| u-v - (u-v)_{B_R} \right\|_{\underline{L}^2(B_R)}.
\end{equation}
Combining~\eqref{e.uv.case2res3},~\eqref{e.uv.case2res1} and~\eqref{e.uv.case2res2}, we obtain~\eqref{e.yesuvcase2}.

\smallskip

\emph{Step 3.} The conclusion. Combining Steps~1 and~2 and defining~$\ep_0:= R^{-\alpha_1(d-\sigma)/2}$ completes the proof. 
\end{proof}

We are now ready to complete the proof of Theorem~\ref{t.regularity.differences}. What remains is a deterministic argument, along the lines of~\cite[Lemma 5.1]{AS}, which uses a Campanato-type iteration and Proposition~\ref{p.differences.error} to transfer the higher regularity enjoyed 
by differences of solutions of the homogenized equation to differences of solutions of the heterogeneous equation.  We prove a slightly more general version of Theorem~\ref{t.regularity.differences} by including a mesoscopic $C^{1,\alpha}$ estimate.

\begin{proposition} \label{p.Conealphafordifferences}
Fix $\sigma \in (0,d)$ and $\mathsf{M}\in [1,\infty)$. There exist constants $\alpha(\data) \in \left( 0, \frac12 \right]$,~$C(\sigma,\mathsf{M},\data) <\infty$ and a random variable~$\X$ satisfying
\begin{equation*}
\X \leq \O_\sigma\left( C \right)
\end{equation*}
such that the following holds. For every $R\geq 2\X$ and $u,v\in H^1(B_R)$ satisfying 
\begin{equation*}
\left\{
\begin{aligned}
& -\nabla \cdot \left( D_pL(\nabla u,x) \right) = 0 & \mbox{in} & \ B_R, \\
& -\nabla \cdot \left( D_pL(\nabla v,x) \right) = 0 & \mbox{in} & \ B_R, \\
& \left\| \nabla u \right\|_{\underline{L}^2(B_R)}, \, 
\left\| \nabla v \right\|_{\underline{L}^2(B_R)} \leq \mathsf{M},
\end{aligned}
\right.
\end{equation*}
and every $ r \in \left[ \X , \tfrac 12 R \right]$, we have the estimate
\begin{equation*}
\left\| \nabla (u{-}v) \right\|_{\underline{L}^2(B_r)} +  \left( \frac{r}{R} {+} \frac1{r^{d{-}\sigma}} \right)^{-\alpha} \frac1r \inf_{\ell \in \mathcal{P}_1} \left\| u{-} v {-}\ell \right\|_{\underline{L}^2(B_r)}
\leq 
 \frac CR \left\| u {-} v {-} (u {-} v)_{B_R} \right\|_{\underline{L}^2(B_R)}. 
\end{equation*}
\end{proposition}

\begin{proof}
Fix $\sigma\in (0,d)$, we let $\X$ be the maximum of the random variables 
in the statements of Corollary~\ref{c.minimalscale}, Theorem~\ref{t.AS.regularity} and Proposition~\ref{p.differences.error}. Without loss of generality, we may assume that $(u - v)_{B_R} = 0$ and that $\X\geq H$ for a large constant $H(\sigma,\mathsf{M},\data)$ to be fixed in the course of the proof. Indeed, if $\X \leq H \leq R$, we have, for $s \in [\X,H]$, that 
\begin{equation*} 
\left\| \nabla (u-v) \right\|_{\underline{L}^2(B_s)} \leq \left(\frac{H}{s}\right)^{\frac d2} \left\| \nabla (u-v) \right\|_{\underline{L}^2(B_H)} \leq 
\frac{C}{R} \left\| u-v \right\|_{\underline{L}^2(B_R)},
\end{equation*}
and similarly for the other term. 
We begin by applying Theorem~\ref{t.AS.regularity} to obtain
\begin{equation}
\label{e.applyLip}
\sup_{r \in [\X,R] } \left( \frac1r  \left\| u - (u)_{B_r} \right\|_{\underline{L}^2(B_r)} + \frac1r \left\| v - (v)_{B_r} \right\|_{\underline{L}^2(B_r)} \right) 
\leq C.
\end{equation}
We can consequently apply~\eqref{e.differences.error} to get, for $r \in [\X,R]$, 
\begin{equation} 
\label{e.differences.errorapplied}
\left\| u - v - (\overline{u}_r - \overline{v}_r) \right\|_{\underline{L}^{2}(B_{r})}
\leq 
C r^{-\alpha(d-\sigma)} \left\| u-v - (u-v)_{B_{2r}} \right\|_{\underline{L}^2(B_{2r})}, 
\end{equation}
where $\overline{u}_r \in u + H_0^1(B_{r})$ and $\overline{v}_r \in v + H_0^1(B_{r})$ solve
\begin{equation}  \label{e.eq.bar u + bar v}
 -\nabla \cdot \left( D_p\overline{L} \left( \nabla\overline{u}_r \right)  \right) =  - \nabla \cdot \left( D_p \overline{L} \left( \nabla \overline{v}_r \right)  \right) = 0.
\end{equation}
Next, by Corollary~\ref{c.minimalscale}, there exist constants $C(\sigma,\data) < \infty$ and $\alpha_1(d,\Lambda) \in (0,1)$ such that
\begin{multline} \label{e.uandv.comparison}
\left\| u - \overline{u}_r \right\|_{\underline{L}^2(B_{r})} + \left\| v - \overline{v}_r  \right\|_{\underline{L}^2(B_{r})}  
\\
\leq 
C r^{-\alpha_1(d-\sigma)}  \left(\left\| u - (u)_{B_{2r}} \right\|_{\underline{L}^2(B_{2r})} + \left\| v - (v)_{B_{2r}} \right\|_{\underline{L}^2(B_{2r})}\right) \leq 
C r^{1-\alpha_1(d-\sigma)} 
.
\end{multline}
Combining~\eqref{e.applyLip} and the previous inequality we get
\begin{equation*} 
\sup_{r \in [\X,R] } \left( \frac1r  \left\| \overline{u}_r  - (\overline{u}_r )_{B_r} \right\|_{\underline{L}^2(B_r)} + \frac1r \left\| \overline{v}_r  - (\overline{v}_r )_{B_r} \right\|_{\underline{L}^2(B_r)} \right) \leq C .
\end{equation*}
Applying the second estimate of Theorem~\ref{t.AS.regularity}, we obtain, for every $r \in [\X,R]$, 
\begin{equation}  
\label{e.u.C11applied0}
\frac1{r} \inf_{\ell \in \mathcal{P}_1}  \| u - \ell \|_{\underline{L}^2(B_{r})}  + \frac1{r} \inf_{\ell \in \mathcal{P}_1}  \| v - \ell \|_{\underline{L}^2(B_{r})} \leq C \left(  \frac{r}{R}  + \frac{1}{r^{d-\sigma} }  \right)^{\alpha_1} , 
\end{equation}
and together with~\eqref{e.uandv.comparison} this leads to 
\begin{equation}  
\label{e.u.C11applied}
\frac1{r} \inf_{\ell \in \mathcal{P}_1}  \| \overline{u}_r  - \ell \|_{\underline{L}^2(B_{r})}  + \frac1{r} \inf_{\ell \in \mathcal{P}_1}  \| \overline{u}_r  - \ell \|_{\underline{L}^2(B_{r})} 
\leq 
C \left(  \frac{r}{R}  + \frac{1}{r^{d-\sigma} }  \right)^{\alpha_1} .
\end{equation}
We choose $\ep_0$ small enough and $H$ large enough, both depending on $(\delta,\theta,\mathsf{M},\data)$, so that $r \geq H$ implies 
\begin{equation*} 
C \left( \ep_0^\beta  + r^{-\alpha(d-\sigma)}  \right)  \leq 1. 
\end{equation*}
Therefore, by~\eqref{e.C1gamma.Lbardiff} and the previous two displays, we get, for $\theta \in (0,1]$ and $r \in [\X,\ep_0 R]$,  
\begin{align} 
\label{e.C1gamma.Lbardiffapplied}
\lefteqn{ 
\inf_{\ell \in \mathcal{P}_1}\left\| \overline{u}_r - \overline{v}_r - \ell \right\|_{\underline{L}^2(B_{\theta r})}
} \ \ & 
\\ & \notag
\leq C \theta^2 \inf_{\ell \in \mathcal{P}_1}\left\| \overline{u}_r -  \overline{v}_r - \ell \right\|_{\underline{L}^2(B_{r})} 
 + C \theta^{-\frac d2} \left(  \frac{r}{R}  + \frac{1}{r^{d-\sigma} }  \right)^{\alpha_1 \beta}  \left\| \overline{u}_r -  \overline{v}_r  - (\overline{u}_r -  \overline{v}_r)_{B_r}\right\|_{\underline{L}^2(B_r)}
\end{align}
and, by applying~\eqref{e.differences.errorapplied} once more, we also get 
\begin{align} 
\label{e.C1gamma.Lbardiffapplied1}
\lefteqn{ 
\inf_{\ell \in \mathcal{P}_1}\left\| u-v - \ell \right\|_{\underline{L}^2(B_{\theta r})}
} \qquad & 
\\ & \notag
\leq C \theta^2 \inf_{\ell \in \mathcal{P}_1}\left\| u-v - \ell \right\|_{\underline{L}^2(B_{r})} 
+ C \theta^{-\frac d2} 
\left(  \frac{r}{R}  + \frac{1}{r^{d-\sigma} }  \right)^{\alpha_1 \beta}  
\left\| u-v   - (u-v)_{B_r}\right\|_{\underline{L}^2(B_r)}.
\end{align}
Taking $\alpha := \frac{\alpha_1 \beta}{2d} $ and setting 
\begin{equation*} 
E_1(r) :=  \left(  \frac{r}{R}  + \frac{1}{r^{d-\sigma} }  \right)^{ - \alpha} 
\frac1r \inf_{\ell \in \mathcal{P}_1}\left\| u-v - \ell \right\|_{\underline{L}^2(B_{r})} 
\end{equation*}
and
\begin{equation*} 
D_1(r) := \sup_{s \in [r,\ep_1 R]} \frac1s \left\| u-v   - (u-v)_{B_s}\right\|_{\underline{L}^2(B_s)},
\end{equation*}
the previous inequality implies that, for every~$r \leq s \leq \ep_0 R$, 
\begin{equation*} 
E_1(\theta s) \leq C\theta^{1/2} E_1(s) + C \theta^{-2-\frac d2}  \left(  \frac{s}{R}  + \frac{1}{s^{d-\sigma} }  \right)^{ \alpha}     D_1(r). 
\end{equation*}
Choosing $\theta := (2C)^{-2}$, $\ep_1 \in (0,\ep_0]$ and $k\in \N$ such that $\theta^{k+1} \ep_1 R < r \leq \theta^{k} \ep_1 R$, an iteration argument yields 
\begin{equation}  \label{e.Eonesum}
\sum_{j=0}^k E_1(\theta^{-j} r) \leq E_1(\ep_1 R) + C\left(  \ep_1  + H^{-(d-\sigma)}   \right)^{\alpha}  D_1(r) . 
\end{equation}
Letting $\ell_j$ be the affine function realizing the infimum in $E_1(\theta^{-j} r)$, we see that 
\begin{equation*} 
\left| \nabla \ell_{0} - \nabla \ell_{k}   \right| \leq C \sum_{j=0}^k E_1(\theta^{-j} r)  \leq C E_1(\ep_1 R) + C \left(  \ep_1+ H^{-(d-\sigma)} \right)^\alpha D_1(r).
\end{equation*}
Therefore,
\begin{align*} 
D_1(r) 
&
\leq  E_1(r)  +C \left| \nabla \ell_{0} \right|
\leq C E_1(\ep_1 R)  + C \left| \nabla \ell_{k} \right| + C \left(  \ep_1+ H^{-(d-\sigma)} \right)^\alpha D_1(r).
\end{align*}
By choosing $\ep_1$ small enough and $H$ large enough, so that
\begin{equation*} 
C \left(  \ep_1+ H^{-(d-\sigma)} \right)^\alpha   \leq \frac12,
\end{equation*}
we deduce, after reabsorption, that
\begin{equation*} 
 \sup_{s \in (r,\ep_1 R]} \frac1s \left\| u-v   - (u-v)_{B_s}\right\|_{\underline{L}^2(B_s)} \leq C E_1(\ep_1 R)  + C \left| \nabla \ell_{k} \right|.
\end{equation*}
We finally obtain that 
\begin{equation*} 
 E_1(\ep_1 R)  +  \left| \nabla \ell_{k} \right| \leq \frac CR \left\| u-v   - (u-v)_{B_R}\right\|_{\underline{L}^2(B_R)}, 
\end{equation*}
and thus
\begin{equation} \label{e.Lipdiffpre000000}
 \sup_{r \in [\X, R]} \frac1r \left\| u-v   - (u-v)_{B_s}\right\|_{\underline{L}^2(B_r )} \leq \frac CR \left\| u-v   - (u-v)_{B_R}\right\|_{\underline{L}^2(B_R)}.
\end{equation}
This completes the proof of the first part of the statement, in view of the Caccioppoli inequality. The second statement is equivalent to 
\begin{equation*} 
\sup_{r \in [\X,R] } E_1(r) \leq \frac CR \left\| u-v   - (u-v)_{B_R}\right\|_{\underline{L}^2(B_R)}, 
\end{equation*}
which also follows easily from~\eqref{e.Eonesum} and~\eqref{e.Lipdiffpre000000}.
\end{proof}

We conclude this section by recording a second large-scale regularity estimate, this one for solutions of the linearized equation. The proof is omitted since it is essentially the same as the one of Theorem~\ref{t.regularity.differences}, except that it is simpler because in place of Proposition~\ref{p.differences.error} we just use a rescaling of Theorem~\ref{t.linearization}.

\begin{proposition}[{Large-scale $C^{0,1}$ estimate for linearized equations}]
\label{p.C01.linearizedeq}
Fix $\sigma \in (0,d)$ and $\mathsf{M}\in [1,\infty)$. There exist constants $\alpha(\data) \in \left(0,\tfrac12\right]$ and~$C(\sigma,\mathsf{M},\data) <\infty$ and a random variable~$\X$ satisfying
\begin{equation*}
\X \leq \O_\sigma\left( C \right)
\end{equation*}
such that the following holds. For every $R\geq 2\X$ and $u,w\in H^1(B_R)$ satisfying 
\begin{equation*}
\left\{
\begin{aligned}
& -\nabla \cdot \left( D_pL(\nabla u,x) \right) = 0 & \mbox{in} & \ B_R, \\
& -\nabla \cdot \left( D_p^2L(\nabla u,x) \nabla w\right) = 0 & \mbox{in} & \ B_R, \\
& \left\| \nabla u \right\|_{\underline{L}^2(B_R)}  
\leq \mathsf{M},
\end{aligned}
\right.
\end{equation*}
we have the estimate, for every $r \in \left[ \X , \tfrac 12 R \right]$, 
\begin{equation*}
 \left( \frac{r}{R} + \frac1{r^{d-\sigma}} \right)^{-\alpha} \inf_{\ell \in \mathcal{P}_1} \frac{1}{r} \left\|  w- \ell   \right\|_{\underline{L}^2(B_{r})}  + 
\left\| \nabla w \right\|_{\underline{L}^2(B_r)}
\leq 
 \frac CR \left\| w - (w)_{B_R}\right\|_{\underline{L}^2(B_R)}. 
\end{equation*}
\end{proposition}

The estimate of the first term in the inequality above represents a ``$C^{1,\alpha}$ estimate down to mesoscopic scales.'' The analogous estimate for differences was stated in~Proposition~\ref{p.Conealphafordifferences}.


\section{Large-scale \texorpdfstring{$C^{1,1}$}{C11}-type regularity: proof of Theorem~\ref{t.C11estimate}}
\label{s.C11estimate}

We turn to the proof of Theorem~\ref{t.C11estimate}. In the first subsection, we prove statements (i) and (ii) of the theorem, which together are the Liouville-type results characterizing the set~$\mathcal{L}_1$ of solutions exhibiting at most linear growth. In Subsection~\ref{ss.linearizecorrector}, we develop a large-scale regularity theory for a linearized equation \emph{around an element of~$\mathcal{L}_1$} and show quantitatively that the first-order linearized correctors describe the tangent spaces of the ``manifold''~$\mathcal{L}_1$. Finally, in Subsection~\ref{ss.Ck1conc} we complete the proof of Theorem~\ref{t.C11estimate} by obtaining the large-scale $C^{1,1}$--type excess decay estimate.

\subsection{The Liouville-type results}
In this subsection we present the proof of Theorem~\ref{t.C11estimate}~(i) and~(ii), namely the Liouville-type classification of~$\mathcal{L}_1$. After the proof~(i) and~(ii), we will prove a similar result for difference of two solutions in~$\mathcal{L}_1$.

\begin{proof}
[{Proof of Theorem~\ref{t.C11estimate}(i)--(ii)}]
Fix $\sigma \in (0,d)$ and $\mathsf{M} \in [1,\infty)$. We set $r_0 := \X $, where~$\X$ is the maximum of minimal scales in Theorem~\ref{t.regularity.differences},
Corollary~\ref{c.minimalscale} and Theorem~\ref{t.AS.regularity}, and a constant $H(\sigma,\mathsf{M},\data)<\infty$ to be fixed.  

\smallskip

\emph{Step 1.}  The proof of statement (ii) of Theorem~\ref{t.C11estimate}. Fix $p\in B_{\mathsf{M}}$ and let $\ell(x):=p\cdot x$. Let $r_m := 2^m r_0$. 
Let $v_{m}$ be the solution of the Dirichlet problem
\begin{equation*}
\left\{
\begin{aligned}
& -\nabla \cdot \left( D_pL(\nabla v_{m},x) \right) = 0 & \mbox{in} & \ B_{r_m}, \\
& v_{m}= \ell & \mbox{on } & \partial B_{r_{m}}.
\end{aligned}
\right.
\end{equation*}
Set $w_m := v_{m+1} - v_{m}$. 
As $|\nabla \ell| \leq \mathsf{M}$, we may apply Corollary~\ref{c.minimalscale} to obtain, for every~$m \in \N$, 
\begin{equation} \label{e.regularity1.almostthere0}
\left\| v_{m} - \ell \right\|_{\underline{L}^{2}(B_{r_m})}\leq Cr_{m}^{1-\beta} 
\end{equation}
with $\beta = \alpha (d-\sigma)$.  The triangle inequality thus gives us, for all~$m \in \N$,
\begin{equation} \label{e.regularity1.almostthere000}
\left\|w_m \right\|_{\underline{L}^{2}(B_{r_m})} \leq Cr_{m}^{1-\beta} .
\end{equation}
Then, by the Lipschitz regularity for differences of solutions in Theorem~\ref{t.regularity.differences}, we have that, for $j\leq m$,
\begin{equation} \label{e.regularity1.almostthere00}
\left\| \nabla w_{m} \right\|_{\underline{L}^{2}(B_{r_j})}
\leq
 \frac{C}{r_{m}} \left\|w_{m} \right\|_{\underline{L}^{2}(B_{r_m})} \leq C r_{m}^{-\beta} .
\end{equation}
We  also get by the above inequality and Poincar\'e's inequality that, for $j \leq m$, 
\begin{equation}  \label{e.wtelescopingscales}
\left| (w_{m})_{B_{r_j}} - (w_{m})_{B_{r_0}}  \right| 
\leq 
C \sum_{i=1}^{j} \left\| w_{m} - (w_{m})_{B_{r_i}} \right\|_{\underline{L}^2 \left( B_{r_i} \right)}
\leq 
C r_j r_{m}^{-\beta} .
\end{equation}
Using the above inequality we obtain, again by Poincar\'e's inequality, that for all $k \in \N$ and $j \in \{0,\ldots,k\}$, 
\begin{equation} \label{e.regularity1.almostthere1}
\left\| w_{k} - (w_{k})_{B_{r_0}}  \right\|_{\underline{L}^{2}(B_{r_j})} \leq C r_j r_{k}^{-\beta} .
\end{equation}
Set now
\begin{equation*} 
u_n := v_n  - (v_n)_{B_{r_0}} ,
\end{equation*}
so that, noticing that $$u_n - v_m =  \sum_{j=m}^{n-1}\left(w_j -  (w_j)_{B_{r_0}} \right) +    \sum_{j=0}^{m-1} (w_j)_{B_{r_0}}, $$
we have by the triangle inequality that
\begin{equation*} 
\left\| u_{n} - \ell  \right\|_{\underline{L}^2\left( B_{r_{m}} \right)} \leq \left\| v_m -\ell  \right\|_{\underline{L}^2\left( B_{r_{m}} \right)}
+ \sum_{j=m}^{n-1} \left\| w_j - (w_j)_{B_{r_0}}  \right\|_{\underline{L}^2\left( B_{r_{m}} \right)} + \sum_{j=0}^{m-1} \left| (w_j)_{B_{r_0}}  \right| .
\end{equation*}
The first two terms on the right can be estimated using~\eqref{e.regularity1.almostthere0} and~\eqref{e.regularity1.almostthere1}, respectively. For the last term we instead use~\eqref{e.regularity1.almostthere000} and~\eqref{e.wtelescopingscales} to obtain
\begin{equation*} 
 \sum_{j=0}^{m-1} \left| (w_j)_{B_{r_0}}  \right|
 \leq  \sum_{j=0}^{m-1} \left\| w_j \right\|_{\underline{L}^2 \left( B_{r_j} \right)} + \sum_{j=0}^{m} \left| (w_j)_{B_{r_0}} - (w_j)_{B_{r_j}}  \right|
  \leq C r_m^{1-\beta}  .
\end{equation*}
Therefore, we have shown that, for $n  \in \N$ and $m \in \{0,\ldots,n\}$,
\begin{equation*} 
\left\| u_{n} - \ell  \right\|_{\underline{L}^2\left( B_{r_{m}} \right)} \leq C r_m^{1-\beta}  .
\end{equation*}
Having arrived at this estimate, we now observe that, for every $j,m,n\in\mathbb{N}$ with $j \leq m<n$, 
we have that $\nabla u_{n} - \nabla u_{m} = \sum_{i = m}^{n-1} \nabla w_{i}$, and hence by~\eqref{e.regularity1.almostthere00}  we get
\begin{equation} \label{e.regularity1.almostthere3}
\left\| \nabla u_{n} - \nabla u_m  \right\|_{\underline{L}^{2}(B_{r_j})} \leq
C2^{-\beta(m-j)} r_{j}^{-\beta} .
\end{equation}
Thus $\{u_n\}_{n=j}^\infty$ is a Cauchy sequence in $H^1(B_{r_j})$ and, consequently, we deduce by a diagonal argument that there exists $u \in \mathcal{L}_1$ such that, for all $j \in \N$,  
\begin{equation} \label{e.regularity1.almostthere4}
\left\| u - \ell  \right\|_{\underline{L}^{2}(B_{r_j})} \leq
C r_{j}^{1-\beta} . 
\end{equation}
This yields~\eqref{e.liouvillec1}.

\smallskip

\emph{Step 2.} The proof of statement (i) of Theorem~\ref{t.C11estimate}.
Fix $u\in\mathcal{L}_{1}$ satisfying
\begin{equation} 
\label{e.regularity1norm}
\limsup_{r \to \infty} \frac1r \left\|  u - (u)_{B_r} \right\|_{\underline{L}^2(B_r)}   \leq \mathsf{M}.
\end{equation}
By the Lipschitz estimate and~\eqref{e.regularity1norm} we get 
 \begin{equation} \label{e.regularity1conv1}
\sup_{r \in [\X,R]} \frac1r \left\| u - (u)_{B_{r}} \right\|_{\underline{L}^2(B_{r})} \leq  C \limsup_{r \to \infty} \frac 1r \left\| u - (u)_{B_{r}} \right\|_{\underline{L}^2(B_{r})}  \leq C . 
\end{equation}
Denoting 
\begin{equation*} 
E(r) := \frac1r  \inf_{\ell \in \mathcal{P}_1}  \| u - \ell \|_{\underline{L}^2(B_r)},
\end{equation*}
 we get by applying Theorem~\ref{t.AS.regularity} that, for $r \in [\X, R]$, 
\begin{equation}  
\label{e.regularity1conv2}
E(\theta r) \leq C \theta^\alpha E(r) +  C\theta^{-\frac{d+2}{2}} r^{-\beta} ,
\end{equation}
where $\beta = \alpha(d-\sigma)$.  Denote by $\ell_{r}$ be the affine function minimizing $E(r)$. 
Notice that, by~\eqref{e.regularity1conv1}, $|\nabla \ell_r | \leq C$ for all $r \geq \X$. 
Taking  $\theta \in (0, 1)$ so small that $C \theta^\alpha \leq \frac12$
and setting $r_{m} := \theta^{-m} r_{0}$, we obtain, for all $m \in \N$, 
\begin{equation}  
\label{e.regularity1conv4}
E(r_{m-1}) \leq \frac12 E(r_{m}) +  Cr_{m}^{-\beta} .
\end{equation}
It follows, after reabsorption,  that 
\begin{equation*} 
\sum_{j=m}^{n-1} E(r_{j}) \leq E(r_{n})  + C  \sum_{j=m+1}^{n}  r_{j}^{-\beta}  \leq E(r_{n})  + C r_{m}^{-\beta} . 
\end{equation*}
The term on the right is uniformly bounded in $n$, because~\eqref{e.regularity1conv1} implies that
\begin{equation} \label{e.regularity1conv5}
\limsup_{n \to \infty}  E(r_{n})   \leq C. 
\end{equation}
The above two displays give that $E(r_{n}) \to 0$ as $n \to \infty$ and, thus
\begin{equation} \label{e.regularity1conv6}
\sum_{j=m}^{\infty} E(r_{j}) \leq C  r_{m}^{-\beta}  . 
\end{equation}
Furthermore, since
\begin{equation*} 
\left| \nabla \ell_{r_j} - \nabla \ell_{r_{j+1}} \right| \leq C \left( E(r_{j}) + E(r_{j+1}) \right),
\end{equation*}
we see by telescoping that, for all $m,n \in \N$ such that $n>m$, 
\begin{equation*} 
\left| \nabla \ell_{r_m} - \nabla \ell_{r_{n}} \right| \leq C  r_{m}^{-\beta} ,
\end{equation*}
and the right-hand side converges to zero as $m \to \infty$.  Therefore $\{\nabla \ell_{r_{n}} \}_{n}$ is a Cauchy sequence and consequently there is an affine function $\ell$ such that 
\begin{equation*} 
 \lim_{n \to \infty} \nabla \ell_{r_{n}} = \nabla \ell.
\end{equation*}
We hence obtain that 
\begin{equation*} 
\left| \nabla \ell_{r_m} - \nabla \ell \right| \leq C r_{m}^{-\beta} .
\end{equation*}
Plugging this into~\eqref{e.regularity1conv6} proves, after easy manipulations, for $\tilde \ell(x) := \nabla \ell \cdot x$ and $r\in [\X, \infty)$, 
\begin{equation} \label{e.regularity1conv7} 
 \left\| u - (u)_{B_{r}} - \tilde \ell \right\|_{\underline{L}^2 \left( B_{r} \right)} \leq C r^{1-\beta} .
\end{equation}
Furthermore, by Step 1 there is $\tilde u \in \mathcal{L}_{1}$ corresponding   $ \tilde \ell$ such that, for $r \geq \X$,
\begin{equation} \label{e.regularity1conv8}
 \left\| \tilde u - \tilde \ell \right\|_{\underline{L}^2\left( B_{r} \right)} \leq C r^{1-\beta} .
\end{equation}
Together with~\eqref{e.regularity1conv7} this yields 
\begin{equation}  \label{e.regularity1conv9}
\left\| u - \tilde u  - (u - \tilde u)_{B_{r_j}} \right\|_{\underline{L}^2\left( B_{r_j} \right)} \leq C r_j^{1-\delta} .
\end{equation}
Thus, applying the Lipschitz estimate for differences we see that, for all $r \geq \X$, 
\begin{equation*} 
\frac1r \left\| u - \tilde u  - (u - \tilde u)_{B_r} \right\|_{\underline{L}^2\left( B_{r} \right)} \leq C \limsup_{r\to \infty} \frac1r \left\| u - \tilde u  - (u - \tilde u)_{B_r} \right\|_{\underline{L}^2\left( B_{r} \right)} = 0 .
\end{equation*}
Therefore $u - \tilde u$ is a constant, and the statement~\eqref{e.liouvillec0} follows from~\eqref{e.regularity1conv8}  by setting $\ell = u - \tilde u + \tilde \ell$. The proof is complete. 
\end{proof}

We conclude this subsection with a quantitative Liouville-type result for differences of solutions. 

\begin{lemma}
\label{l.regularity2}
Fix $\sigma \in (0,d)$ and $\mathsf{M} \in [1,\infty)$. There exists constants $C(\sigma,\mathsf{M},\data)<\infty$ and $\alpha(\data) \in (0,1)$, and a random variable $\X$ satisfying 
\begin{equation}
\label{e.Xintgrabb}
\X \leq \O_\sigma(C). 
\end{equation}
such that the following holds.  Suppose that $ u,v \in \mathcal{L}_1$ satisfy
\begin{equation} \label{e.C11uvcond}
\limsup_{r \to \infty} \frac1r \left( \left\| u - (u)_{B_r} \right\|_{\underline{L}^2(B_{r})} \vee \left\| v - (v)_{B_r} \right\|_{\underline{L}^2(B_{r})} \right)
\leq
\mathsf{M}. 
\end{equation}
Then there is an affine function $\ell$ such that, for $r \geq \X$, 
\begin{equation} \label{e.regularity2res1}
\left\| u - v   - \ell \right\|_{\underline{L}^2\left( B_r \right)} 
\leq 
C r^{-\alpha(d-\sigma)} 
\left\|  \ell \right\|_{\underline{L}^2\left( B_r \right)}
\end{equation}
and
\begin{equation} 
\label{e.diffdoubling}
\sup_{r \in [\X,\infty)} \left\| \nabla u - \nabla v \right\|_{\underline{L}^2\left( B_r \right)} 
\leq 
C \inf_{r \in [\X,\infty)} \left\| \nabla u - \nabla v \right\|_{\underline{L}^2\left( B_r \right)}  .
\end{equation}
\end{lemma}

\begin{proof}
Let $\X$ be the maximum of the minimal scales in Theorem~\ref{t.C11estimate}(i)--(ii), Corollary~\ref{c.minimalscale}, Proposition~\ref{p.differences.error}, and a constant $H(\sigma,\mathsf{M},\data)<\infty$ to be selected below. Then~\eqref{e.Xintgrabb} holds. Set $z := u-v$ and denote
\begin{equation*} 
D(r) := \frac1r \inf_{\ell \in \mathcal{P}_1}\left\| u - \ell \right\|_{\underline{L}^2\left( B_r \right)} \vee \inf_{\ell \in \mathcal{P}_1} \left\| v- \ell \right\|_{\underline{L}^2\left( B_r \right)}
\end{equation*}
and
\begin{equation*} 
E(r) := \frac1r \inf_{\ell \in \mathcal{P}_1}\left\| z -  \ell \right\|_{\underline{L}^2\left( B_r \right)} .
\end{equation*}
Let $\overline{u}_r$, $\overline{v}_r$, and $\overline{z}_r = \overline{u}_r - \overline{v}_r$  stand for the homogenized solutions at the scale $r$ provided by Proposition~\ref{p.differences.error} applied at the scale $R = r$. In particular, we have that 
\begin{equation}  \label{e.differences.error.applied}
\left\|  z - \overline{z}_r  \right\|_{\underline{L}^2 \left( B_{r} \right)} \leq C r^{-\beta} \left\|   z - (z)_{B_r}   \right\|_{\underline{L}^2 \left( B_{r} \right)}.
\end{equation}
 Next, Theorem~\ref{t.C11estimate}(i) implies, by taking $H$ so large that $C  H^{-\frac12 \alpha(d-s)} \leq 1$,  that there exists~$\beta:=\frac12 \alpha(d-s)$ such that  
\begin{equation} \label{e.C11Ersmall}
D(r) \leq \frac12 r^{-\beta}. 
\end{equation}
Corollary~\ref{c.minimalscale} then yields, by taking $H$ larger and $\beta$ smaller, if necessary,  that
\begin{equation*} 
\overline{D}(r) :=  \frac1r \inf_{\ell \in \mathcal{P}_1}\left\| \overline{u}_r - \ell \right\|_{\underline{L}^2\left( B_r \right)} \vee \inf_{\ell \in \mathcal{P}_1}\left\| \overline{v}_r - \ell \right\|_{\underline{L}^2\left( B_r \right)}\leq  r^{-\beta}. 
\end{equation*}
The basic decay estimate for the differences of homogenized solutions is given in Proposition~\ref{p.C1beta.differences.Lbar}: it states that
\begin{align*}
\lefteqn{ \inf_{\ell \in \mathcal{P}_1}\left\| \overline{z}_r - \ell \right\|_{\underline{L}^2(B_{\theta r})}
} \quad & 
\\ & \notag
\leq C \theta^2 \inf_{\ell \in \mathcal{P}_1}\left\| \overline{z}_r - \ell \right\|_{\underline{L}^2(B_{r/2})} 
+ C \left( \overline{D}(r) \vee 1 \right)^{\eta}  \theta^{-\frac d2}  \overline{D}(r)^\gamma \left\| \overline{z}_r  - (\overline{z}_r)_{B_r}\right\|_{\underline{L}^2(B_{r/2})}.
\end{align*}
Using~\eqref{e.differences.error.applied} for the differences together with the triangle inequality, yields, again for smaller $\beta$, 
\begin{equation*} 
E(\theta r)
\leq C \left( \theta  +  \theta^{- \frac {d+2}{2}} H^{-\beta} \right) E(r) + C \theta^{- \frac {d+2}{2}} r^{-\beta} |\nabla \ell_r| .
\end{equation*}
Comparing this to~\eqref{e.regularity1conv2} allows us to reason exactly as in the proof of Theorem~\ref{t.C11estimate}(i) to obtain an affine function $\ell$ such that 
\begin{equation} \label{e.regularity21} 
 \frac1r \left\| z - (z)_{B_{r}} -  \ell \right\|_{\underline{L}^2 \left( B_{r} \right)} \leq C r^{1-\beta} \left|  \nabla \ell \right|,
\end{equation}
giving~\eqref{e.regularity2res1}. The above display also yields,  for a positive constant $A(d)$, that, for $r\geq \X \geq H$, 
\begin{equation*} 
 \left| \frac1r \left\| z - (z)_{B_{r}}  \right\|_{\underline{L}^2 \left( B_{r} \right)} - A \left| \nabla \ell \right|  \right| \leq
C H^{-\beta} |\nabla \ell|  
.
\end{equation*}
This gives~\eqref{e.diffdoubling}  by Caccioppoli and Poincar\'e's inequalities provided that we take a lower bound $H$ for $\X$ such that $ 2C H^{-\beta}  \leq A$. The proof is complete.
\end{proof}


\subsection{Linearized equations around first-order correctors}
\label{ss.linearizecorrector}

In this subsection, we establish some estimates for the linearized equations \emph{around first-order correctors}, namely the full slate of large-scale $C^{k,1}$ estimates as well as a quadratic estimate for linearization errors. The former is essentially a verbatim consequence of the arguments of~\cite[Chapter 3]{AKMbook} combined with homogenization results already obtained in this paper. The only work to be done is to slightly post-process the statements in~\cite{AKMbook} to be compatible with the statement of the result we need.

\smallskip

To state the next theorem, we introduce the \emph{linearized correctors}. Let $\phi \in \mathcal{L}_1$. 
By Theorem~\ref{t.C11estimate}(i), we have that there is a unique affine function $\ell_\phi$ such that
\begin{equation}  \label{e.snaptoaffine}
\lim_{r \to \infty} \frac1r \left\|  \phi - \ell_\phi \right\|_{\underline{L}^2 \left( B_{r} \right)} = 0. 
\end{equation}
For an open set $U\subseteq\Rd$, we denote 
\begin{equation*} 
\A[\phi](U) := 
\left\{ w \in H_{\textrm{loc}}^1(U) \, : \, - \nabla \cdot \left( D_p^2L(\nabla \phi,\cdot) \nabla w \right) = 0 \mbox{ in } U \right\}.
\end{equation*}
We next define, for each $k\in\N$, 
\begin{equation*} 
\A_k[\phi] := \left\{ w \in  \A_k[\phi](\R^d) \, : \,  \lim_{r\to \infty} r^{-k-1} \left\| w \right\|_{\underline{L}^2 \left( B_{r} \right)} = 0 \right\}  
\end{equation*}
and denote the $D^2_p\overline{L}\left( \nabla \ell_\phi \right)$--harmonic polynomials of degree~$k$ by
\begin{equation*} 
\Ahom_k[\phi] := \left\{ q \in \mathcal{P}_k \, : \, - \nabla \cdot \left( D_p^2 \overline{L}(\nabla \ell_\phi) \nabla q \right) = 0  \right\}.
\end{equation*}
The next theorem, which can be compared to~\cite[Theorem 3.8]{AKMbook}, asserts that the vector space~$\A_k[\phi]$ essentially homogenizes to $\Ahom_k[\phi]$ and, in particular, has the same dimension. It also gives a large-scale $C^{k,1}$ estimate which states that solutions of the linearized equation (around $\phi$) can be approximated by elements of~$\A_k[\phi]$ in the same way that a harmonic function can be approximated by a polynomial of degree~$k$ (or an analytic function by its $k$th order Taylor series).

\begin{theorem}[$C^{k,1}$ regularity for linearized equation around an element of $\mathcal{L}_1$] 
\label{t.regularitylin}
\emph{}\\
Fix $\sigma \in (0,d)$ and $\mathsf{M} \in [1,\infty)$. 
There exist~$\delta(\sigma,\data)\in \left( 0, \frac12 \right]$,~$C(\data)<\infty$ and a random variable $\X$ satisfying the estimate
\begin{equation}
\label{e.X1}
\X \leq \O_\sigma\left(C(\sigma,\mathsf{M},\data)\right)
\end{equation}
such that, for every $\phi  \in  \mathcal{L}_1$ satisfying $\limsup_{r \to \infty} \frac1r \left\| \phi - (\phi)_{B_r} \right\|_{\underline{L}^2 \left( B_{r} \right)} \leq \mathsf{M}$ and~$k\in\N$, the following statements are valid:
\begin{enumerate}
\item[{$\mathrm{(i)}_k$}] There exists $C(k,\mathsf{M},\data)<\infty$ such that, for every $w \in \A_k[ \phi ]$, there exists $q\in \overline{\A}_k[ \phi ]$ such that, for every $r \geq \X$,
\begin{equation} 
\label{e.liouvillec}
\frac1r \left\| w - q \right\|_{\underline{L}^2(B_r)} 
\leq 
C r^{-\delta} \left\| q \right\|_{\underline{L}^2(B_r)} .
\end{equation}

\bigskip 

\item[{$\mathrm{(ii)}_k$}]For every $p\in \overline{\A}_k[\phi]$, there exists $w \in \A_k[\phi]$ satisfying~\eqref{e.liouvillec} for every $r \geq \X$. 

\bigskip 

\item[{$\mathrm{(iii)}_k$}]
There exists $C(k,\mathsf{M},\data)<\infty$ such that, for every $R\geq \X$ and $w \in \A[\phi](B_R)$, there exists $\xi \in \A_k[\phi]$ such that, for every $r \in \left[ \X,  R \right]$, we have the estimate
\begin{equation}
\label{e.intrinsicreg2}
\left\| w - \xi \right\|_{\underline{L}^2(B_r)} \leq C \left( \frac r R \right)^{k+1} \left\| w \right\|_{\underline{L}^2(B_R)}.
\end{equation}
\end{enumerate}
\end{theorem}

\begin{proof}
Fix $\sigma \in (0,d)$ and $\phi \in \mathcal{L}_1$. Let $\X$ be the maximum of minimal scales appearing in Theorem~\ref{t.AS.regularity}, Corollary~\ref{c.linearization} and Theorem~\ref{t.C11estimate}(i)--(ii), which we already proved. Then $\X_\sigma \leq O_\sigma(C)$. Fix $r \geq 2 \X$. By Theorem~\ref{t.C11estimate} (i) we have that there is an affine function $\ell_\phi$ such that
\begin{equation*}  \label{e.phitoaffine0000}
\left\| \phi - \ell_\phi  \right\|_{\underline{L}^2(B_r)} \leq C r^{1-\alpha(d-\sigma)} .
\end{equation*}
Let~$w \in \A[\phi](B_r)$. 
Consider the following system of homogenized solutions:
\begin{equation*}
\left\{
\begin{aligned}
& -\nabla \cdot \left( D_p \overline{L}(\nabla \overline{u}_{r}) \right) = 0 & \mbox{in} & \ B_{r}, \\
& -\nabla \cdot \left( D_p^2 \overline{L}(\nabla \overline{u}_{r}) \nabla \overline{w}_r \right) = 0 & \mbox{in} & \ B_{4r/5}, \\
& -\nabla \cdot \left( D_p^2 \overline{L}(\nabla \ell_\phi) \nabla \overline{w}  \right) = 0 & \mbox{in} & \ B_{3r/4} ,\\
&  \overline{u}_{r} = \phi & \mbox{on } & \partial B_{r} ,\\ 
&  \overline{w}_{r} = w & \mbox{on } & \partial B_{4r/5} , \\ 
&  \overline{w} = \overline{w}_{r} & \mbox{on } & \partial B_{3r/4} .
\end{aligned}
\right.
\end{equation*}
We claim that 
\begin{equation}  \label{e.harmapprox000}
\left\| \overline{w} - w  \right\|_{\underline{L}^2 \left( B_{r/2} \right)} \leq C r^{-\alpha(d-\sigma)} \left\| w -  (w)_{B_r}  \right\|_{\underline{L}^2 \left( B_{r} \right)} .
\end{equation}
This estimate, which asserts that $w$ can be well-approximated by a $D_p^2 \overline{L}(\nabla \ell_\phi)$--harmonic function, is analogous to~\cite[Lemma 3.4]{AKMbook}. Indeed, using this estimate as a replacement for~\cite[Lemma 3.4]{AKMbook}, we may repeat, essentially verbatim, the proof of~\cite[Theorem 3.8]{AKMbook} to obtain the theorem.

\smallskip

To show~\eqref{e.harmapprox000}, we first obtain by Corollary~\ref{c.linearization} and Meyers estimate that
\begin{equation} \label{e.harmapprox001}
\left\| w -  \overline{w}_r  \right\|_{\underline{L}^2 \left( B_{4r/5} \right)} \leq C r^{-\alpha(d-\sigma)} \left\| w -  (w)_{B_r}  \right\|_{\underline{L}^2 \left( B_{r} \right)} .
\end{equation}
Second, by  Theorem~\ref{t.AS.regularity} we have that
\begin{equation*} 
  \left\| \phi - \overline{u}_{r}   \right\|_{\underline{L}^2(B_r)} \leq C r^{1-\alpha(d-\sigma)}. 
\end{equation*}
Thus, in view of~\eqref{e.phitoaffine0000}, we get by the triangle inequality that 
\begin{equation*} 
  \left\|  \overline{u}_{r} - \ell_\phi    \right\|_{\underline{L}^2(B_r)} \leq C r^{1-\alpha(d-\sigma)}. 
\end{equation*}
Then Caccioppoli inequality yields that 
\begin{equation*} 
  \left\| \nabla \overline{u}_{r} - \nabla \ell_\phi  \right\|_{L^2(B_{3r/4})} \leq \frac{C}{r} \left\|  \overline{u}_{r}  - \ell_\phi \right\|_{\underline{L}^2(B_r)} \leq C r^{-\alpha(d-\sigma)}
\end{equation*}
and, by Proposition~\ref{p.C1beta.differences.Lbar} we arrive at 
\begin{equation*} 
  \left\| \nabla \overline{u}_{r} - \nabla \ell_\phi  \right\|_{L^\infty(B_{r/2})} \leq C r^{-\alpha(d-\sigma)}.
\end{equation*}
We deduce, by the equations, 
\begin{equation*} 
-\nabla \cdot \left( D_p^2 \overline{L}(\nabla \ell_\phi) \nabla (\overline{w}_{r} -  \overline{w}) \right) = 
\nabla \cdot \left( \left( D_p^2 \overline{L}(\nabla \overline{u}_{r}) -   D_p^2 \overline{L}(\nabla \ell_\phi) \right) \nabla \overline{w}_{r}  \right) ,
\end{equation*}
and hence there is possibly smaller $\alpha(\data) \in \left(0,\frac12 \right]$ such that
\begin{equation*} 
\left\| \nabla  \overline{w}_{r} - \nabla  \overline{w}    \right\|_{\underline{L}^2 \left( B_{3r/4} \right)} \leq C r^{-\alpha(d-\sigma)-1} 
 \left\| \overline{w}_{r}  - (\overline{w}_{r} )  \right\|_{\underline{L}^2 \left( B_{4r/5} \right)} .
\end{equation*}
We now obtain~\eqref{e.harmapprox000} by the triangle inequality, Poincar\'e's inequality, Caccioppoli estimate and~\eqref{e.harmapprox001}, concluding the proof. 
\end{proof}

We next present a result concerning the superlinear response to the linear approximation of~$\L_1$ by linearized correctors in~$\A_1$. This states roughly that $\A_1[\phi]$ is the tangent space of $\L_1$ at $\phi$ and that the ``manifold'' $\L_1$ has a structure that is at least $C^{1,\delta}$ for some small $\delta>0$. It is important for our purposes that these estimates hold ``all the way down to the microscopic scale,'' quantified by a minimal scale~$\X$ with optimal stochastic integrability. This result is of independent interest. 

\begin{lemma}[Superlinear response to linearization for $\L_1$]
\label{l.tildephivsphi1} 
Fix $\sigma \in (0,d)$ and $\mathsf{M} \in [1,\infty)$. There exist constants $\delta(d,\Lambda) \in \left(0,\frac12 \right]$ and~$C(\sigma,\mathsf{M},\data)<\infty$, and a random variable $\X$ satisfying $\X = \O_\sigma(C)$ such that the following two statements are valid: 

\begin{enumerate}
\item[(i)]
For every $u,v \in \mathcal{L}_1$ satisfying
\begin{equation} 
\label{e.C11phitildephicond2}
\limsup_{r \to \infty} \frac1r \left( \left\| u - (u)_{B_r} \right\|_{\underline{L}^2(B_{r})} \vee \left\|  v -  (v)_{B_r}  \right\|_{\underline{L}^2(B_{r})} \right) \leq \mathsf{M},
\end{equation}
there exists $w \in \A_1[u]$ such that, for every $r \geq \X$,  
\begin{equation}
\label{e.tildephivsphi2}
\left\| \nabla v - \nabla u  -  \nabla w \right\|_{\underline{L}^2(B_r)} \leq C \left( \left\| \nabla u  - \nabla v \right\|_{\underline{L}^2(B_r)} \right)^{1+\delta}.
\end{equation}

\item[(ii)]
For every $u \in \mathcal{L}_1$ and $w \in \A_1[u]$
satisfying
\begin{equation} 
\label{e.C11phitildephicond}
\limsup_{R \to \infty}
\frac1r 
\left\| u - (u)_{B_R}  \right\|_{\underline{L}^2(B_{R})} 
\leq \mathsf{M},
\end{equation}
there exists $v \in \mathcal{L}_1$ such that, for all $r \geq \X$, 
\begin{equation} \label{e.tildephivsphi1}
\left\| \nabla v  -  \nabla u - \nabla w \right\|_{\underline{L}^2(B_{r})} \leq C \left(   \left\| \nabla w \right\|_{\underline{L}^2(B_r)} \right)^{1+\delta}.
\end{equation}
\end{enumerate}
\end{lemma}

\begin{proof}
Let $\X$ be the maximum of minimal scales in Theorem~\ref{t.C11estimate}(i)--(ii), Theorem~\ref{t.regularitylin} and Lemma~\ref{l.regularity2}, as well as a constant~$H(\sigma,\mathsf{M},\data)<\infty$ to be fixed.  Clearly we have~$\X \leq \O_\sigma(C)$. 

\smallskip

\emph{Step 1}.  The proof of statement (i). Fix $u,v \in \mathcal{L}_1$  satisfying~\eqref{e.C11phitildephicond2}.  Denote $z = v - u$. Throughout the proof, for given $r \in [\X , \infty)$ we denote by $\psi_r$ a function realizing the infimum below:
\begin{equation*} 
\left\| z - \psi_r \right\|_{\underline{L}^2 \left( B_{r} \right)} = \inf_{\psi \in \A_1[u] }\left\| z - \psi \right\|_{\underline{L}^2 \left( B_{r} \right)}.
\end{equation*}

\smallskip 

We first collect some consequences of our preliminary results.  By~\eqref{e.diffdoubling},
\begin{equation} \label{e.Harnackz}
\sup_{t \in [ \X , \infty) } \left\| \nabla z  \right\|_{\underline{L}^2 \left( B_{t} \right)} \leq  C \inf_{t \in [ \X , \infty)} \left\| \nabla z  \right\|_{\underline{L}^2 \left( B_{t} \right)} .
\end{equation}
Moreover, by Theorem~\ref{t.regularitylin} and Lemma~\ref{l.regularity2}, we find $\psi_0 \in \A_1[u]$ such that, for $r\geq \X$,
\begin{equation*} 
\left\| z - \psi_0 \right\|_{\underline{L}^2 \left( B_{r} \right)} \leq C r^{-2\beta} \left\| \psi_0 \right\|_{\underline{L}^2 \left( B_{r} \right)},
\end{equation*}
where we denote $\beta := \frac12 \alpha (d-\sigma)$. Taking $H$ so large that $C H^{-\beta} \leq \frac12$, we also obtain by the previous two displays that, for all $r \geq \X$, 
\begin{equation}  \label{e.naivezvspsi}
\frac1r \left\| z - \psi_0 \right\|_{\underline{L}^2 \left( B_{r} \right)} \leq Cr^{-2\beta} \inf_{t \in [ \X , \infty)} \left\| \nabla z  \right\|_{\underline{L}^2 \left( B_{t} \right)}.
\end{equation}
Define 
\begin{equation*} 
\hat r := \inf_{t \in [ \X , \infty)} \left( \left\| \nabla z  \right\|_{\underline{L}^2 \left( B_{t} \right)} \right)^{ - \frac{\delta_0}{\beta}} ,
\end{equation*}
where $\delta_0$ is as in~\eqref{e.basiclinearizationerror} below, coming from the Meyers estimate for $z$. 
On the one hand, for $r \geq \X \vee \hat r$ we have by~\eqref{e.naivezvspsi} that 
\begin{equation} \label{e.rlargerthanhatr}
\inf_{\psi \in \A_1[u] } \frac1r \left\| z - \psi \right\|_{\underline{L}^2 \left( B_{r} \right)}  
\leq 
\frac1r  \left\| z - \psi_0 \right\|_{\underline{L}^2 \left( B_{r} \right)} 
\leq 
Cr^{-\beta} \inf_{t \in [ \X , \infty)} \left( \left\| \nabla z  \right\|_{\underline{L}^2 \left( B_{t} \right)} \right)^{1+\delta_0} . 
\end{equation}
On the other hand, if $\hat r> \X$, then for $r \in [\X, \hat r)$  we get
\begin{equation} \label{e.rsmallerthanhatr}
\inf_{t \in [ \X , \infty)} \left( \left\| \nabla z  \right\|_{\underline{L}^2 \left( B_{t} \right)} \right)^{\delta_0} 
 =   
\hat{r}^{\, -\beta }  
\leq r^{-\beta }  .
\end{equation}
We next focus on the case $\hat r > \X$ in more detail. Observe that $z$ solves 
\begin{equation} \label{e.zeginreg}
-\nabla \cdot \left( D_p^2 L( \nabla u,\cdot) \nabla z  \right) = \nabla \cdot \left( \int_{0}^1 \left( D_p^2 L( \nabla u  {+} t\nabla z,\cdot)   {-} D_p^2 L( \nabla u,\cdot)   \right)\,dt  \, \nabla z \right).
\end{equation}
Letting $w_r$ solve, for $r \in [\X , \hat r)$, 
\begin{equation*}
\left\{ \begin{aligned}
& -\nabla \cdot \left( D_p^2 L( \nabla u,\cdot) \nabla w_r  \right)= 0 
& \mbox{in} & \ B_{r/2}, \\
& w_r = z - \psi_r  & \mbox{on} & \ \partial B_{r/2},
\end{aligned} \right.
\end{equation*}
we have that 
\begin{equation*} 
\left\| \nabla  z {-} \nabla \psi_r    {-} \nabla  w_r \right\|_{\underline{L}^2 \left( B_{r/2} \right)} 
\leq C 
\left\| \int_{0}^1 \left( D_p^2 L( \nabla u  {+} t\nabla z,\cdot)   {-} D_p^2 L( \nabla u,\cdot)   \right)\,dt \nabla z  \right\|_{\underline{L}^2 \left( B_{r/2} \right)} . 
\end{equation*}
It follows by the Meyers estimate and regularity of $p\mapsto D_p^2L(p,\cdot)$ that there exist constants $C(\data)<\infty$ and $\delta_0(\data) \in\left( 0, \frac12 \right]$ such that, for all $\delta \in [0,2 \delta_0]$, 
\begin{equation}  \label{e.basiclinearizationerror}
\left\| \int_{0}^1 \left( D_p^2 L( \nabla u {+} t\nabla z,\cdot)   {-} D_p^2 L( \nabla u,\cdot)   \right)\,dt \nabla z  \right\|_{\underline{L}^2 \left( B_{r/2} \right)} 
\leq  
C \left(  \left\| \nabla z \right\|_{\underline{L}^2 \left( B_{r} \right)}\right)^{1{+} \delta} .
\end{equation}
Since we assume that $r \in [\X,\hat r)$, we obtain by~\eqref{e.rsmallerthanhatr},~\eqref{e.Harnackz}, and Poincar\'e's inequality that
\begin{equation*} 
\left\|  z-  \psi_r  - w_r \right\|_{\underline{L}^2 \left( B_{r/2} \right)}  
\leq 
C r^{1-\beta} \inf_{t \in [ \X , \infty)} \left( \left\| \nabla z  \right\|_{\underline{L}^2 \left( B_{t} \right)} \right)^{1+\delta_0} .
\end{equation*}
Furthermore, by Theorem~\ref{t.regularitylin}, there is $ \tilde \psi \in \A_1[u]$ such that 
\begin{equation*} 
\left\| \tilde \psi   -  w_r \right\|_{\underline{L}^2 \left( B_{\theta r} \right)} \leq C \theta^2
\left\| w_r \right\|_{\underline{L}^2 \left( B_{r} \right)} .
\end{equation*}
We then obtain, by the triangle inequality, that 
\begin{equation*} 
\left\| z  - \tilde \psi  - \psi_r  \right\|_{\underline{L}^2 \left( B_{\theta r} \right)}
\leq C \theta^2 \left\|  z   - \psi_r \right\|_{\underline{L}^2 \left( B_{ r} \right)} 
+ C \theta^{-\frac d2} \left\| z  -   \psi_r  -  w_r\right\|_{\underline{L}^2 \left( B_{r/2} \right)} ,
\end{equation*}
so that for $r \geq \theta^{-1}\X$,  by choosing $\theta C = \frac12$, we arrive at
\begin{equation*} 
\frac{1}{\theta r} \left\| z - \psi_{\theta r} \right\|_{\underline{L}^2 \left( B_{\theta r} \right)}
\leq \frac12
\frac{1}{r} \left\| z - \psi_r \right\|_{\underline{L}^2 \left( B_{r} \right)} + C r^{-\beta}  \inf_{t \in [ \X , \infty)} \left( \left\| \nabla z  \right\|_{\underline{L}^2 \left( B_{t} \right)} \right)^{1+\delta_0}. 
\end{equation*}
An iteration yields, for every $r \in [\X,\hat r]$, that 
\begin{equation*} 
\frac1r \left\|  z - \psi_r \right\|_{\underline{L}^2 \left( B_{r} \right)} \leq C 
 \frac1{\hat r} \left\|  z - \psi_{\hat r} \right\|_{\underline{L}^2 \left( B_{\hat r} \right)}
+ Cr^{-\beta} \inf_{t \in [ \X , \infty)} \left( \left\| \nabla z  \right\|_{\underline{L}^2 \left( B_{t} \right)} \right)^{1+\delta_0} ,
\end{equation*}
and the first term can be bounded by~\eqref{e.rlargerthanhatr}. In conclusion, we have proved that, for all $r \geq \X$, 
\begin{equation*} 
 \frac1r \left\|  z - \psi_r \right\|_{\underline{L}^2 \left( B_{r} \right)} \leq Cr^{-\beta}  \inf_{t \in [ \X , \infty)} \left( \left\| \nabla z  \right\|_{\underline{L}^2 \left( B_{t} \right)} \right)^{1+\delta_0} .
\end{equation*}
Furthermore, we get by the above display and the fact that $\psi_{2^j \X},\psi_{2^{j+1} \X} \in \A_1[u]$ that, for all $j \in \N$ and $r \in [\X, 2^j \X]$, 
\begin{align} \notag 
\frac1r \left\| \psi_{2^j \X}- \psi_{2^{j+1}\X} \right\|_{\underline{L}^2 \left( B_{r} \right)} 
& \leq 
C \frac1{2^j \X} \left\| \psi_{2^j \X}- \psi_{2^{j+1}\X} \right\|_{\underline{L}^2 \left( B_{2^{j} \X} \right)} 
\\ \notag &
\leq 
C 2^{- j \beta}    \inf_{t \in [ \X , \infty)} \left( \left\| \nabla z  \right\|_{\underline{L}^2 \left( B_{t} \right)} \right)^{1+\delta_0} ,
\end{align}
which leads easily to
\begin{equation*} 
\sup_{s \in [\X,\infty)} \frac1r \left\| \psi_{\X}- \psi_{s} \right\|_{\underline{L}^2 \left( B_{r} \right)} \leq C \inf_{t \in [ \X , \infty)} \left( \left\| \nabla z  \right\|_{\underline{L}^2 \left( B_{t} \right)} \right)^{1+\delta_0} .
\end{equation*}
Therefore, we obtain, with $w = \psi_\X \in \A_1[u]$, that, for all $r \geq \X$, 
\begin{equation*} 
 \frac1r \left\|  z - w \right\|_{\underline{L}^2 \left( B_{r} \right)} \leq C \inf_{t \in [ \X , \infty)} \left( \left\| \nabla z  \right\|_{\underline{L}^2 \left( B_{t} \right)} \right)^{1+\delta_0} .
\end{equation*}
Finally, by the equations of $z$ and $w$, the Meyers and Caccioppoli estimates and~\eqref{e.Harnackz}, we obtain, for every $r \geq \X$, 
\begin{equation*} 
\left\| \nabla z - \nabla w \right\|_{\underline{L}^2 \left( B_{r} \right)} \leq \frac{C}{r} \left\| z - w \right\|_{\underline{L}^2 \left( B_{2r} \right)}
+ 
C \inf_{t \in [ \X , \infty)} \left( \left\| \nabla z  \right\|_{\underline{L}^2 \left( B_{t} \right)} \right)^{1+\delta_0}.
\end{equation*}
Combining previous two displays yields~\eqref{e.tildephivsphi2}.

\smallskip

\emph{Step 2}. The proof of statement~(ii). 
Fix $u \in \mathcal{L}_1$ satisfying~\eqref{e.C11phitildephicond} and $w \in \A_1[u]$.  Observe first that by Theorem~\ref{t.regularitylin} we have that 
\begin{equation} \label{e.Harnack2}
\sup_{t \in [ \X , \infty) } \left\| \nabla w  \right\|_{\underline{L}^2 \left( B_{t} \right)} \leq  C \inf_{t \in [ \X , \infty)} \left\| \nabla w  \right\|_{\underline{L}^2 \left( B_{t} \right)} .
\end{equation}
Given small $\eta \in  (0,1)$ to be fixed shortly, we may assume that 
\begin{equation}  \label{e.wsmallass000}
 \sup_{t \in [ \X , \infty)} \left\| \nabla w  \right\|_{\underline{L}^2 \left( B_{t} \right)} \leq \eta. 
\end{equation}
 Indeed, otherwise 
 \begin{equation*} 
\inf_{t \in [ \X , \infty)} \left\| \nabla w  \right\|_{\underline{L}^2 \left( B_{t} \right)} \geq \frac{\eta}{C}, 
\end{equation*}
and we may simply take $v = u$ in this case to obtain~\eqref{e.tildephivsphi1} trivially. We henceforth assume~\eqref{e.wsmallass000}. 

\smallskip

By Theorem~\ref{t.C11estimate}(i)  and Theorem~\ref{t.regularitylin}, we find that there exist affine functions~$\ell_w$ and~$\ell_u$ such that, for every~$r \geq \X$,  
\begin{equation*} 
\left\| w - \ell_w  \right\|_{\underline{L}^2 \left( B_{r} \right)} \leq Cr^{-\alpha(d-\sigma)} \left\| w   \right\|_{\underline{L}^2 \left( B_{r} \right)} 
\quad \mbox{and}  \quad 
\left\| u - \ell_u  \right\|_{\underline{L}^2 \left( B_{r} \right)} \leq Cr^{1-\alpha(d-\sigma)} .
\end{equation*}
Observe that we have
\begin{equation*} 
|\nabla \ell_w| \leq 
\frac Cr \left\| \ell_w -(\ell_w)   \right\|_{\underline{L}^2 \left( B_{r} \right)} 
\leq C r^{-\alpha(d-\sigma)-1} \left\| \ell_w   \right\|_{\underline{L}^2 \left( B_{r} \right)} + C \left\| \nabla w  \right\|_{\underline{L}^2 \left( B_{r} \right)},
\end{equation*}
so thus by sending $r \to \infty$ yields by~\eqref{e.wsmallass000} that
\begin{equation*} 
|\nabla \ell_w| \leq  C \eta. 
\end{equation*}
Theorem~\ref{t.C11estimate}(ii) then yields~$v \in \mathcal{L}_1$ such that
\begin{equation*} 
\left\| v - \ell_u - \ell_w  \right\|_{\underline{L}^2 \left( B_{r} \right)} \leq Cr^{1-\alpha(d-\sigma)}  .
\end{equation*}
By the Caccioppoli inequality,
\begin{align} \notag 
\lefteqn{ \sup_{t \in [\X,\infty)} \left\| \nabla u  - \nabla v \right\|_{\underline{L}^2(B_{t})} } \quad &
\\ \notag &
 \leq  \limsup_{t \to \infty} \frac{C}{t } \left\| (u-\ell_u)    - (v  - \ell_u - \ell_w) \right\|_{\underline{L}^2(B_{2t})} + \limsup_{t \to \infty} \frac{C}{t} \left\| \ell_w \right\|_{\underline{L}^2(B_{t})} 
\leq C \eta .
\end{align}
Step 1 yields the existence of $\tilde w \in \A_1[u]$ such that
\begin{equation*} 
\sup_{t \in [ \X,\infty) } \left\| \nabla v - \nabla u - \nabla \tilde w  \right\|_{\underline{L}^2 \left( B_{t} \right)}  \leq C \left( \inf_{t \in [ \X,\infty) }  \left\| \nabla u  - \nabla v \right\|_{\underline{L}^2(B_t)} \right)^{1+\delta}.
\end{equation*}
By Caccioppoli inequality, similarly as in Step 1, we get, for $r,t \in [\X,\infty)$, 
\begin{equation*} 
 \left\| \nabla v - \nabla u - \nabla w  \right\|_{\underline{L}^2 \left( B_{t} \right)}
 \leq 
C t^{-\alpha(d-\sigma)} + C\left(  \left\| \nabla u  - \nabla v \right\|_{\underline{L}^2(B_r)} \right)^{1+\delta} .
\end{equation*}
It follows by the Lipschitz estimate for $w - \tilde w  \in \A_1[u]$ that
\begin{align} \notag 
\left\| \nabla w - \nabla \tilde w \right\|_{\underline{L}^2 \left( B_{r} \right)} 
& 
\leq 
C \limsup_{t \to \infty} \left\| \nabla w - \nabla \tilde w \right\|_{\underline{L}^2 \left( B_{t} \right)}
\\ \notag &
\leq 
C \sup_{t \in [ \X,\infty) } \left\| \nabla v - \nabla u - \nabla \tilde w  \right\|_{\underline{L}^2 \left( B_{t} \right)} +  C \left(  \left\| \nabla u  - \nabla v \right\|_{\underline{L}^2(B_r)} \right)^{1+\delta} 
\\ \notag &
\leq C\left( \left\| \nabla u  - \nabla v \right\|_{\underline{L}^2(B_r)} \right)^{1+\delta}.
\end{align}
Therefore, we have that 
\begin{equation*} 
 \left\| \nabla v - \nabla u - \nabla w  \right\|_{\underline{L}^2 \left( B_{r} \right)}  \leq C \left(  \left\| \nabla u  - \nabla v \right\|_{\underline{L}^2(B_r)} \right)^{1+\delta}.
\end{equation*}
Since $ \left\| \nabla u  - \nabla v \right\|_{\underline{L}^2(B_r)} \leq C \eta$, we may choose $(C\eta)^\delta = \frac12$, and conclude that 
\begin{equation*} 
\left\| \nabla u  - \nabla v \right\|_{\underline{L}^2(B_r)} \leq 2 \left\|  \nabla w \right\|_{\underline{L}^2(B_r)} ,
\end{equation*}
giving~\eqref{e.tildephivsphi1} by the last three displays, and concluding the proof.  
\end{proof}


\subsection{The large-scale $C^{1,1}$-type excess decay estimate}
\label{ss.Ck1conc}
We have now assembled the ingredients necessary to complete the proof of Theorem~\ref{t.C11estimate}. 

\begin{proof}
[{Proof of Theorem~\ref{t.C11estimate} (iii)}]
We fix $\sigma \in (0,d)$ and $\mathsf{M} \in [1,\infty)$. For these parameters, let $\X$ be the maximum of random minimal scales appearing in Theorem~\ref{t.regularity.differences}, Theorem~\ref{t.AS.regularity}, Theorem~\ref{t.regularitylin} and Lemma~\ref{l.tildephivsphi1}. Let $R \geq \X$. Suppose that $u \in \L(B_R)$ satisfies $\frac1R \left\| u - (u)_{B_R} \right\|_{\underline{L}^2 \left( B_{R} \right)}  \leq \mathsf{M}$.

\smallskip

The proof proceeds in several steps. In the first four steps we show an intermediate result, that is, we show that for all $\alpha \in (0,1)$ there exists a constant $C(\sigma, \alpha,\mathsf{M},\data)<\infty$ and $\phi \in \mathcal{L}_1$ such that 
\begin{equation}  \label{e.C11-}
\left\| \nabla u - \nabla \phi \right\|_{\underline{L}^2 \left( B_{r} \right)} \leq C \left( \frac rR\right)^{\alpha}  \frac 1R \left\| u - (u)_{B_R} \right\|_{\underline{L}^2 \left( B_{R} \right)} .
\end{equation}
In the last two steps we demonstrate how this can be improved to~\eqref{e.C11}.

\smallskip
We denote, in short,
\begin{equation}  \label{e.MRC11}
\mathsf{M}_R := \inf_{\phi \in \mathcal{L}_1} \frac1R \left\| u - \phi - (u - \phi)_{B_R} \right\|_{\underline{L}^2(B_{R})}, 
\end{equation}
which obviously satisfies $\mathsf{M}_R \leq C$ by the normalization 
\begin{equation}  \label{e.normuC11}
\frac1R \left\| u - (u)_{B_R} \right\|_{\underline{L}^2 \left( B_{R} \right)}  \leq \mathsf{M}.
\end{equation}

\smallskip

\emph{Step 1.} 
We first show that, for all $\eta \in (0,1]$, there exist constants $\gamma(\eta,\mathsf{M},\sigma,\data) \in \left(0,\frac12\right]$ and 
$H(\eta,\mathsf{M},\sigma,\data)<\infty$ such that if $\gamma R \geq \X \vee H$, then, for $r \in [\X \vee H, \gamma R]$, 
\begin{equation} \label{e.uminusphialwayssmall}
\inf_{ \phi \in \mathcal{L}_1} \left\| \nabla u - \nabla \phi \right\|_{\underline{L}^2 \left( B_{r} \right)} \leq \eta. 
\end{equation}
The parameter $\eta$ will be fixed in the end of Step 3. 
Set $r_0 := \gamma R$ and assume that $r_0 \geq \X \vee H$. To prove~\eqref{e.uminusphialwayssmall}, we first observe by Theorem~\ref{t.AS.regularity} and~\eqref{e.normuC11} that there exists $\ell \in \mathcal{P}_1$ such that $|\nabla \ell| \leq C$  and
\begin{equation*} 
\left\| u - \ell  \right\|_{\underline{L}^2(B_{2 r_0})} 
  \leq  Cr_0 \left(  \gamma^{\alpha} +  \gamma^{-1-\frac d2} H^{-\beta(d-\sigma)} \right).
\end{equation*}
By statement (ii) of Theorem~\ref{t.C11estimate}, there exists $\phi \in \mathcal{L}_1$ such that 
\begin{equation*} 
 \left\| \phi - \ell  \right\|_{\underline{L}^2(B_{r_0})} \leq C r_0^{1 -\beta(d-\sigma)} .
\end{equation*}
Combining above estimates with the Caccioppoli estimate implies that there exists $\phi \in \mathcal{L}_1$ such that 
\begin{equation*} 
 \left\| \nabla u - \nabla \phi  \right\|_{\underline{L}^2(B_{r_0})} 
  \leq  C\left( \gamma^{\alpha} +  \gamma^{-1-\frac d2} H^{-\beta(d-\sigma)} \right).
\end{equation*}
By the Lipschitz estimate in Theorem~\ref{t.regularity.differences} we then obtain that, for all $r \in [\X, r_0 ]$, 
\begin{equation*} 
 \left\| \nabla u - \nabla \phi  \right\|_{\underline{L}^2(B_{r})} 
  \leq  C\left( \gamma^{\alpha} +  \gamma^{-1-\frac d2} H^{-\beta(d-\sigma)} \right).
\end{equation*}
We then choose $\gamma(\eta,C,\alpha) \in \left(0,\frac12\right]$ small and then $H(\gamma,\eta,C,\sigma,\beta)<\infty$ large so that 
\begin{equation*} 
C\left( \gamma^{\alpha} +  \gamma^{-1-\frac d2} H^{-\beta(d-\sigma)} \right) \leq  \eta.
\end{equation*}
We therefore have that~\eqref{e.uminusphialwayssmall} is valid.

\smallskip

\emph{Step 2.} 
Statement of the induction assumption.  Fix $\alpha \in (0,1)$. Let $r_0 = \gamma R$ and assume that $r_0 \geq \X \wedge H$, where $\gamma$ and $H$ are as in Step 1. Let $\theta  \in \left(0,\frac12\right]$ to be a constant to be fixed and set, for $j \in \N$, $r_j := \theta^j r_0$. Take $n \in \N_0$ be such that $\X \vee H \in (r_{n+1},r_n]$.  We assume inductively that for some $m \in \{0,\ldots,n\}$ there exists  such that, 
for all $j \in \{0,\ldots,m\}$, 
\begin{equation} \label{e.C11minusindass1}
 \inf_{\phi \in \mathcal{L}_1 } \left\| \nabla u - \nabla \phi   \right\|_{\underline{L}^2(B_{r_j})}   \leq \theta^{\alpha j}  
 \inf_{\phi \in \mathcal{L}_1 } \left\| \nabla u - \nabla \phi   \right\|_{\underline{L}^2(B_{r_0})}   .
\end{equation}
This is trivially true for $m=0$, which serves as our initial step for induction.  We denote by $\phi_j$ a member of $\mathcal{L}_1$ realizing the infimum in $ \inf_{\phi \in \mathcal{L}_1 } \left\| \nabla u - \nabla \phi   \right\|_{\underline{L}^2(B_{r_j})}$.

\smallskip

\emph{Step 3.} We show the induction step, that is, for $\alpha \in (0,1)$, there exists $\theta(\alpha,\data) \in (0,1)$ and $\phi_{m+1} \in \mathcal{L}_1$ such that 
\begin{equation} \label{e.C11minusdecaypre2}
\left\| \nabla u - \nabla \phi_{m+1} \right\|_{\underline{L}^2(B_{r_{m}})}  \leq  \theta^{\alpha}   \left\| \nabla u - \nabla \phi_{m} \right\|_{\underline{L}^2(B_{r_{m+1}})} ,
\end{equation}
which obviously proves the induction step. Let $h$ solve 
\begin{equation*}
\left\{ \begin{aligned}
& -\nabla \cdot \left( D_p^2 L( \nabla \phi_m,\cdot) \nabla h  \right)= 0 
& \mbox{in} & \ B_{r_m/2}, \\
& h  = u - \phi_m  & \mbox{on} & \ \partial B_{r_m/2},
\end{aligned} \right.
\end{equation*}
Similarly to the proof of Lemma~\ref{l.tildephivsphi1}, we have that 
\begin{equation*}
\left\| 
\nabla u - \nabla \phi_{m} + \nabla h 
 \right\|_{\underline{L}^2(B_{r_{m}/2})} 
\leq 
C   \left( \left\| \nabla u - \nabla \phi_{m} \right\|_{\underline{L}^{2}(B_{r_{m}})} \right)^{1+\delta}.
\end{equation*}
Now~\eqref{e.uminusphialwayssmall} implies that  
\begin{equation} \label{e.homogenization.estimates.cor.applied}
 \left\| 
\nabla u - \nabla \phi_{m} + \nabla h 
 \right\|_{\underline{L}^2(B_{r_{m}/2})}   \leq C \eta^\delta  \left\| \nabla u - \nabla \phi_{m} \right\|_{\underline{L}^2(B_{r_{m}})} .
\end{equation}
Furthermore, by Theorem~\ref{t.regularitylin}, Caccioppoli inequality and the triangle inequality, we find $\psi \in \mathcal{A}_1\left[ \phi_{m} \right]$ such that 
\begin{equation} \label{e.C11w2vspsi}
\left\| \nabla h - \nabla \psi  \right\|_{\underline{L}^2(B_{r_{m+1}})}
 \leq 
C \theta  \left\| \nabla h   \right\|_{\underline{L}^2\left(B_{r_{m}/2}\right)}.
\end{equation}
By~\eqref{e.homogenization.estimates.cor.applied} and the triangle inequality, we get
\begin{align} \label{e.C11w2}
\left\| \nabla h   \right\|_{\underline{L}^2\left(B_{r_{m}/2}\right)}
 & 
\leq   \left( 1+ C  \eta^{\delta}  \right)   \left\| \nabla u - \nabla \phi_{m}  \right\|_{\underline{L}^2\left(B_{r_{m}}\right)},
\end{align}
and hence
\begin{equation} \label{e.C11w2vspsi3}
\left\| \nabla h - \nabla \psi \right\|_{\underline{L}^2(B_{r_{m+1}})} 
 \leq 
C \theta   \left\| \nabla u - \nabla \phi_{m}  \right\|_{\underline{L}^2\left(B_{r_{m}}\right)}.
\end{equation}
Now, applying Lemma~\ref{l.tildephivsphi1}, we find~$\phi_{m+1} \in \mathcal{L}_1$ such that 
\begin{equation*} 
\left\| \nabla \phi_m - \nabla \phi_{m+1} - \nabla \psi \right\|_{\underline{L}^2(B_{r_{m+1}})} \leq C \left( \left\|  \nabla \psi \right\|_{\underline{L}^2(B_{r_{m+1}})}\right)^{1+\delta} .
\end{equation*}
By the Lipschitz estimate,~\eqref{e.C11w2vspsi} and~\eqref{e.C11w2}, we have that 
\begin{equation*} 
\left\|  \nabla \psi \right\|_{\underline{L}^2(B_{r_{m+1}})} \leq C \left\|  \nabla \psi \right\|_{\underline{L}^2(B_{r_{m}/2})} \leq
C \left\| \nabla u - \nabla \phi_{m}  \right\|_{\underline{L}^2\left(B_{r_{m}}\right)}.
\end{equation*}
Thus we get by~\eqref{e.uminusphialwayssmall} that
\begin{equation*}  \label{e.C11w2vspsi4}
\left\| \nabla \phi_m - \nabla \phi_{m+1} - \nabla \psi \right\|_{\underline{L}^2(B_{r_{m+1}})} \leq C \eta^\delta  \left\| \nabla u - \nabla \phi_{m}  \right\|_{\underline{L}^2\left(B_{r_{m}}\right)}.
\end{equation*}
By the triangle inequality we have that
\begin{align} \notag 
 \left\| \nabla u - \nabla \phi_{m+1} \right\|_{\underline{L}^2(B_{r_{m+1}})}
& 
\leq \left\|  \nabla h -  \nabla \psi \right\|_{\underline{L}^2(B_{r_{m+1}})} 
+ \theta^{-\frac d2} \left\|  \nabla u -  \nabla \phi_{m} +  \nabla h  \right\|_{\underline{L}^2(B_{ r_m / 2 } )}    
\\  \notag & \qquad
 + \left\|  \nabla \phi_m -  \nabla \phi_{m+1} -  \nabla \psi \right\|_{\underline{L}^2(B_{r_{m+1}})} .
\end{align}
The terms on the right can be estimated using~\eqref{e.C11w2vspsi3},~\eqref{e.homogenization.estimates.cor.applied} and~\eqref{e.C11w2vspsi4}, respectively,  and we arrive at
\begin{equation*} 
 \left\| \nabla u - \nabla \phi_{m+1} \right\|_{\underline{L}^2(B_{r_{m+1}})} \leq
C \left( \theta +  \theta^{-\frac d2} \eta^\delta \right) \left\| \nabla u - \nabla \phi_{m}  \right\|_{\underline{L}^2\left(B_{r_{m}}\right)}. 
\end{equation*}
We then choose first $\theta$ so that $C \theta^{1-\alpha} = \frac12$, and then $\eta$ so  that $C \theta^{-\alpha-\frac d2} \eta^\delta =\frac12$. With these choices, we get~\eqref{e.C11minusdecaypre2} from the above display.

\smallskip

\emph{Step 4.} Proof of $C^{1,1-}$. We show that for all $\alpha \in (0,1)$ there exists a constant $C( \alpha,\sigma,\mathsf{M},\data)<\infty$ and $\phi \in \mathcal{L}_1$ such that 
\begin{equation}  \label{e.C11pre}
\left\| \nabla u - \nabla \phi \right\|_{\underline{L}^2 \left( B_{r} \right)} \leq C \left( \frac rR\right)^{\alpha}  \mathsf{M}_R ,
\end{equation}
where $\mathsf{M}_R$ is defined in~\eqref{e.MRC11}. In the last two steps we demonstrate how this can be improved to~\eqref{e.C11}.

\smallskip

We show that the corrector can be chosen uniformly in $m$, that is, $\phi_m$ can be replaced with $\phi \in \mathcal{L}_1$ for all $m$.  By~\eqref{e.diffdoubling}, we have that, for $r \in [\X,\infty)$, 
\begin{equation*} 
\left\| \nabla \phi_{m+1} - \nabla \phi_m \right\|_{\underline{L}^2\left(B_r\right)} 
\leq  C \left\|\nabla \phi_{m+1} - \nabla \phi_m  \right\|_{\underline{L}^2\left(B_{r_{m+1}}\right)} .
\end{equation*}
The triangle inequality implies that 
\begin{equation*} 
\left\|\nabla \phi_{m+1} - \nabla \phi_m  \right\|_{\underline{L}^2\left(B_{r_{m+1}}\right)}
\leq 
\left\|\nabla u - \nabla \phi_{m+1}  \right\|_{\underline{L}^2\left(B_{r_{m+1}}\right)} + \theta^{-\frac d2} \left\| \nabla u - \nabla \phi_m  \right\|_{\underline{L}^2\left(B_{r_m}\right)}  . 
\end{equation*}
It follows, by the triangle inequality and~\eqref{e.C11minusindass1}, that for $r \in [\X,\infty)$ we get
\begin{equation*} 
\left\|\nabla \phi_{m+1} - \nabla \phi_m  \right\|_{\underline{L}^2\left(B_r \right)} \leq C \left( \frac {r_m}{r_0} \right)^{\alpha} \inf_{\psi \in \mathcal{L}_1}\left\| \nabla u - \nabla \psi \right\|_{\underline{L}^2(B_{r_0})}.   
\end{equation*}
Summation then yields, for $r \in [r_{m+1},r_m]$, 
\begin{equation*} 
\left\|\nabla \phi_{m} - \nabla \phi_{n+1}  \right\|_{\underline{L}^2\left(B_{r}\right)} \leq C \left( \frac {r}{r_0} \right)^{\alpha} \inf_{\psi \in \mathcal{L}_1}\left\| \nabla u - \nabla \psi \right\|_{\underline{L}^2(B_{r_0})}.    
\end{equation*}
Therefore, by the triangle inequality, we get
\begin{equation} \label{e.C11decay3}
\left\| \nabla u - \nabla \phi_{n+1} \right\|_{\underline{L}^2(B_{r})}  
\leq C \left( \frac{r}{r_0} \right)^\alpha \inf_{\psi \in \mathcal{L}_1}\left\| \nabla u - \nabla \psi \right\|_{\underline{L}^2(B_{r_0})},
\end{equation}
and hence we may take $\phi = \phi_{n+1}$ to obtain~\eqref{e.C11pre} by applying Caccioppoli inequality and giving up a volume factor.

\smallskip

\emph{Step 5.} In this step we show that there exist constants  $\beta(\data)>0$, $C(\sigma,\mathsf{M},\data)<\infty$ 
and  $ \psi \in \A_2[\phi] $ such that 
\begin{equation}  \label{e.C2}
\left\|  \nabla u- \nabla \phi - \nabla \psi   \right\|_{\underline{L}^2 \left( B_{r} \right)} \leq C  \left( \frac{r}{R}  \right)^{1+\beta} \mathsf{M}_R. 
\end{equation}
We immediately choose $\alpha = \frac{2+\delta}{2+2\delta}$ in~\eqref{e.C11pre} and $\beta = (1+\delta) \alpha -1 = \frac{\delta}{1+\delta}$, where $\delta$ is as in Lemma~\ref{l.tildephivsphi1}. We will show that~\eqref{e.C2} is valid with this $\beta$. 
Denote $z = u-\phi$ and let $\psi_{m}  \in \mathcal{A}_2[\phi]$ be such that 
\begin{equation*} 
\left\| \nabla z  - \nabla \psi_{m} \right\|_{\underline{L}^2 \left( B_{r_m} \right)} = \inf_{\psi \in \mathcal{A}_2[\phi]}\left\| \nabla z  - \nabla \psi \right\|_{\underline{L}^2 \left( B_{r_m} \right)}. 
\end{equation*}
By~\eqref{e.C11pre}, 
\begin{equation} \label{e.C11psiobvious}
\left\| \nabla \psi_{m} \right\|_{\underline{L}^2 \left( B_{r_m} \right)}  \leq 
C \left( \frac {r_m}R \right)^{\alpha}   \mathsf{M}_R .
\end{equation}
Let $h \in H^1(B_{r_m/2})$ solve 
\begin{equation*}
\left\{ \begin{aligned}
& -\nabla \cdot \left( D_p^2 L( \nabla \phi,\cdot) \nabla h  \right)= 0 
& \mbox{in} & \ B_{r_m/2}, \\
& h  = z - \psi_m  & \mbox{on} & \ \partial B_{r_m/2},
\end{aligned} \right.
\end{equation*}
so that, as in the proof of Lemma~\ref{l.tildephivsphi1},
\begin{equation*} 
\left\|  \nabla z      - \nabla \psi_m  - \nabla h  \right\|_{\underline{L}^2 \left( B_{r_{m}/2} \right)}  \leq C  \left( \left\| \nabla  z \right\|_{\underline{L}^2 \left( B_{r_{m}} \right)} \right)^{1+\delta} .
\end{equation*}
Using~\eqref{e.C11pre} and~\eqref{e.uminusphialwayssmall}, recalling that $\mathsf{M}_R \leq C$, we get
\begin{equation*} 
\left\| \nabla z -  \nabla \psi_{m} -  \nabla h \right\|_{\underline{L}^2 \left( B_{r_{m+1}} \right)} \leq C \left( \frac{r_{m+1}}{R}  \right)^{1+\beta}   \mathsf{M}_R . 
\end{equation*}
On the other hand, by Theorem~\ref{t.regularitylin}, there exists $\tilde \psi_m \in \mathcal{A}_2[\phi]$ such that 
\begin{equation*} 
\left\|  \nabla h - \nabla \tilde \psi_{m+1}   \right\|_{\underline{L}^2 \left( B_{r_{m+1}} \right)} \leq C \theta^2 
\left\|   \nabla h   \right\|_{\underline{L}^2 \left( B_{r_{m}/2} \right)}  .
\end{equation*}
Therefore, by the triangle inequality and the previous two displays, by choosing $\theta$ small appropriately, we get
\begin{align} \notag 
\lefteqn{\inf_{\psi \in \mathcal{A}_2[\phi]}\left\| \nabla z  - \nabla \psi \right\|_{\underline{L}^2 \left( B_{r_{m+1}} \right)}} \qquad &
\\ \notag & \leq 
\left\|  \nabla z - \nabla (\psi_{m} + \tilde \psi_{m+1}  )   \right\|_{\underline{L}^2 \left( B_{r_{m+1}} \right)}
\\ \notag & \leq 
 \left\|  \nabla h - \nabla \tilde \psi_{m+1}   \right\|_{\underline{L}^2 \left( B_{r_{m+1}} \right)} + \left\|  \nabla z -  \nabla \psi_{m} -  \nabla h   \right\|_{\underline{L}^2 \left( B_{r_{m+1}} \right)}
 \\ \notag & \leq 
 C \theta^2  \left\| \nabla h \right\|_{\underline{L}^2 \left( B_{r_{m}/2} \right)} + 2 \theta^{-\frac d2} \left\|  \nabla z -  \nabla \psi_{m} -  \nabla h   \right\|_{\underline{L}^2 \left( B_{r_{m}/2} \right)}
  \\ \notag & \leq 
  \theta^{\alpha(1+\delta)}  \left\| \nabla  z - \nabla  \psi_m  \right\|_{\underline{L}^2 \left( B_{r_{m}} \right)} + C \left( \frac{r_{m+1}}{R}  \right)^{1+\beta}   \mathsf{M}_R .
\end{align}
An iteration then gives
\begin{equation*} 
\left\|  \nabla z - \nabla \psi_m   \right\|_{\underline{L}^2 \left( B_{r_{m}} \right)} 
\leq 
C \left( \frac{r_{m}}{R}  \right)^{1+\beta}  \mathsf{M}_R
.
\end{equation*}
By the triangle inequality, we obtain 
\begin{equation*} 
\left\|  \nabla \psi_{m+1} - \nabla \psi_m   \right\|_{\underline{L}^2 \left( B_{r_{m}} \right)} 
\leq 
C \left( \frac{r_{m}}{R}  \right)^{1+\beta }  \mathsf{M}_R
.
 \end{equation*}
Since $\psi_{m+1} -  \psi_m \in \mathcal{A}_2[\phi]$, we have that, for $r \in [r_m , R]$, 
\begin{equation*} 
\left\|  \nabla \psi_{m+1} - \nabla \psi_m   \right\|_{\underline{L}^2 \left( B_{r} \right)}  
\leq C \left(\frac r{r_m}\right) \left\|  \nabla \psi_{m+1} - \nabla \psi_m   \right\|_{\underline{L}^2 \left( B_{r_{m}} \right)} .
\end{equation*}
Combining, for $r \in [r_m , R]$, 
\begin{equation*} 
\left\|  \nabla \psi_{m+1} - \nabla \psi_m   \right\|_{\underline{L}^2 \left( B_{r} \right)} 
\leq 
C  \left(\frac {r_m}{r}\right)^{\beta}   
\left( \frac{r}{R}  \right)^{1+\beta} \mathsf{M}_R
. 
\end{equation*}
Thus, by summing over the scales, we obtain
\begin{equation*} 
\left\|  \nabla \psi_{m} - \nabla \psi_{n}   \right\|_{\underline{L}^2 \left( B_{r_m} \right)} \leq C 
\left( \frac{r_m}{R}  \right)^{1+\beta} \mathsf{M}_R
. 
\end{equation*}
The triangle inequality thus yields 
\begin{equation*} 
\left\|  \nabla z - \nabla \psi_n   \right\|_{\underline{L}^2 \left( B_{r_m} \right)} 
\leq 
C \left( \frac{r_m}{R}  \right)^{1+\beta} \mathsf{M}_R
,
\end{equation*}
from which~\eqref{e.C2} follows easily by taking $\psi = \psi_n$.

\smallskip

\emph{Step 6.} Conclusion. We show that there exist a constant $C(\sigma,\mathsf{M},\data)<\infty$ and~$\tilde \phi \in \mathcal{L}_1$ such that 
\begin{equation}  \label{e.C11here}
\left\| \nabla u -\nabla \tilde \phi \right\|_{\underline{L}^2 \left( B_{r} \right)}  \leq C \left(\frac rR \right) \mathsf{M}_R
. 
\end{equation}
Let $\psi$ be as in Step 5. By Theorem~\ref{t.regularitylin}, we find $p\in \overline{\A}_2[\phi]$  and $\tilde p := \frac12 \left( \nabla^2 p \right) x^{\otimes 2}$, correspondingly, such that 
\begin{equation*} 
\left\| \psi - p \right\|_{\underline{L}^2 \left( B_{r} \right)} \leq C r^{-\beta}\left\| p \right\|_{\underline{L}^2 \left( B_{r} \right)} 
\quad \mbox{and} \quad
\left\| \tilde \psi - \tilde p \right\|_{\underline{L}^2 \left( B_{r} \right)} \leq C r^{-\beta} \left\| \tilde p \right\|_{\underline{L}^2 \left( B_{r} \right)} 
.
\end{equation*}
Moreover, we have that 
\begin{equation*} 
 \left\|  \tilde p \right\|_{\underline{L}^2 \left( B_{R} \right)} 
    \leq C  \left\|   p \right\|_{\underline{L}^2 \left( B_{R} \right)} 
    \leq C R  \left\|  \nabla \psi \right\|_{\underline{L}^2 \left( B_{R} \right)} 
    \leq 
    C R\mathsf{M}_R    
    . 
\end{equation*}
We now obtain, for $r \in [ \X, R]$,  that
\begin{align} \label{e.C11psidoubling}
\left\| \nabla \tilde \psi \right\|_{\underline{L}^2(B_r)} 
\leq C \left(\frac rR \right) \left\| \nabla \tilde \psi \right\|_{\underline{L}^2(B_R)} \leq 
C \left(\frac rR \right) \mathsf{M}_R  
.
\end{align}
Furthermore, clearly $\psi - \tilde \psi  \in \A_1[\phi]$ and, by the above inequality,~\eqref{e.C2} and~\eqref{e.C11pre}, 
\begin{align} \notag 
\left\| \nabla (\psi - \tilde \psi)  \right\|_{\underline{L}^2 \left( B_{r} \right)} 
& 
\leq 
\left\| \nabla z - \nabla \psi  \right\|_{\underline{L}^2 \left( B_{r} \right)} + 
 \left\| \nabla z  \right\|_{\underline{L}^2 \left( B_{r} \right)} + 
\left\| \nabla \tilde \psi \right\|_{\underline{L}^2 \left( B_{r} \right)} 
\leq
C \left(\frac rR \right)^{\alpha} \mathsf{M}_R  .
\end{align}
Lemma~\ref{l.tildephivsphi1} then yields that there is $\tilde \phi \in \mathcal{L}_1$ such that 
\begin{equation*} 
\left\| \nabla \tilde \phi -\nabla \phi + \nabla (\psi - \tilde \psi) \right\|_{\underline{L}^2 \left( B_{r} \right)} 
\leq 
C  \left(\left\| \nabla (\psi - \tilde \psi )\right\|_{\underline{L}^2 \left( B_{r} \right)}  \right)^{1+\delta},
\end{equation*}
and thus, since $\mathsf{M}_R \leq C$,
\begin{equation}  \label{e.C11correctphioncemore}
\left\| \nabla \tilde \phi -\nabla \phi + \nabla (\psi - \tilde \psi) \right\|_{\underline{L}^2 \left( B_{r} \right)}  
\leq 
C \left(\frac rR \right)^{1+\beta} \mathsf{M}_R  
.
\end{equation}
By the triangle inequality we deduce that 
\begin{equation*} 
\left\| \nabla u -\nabla \tilde \phi \right\|_{\underline{L}^2 \left( B_{r} \right)} 
\leq 
\left\| \nabla (u  - \phi -\psi) \right\|_{\underline{L}^2 \left( B_{r} \right)} 
+ 
\left\| \nabla( \tilde \phi -(\phi + \psi - \tilde \psi)) \right\|_{\underline{L}^2 \left( B_{r} \right)} 
+ 
\left\| \nabla \tilde \psi \right\|_{\underline{L}^2 \left( B_{r} \right)}.
\end{equation*}
The terms on the right can be estimated by~\eqref{e.C2},~\eqref{e.C11correctphioncemore} and~\eqref{e.C11psidoubling}, respectively.
We thus obtain~\eqref{e.C11here} and the proof is now complete.  
\end{proof}

\begin{remark} 
\label{r.C2beta}
In the course of proving~\eqref{e.C11}, we also obtain an estimate which can be interpreted as large-scale~$C^{2,\beta}$-type estimate. Indeed, we showed that there exist $\beta(\data)$ and $C(\sigma,\mathsf{M},\data)<\infty$ such that the following holds. Let $R \geq \X$ and suppose that $u \in \L(B_R)$ satisfying~$\frac1R \left\| u - (u)_{B_R} \right\|_{\underline{L}^2 \left( B_{R} \right)}  \leq \mathsf{M}$. Then there exists $\phi \in \mathcal{L}_1$ and $\psi \in \A_2[\phi]$ such that, for $r \in [\X,R]$, 
\begin{equation}  \label{e.C2beta-remark}
\left\|  \nabla u - \nabla \phi -  \nabla \psi   \right\|_{\underline{L}^2 \left( B_{r} \right)} \leq C \left( \frac{r}{R}  \right)^{1+\beta} \inf_{\xi \in \mathcal{L}_1} \frac1R \left\| u - \xi - (u - \xi)_{B_R} \right\|_{\underline{L}^2(B_{R})}. 
\end{equation}
\end{remark}


\section{Improved regularity of the homogenized Lagrangian}
\label{s.Lreg}

In this section, we prove Theorem~\ref{t.Lreg}. In addition to the hypotheses stated in Section~\ref{ss.assump}, we suppose that there exists an exponent $\gamma \in (0,1]$ such that
\begin{equation}
\label{e.extra.assump}
\P \left[
\sup_{p\in\Rd} \left(  \left[ \frac{D_pL(p,\cdot)}{1+|p|} \right]_{C^{0,\gamma}(\Rd)} + \left[ D^2_pL(p,\cdot) \right]_{C^{0,\gamma}(\Rd)}  \right) \leq \mathsf{M}_0
\right] = 1. 
\end{equation}
In particular, unlike the rest of the paper, here we assume $L(p,\cdot)$ has some regularity on the unit scale. 
A condition like this is needed to  control the very small scales and allow us to convert the large-scale $C^{0,1}$ estimate for differences into pointwise estimates, which seems to be necessary in order to improve the scaling of the linearization error. 

\begin{proposition}[Pointwise gradient estimate for solutions]
\label{p.pointwise.solutions}
Assume~\eqref{e.extra.assump} holds for some $\gamma\in (0,1]$. Fix $\sigma\in (0,d)$ and $\mathsf{M} \in [1,\infty)$. Let $\X$ be the random variable in Theorem~\ref{t.AS.regularity}. Then there exists $\beta(\gamma,\data)\in\left(0,\tfrac12\right]$ such that, for every $R \geq 4$ and $u\in H^1(B_R)$ satisfying
\begin{equation*}
\left\{
\begin{aligned}
& -\nabla \cdot \left( D_pL(\nabla u,x) \right) = 0 & \mbox{in} & \ B_R, \\
& \left\| \nabla u \right\|_{\underline{L}^2(B_R)}
\leq \mathsf{M},
\end{aligned}
\right.
\end{equation*}
we have the estimate 
\begin{equation}
\left\| \nabla u \right\|_{C^{0,\beta}(B_1)}
\leq 
C \X^{\frac d2} R^{-1} \left\| u - \left( u \right)_{B_R} \right\|_{\underline{L}^2(B_R)}.
\end{equation}
\end{proposition}
\begin{proof}
By the Proposition~\ref{p.C1beta.coeff}, there exists $\beta(\gamma,\data)\in \left(0,\tfrac12\right]$ such that
\begin{equation*}
\left\| \nabla u \right\|_{C^{0,\beta\wedge\gamma}(B_1)}
\leq 
C\left\| \nabla u \right\|_{L^2(B_2)}. 
\end{equation*}
Giving up a volume factor and applying Theorem~\ref{t.AS.regularity}, we have 
\begin{equation*}
\left\| \nabla u \right\|_{ L^2(B_2)}
\leq 
C \left( R \wedge \X \right)^{\frac d2} 
\left\| \nabla u \right\|_{ L^2\left(B_{\X\wedge R}\right)}
\leq 
C \X^{\frac d2} R^{-1} 
\left\| u - \left( u \right)_{B_R} \right\|_{\underline{L}^2(B_R)}. \qedhere
\end{equation*}
\end{proof}

\begin{proposition}[Pointwise gradient estimate for differences]
\label{p.pointwise.differences}
Assume that~\eqref{e.extra.assump} holds for some $\gamma \in (0,1]$. Let $\mathsf{M} \in [1,\infty)$. Then there exist $\alpha(\gamma,\data),\delta(\gamma,\data)\in \left(0,\tfrac12\right]$, $C(\gamma,\mathsf{M},\data)<\infty$ and a random variable $\X$ satisfying
\begin{equation*}
\X = \O_\delta \left(C \right)
\end{equation*}
such that the following holds.  For every~$R \geq 4$ and $u,v\in H^1(B_R)$ satisfying 
\begin{equation*}
\left\{
\begin{aligned}
& -\nabla \cdot \left( D_pL(\nabla u,x) \right) = 0 & \mbox{in} & \ B_R, \\
& -\nabla \cdot \left( D_pL(\nabla v,x) \right) = 0 & \mbox{in} & \ B_R, \\
& \left\| \nabla u \right\|_{\underline{L}^2(B_R)}, \, 
\left\| \nabla v \right\|_{\underline{L}^2(B_R)} \leq \mathsf{M},
\end{aligned}
\right.
\end{equation*}
we have the estimate
\begin{equation*}
\left\| \nabla u- \nabla v \right\|_{C^{0,\alpha}\left(B_{1/2}\right)}
\leq 
\frac \X R \left\| u-v \right\|_{\underline{L}^2(B_R)}. 
\end{equation*}
\end{proposition}
\begin{proof}
Pick any $\sigma \in (0,d)$ and let $\X$ be the maximum of the random variables in Theorems~\ref{t.regularity.differences} and~\ref{t.AS.regularity}. Let $u,v\in H^1(B_R)$ be as in the statement of the proposition and observe that the difference $w:= u-v$ satisfies the equation
\begin{equation*}
\nabla \cdot \left( \a \nabla w \right) = 0 
\quad \mbox{in} \ B_R
\end{equation*}
with  
\begin{equation*}
\a(x):= \int_0^1 D^2_pL(t\nabla u(x)+(1-t)\nabla v(x),x)\,dt
\end{equation*}
which satisfy, by~\eqref{e.extra.assump} and the previous lemma, 
\begin{align*}
\left[ \a \right]_{C^{0,\gamma(\beta\wedge\gamma')}(B_1)}
&
\leq 
\mathsf{M}_0^\gamma + C \mathsf{M}_0 \left(  \left[ \nabla u \right]_{C^{0,\beta\wedge\gamma'}(B_1)} + \left[ \nabla v \right]_{C^{0,\beta\wedge\gamma'}(B_1)} \right)^{\gamma}
\\ & 
\leq 
\mathsf{M}_0^\gamma + C \mathsf{M}_0
\left( \X^{\frac d2} \mathsf{M} \right)^\gamma
\\ & 
\leq C \mathsf{M}_0
\left( \X^{\frac d2} \mathsf{M} \right)^\gamma.
\end{align*}
The Schauder estimates imply that 
\begin{equation*}
\left\| \nabla w \right\|_{L^\infty(B_{1/2})} 
+
\left[ \nabla w \right]_{C^{0,\gamma(\beta\wedge\gamma')}(B_{1/2})}
\leq 
C \left[ \a \right]_{C^{0,\gamma(\beta\wedge\gamma')}(B_1)}^{\frac{d}{2\gamma(\beta\wedge\gamma')}}
\left\| \nabla w \right\|_{L^2(B_1)}. 
\end{equation*}
By Theorem~\ref{t.regularity.differences}, 
\begin{equation*}
\left\| \nabla w \right\|_{L^2(B_1)} 
\leq 
C \X^{\frac d2} \left\| \nabla w \right\|_{L^2(B_{\X})}
\leq 
C \X^{\frac d2} R^{-1} \left\| w \right\|_{L^2(B_{R})}.
\end{equation*}
Putting these together, we find that 
\begin{equation*}
\left\| \nabla w \right\|_{L^\infty(B_{1/2})} 
+
\left[ \nabla w \right]_{C^{0,\gamma(\beta\wedge\gamma')}(B_{1/2})}
\leq
C \left( \mathsf{M}_0^{\frac1\gamma} \X^{\frac d2} \mathsf{M}\right)^{\frac{d}{2(\beta\wedge\gamma')}}
\X^{\frac d2}  R^{-1} \left\| w \right\|_{L^2(B_{R})}.
\end{equation*}
Allowing the constant~$C$ to depend on $(\gamma,\gamma',\mathsf{M}_0,\mathsf{M},\data)$, we obtain, for a large exponent $q(\gamma',\data)\in (1,\infty)$,
\begin{equation*}
\left\| \nabla w \right\|_{L^\infty(B_{1/2})} 
+
\left[ \nabla w \right]_{C^{0,\gamma(\beta\wedge\gamma')}(B_{1/2})}
\leq
C \X^q  R^{-1} \left\| w \right\|_{L^2(B_{R})}.
\end{equation*}
This completes the proof. 
\end{proof}
%
%
%
We now give the proof of Theorem~\ref{t.Lreg}. 

\begin{proof}[{Proof of Theorem~\ref{t.Lreg}}]
Let $\phi_\xi$ denote the first-order corrector for the nonlinear equation with slope~$\xi$, and $\psi_{\xi,\eta}$ the first-order corrector for the linearized equation around $x\mapsto x+\phi_{\xi}(x)$ with slope $\eta$. In other words, for every $\xi,\eta\in\Rd$, the gradient fields~$\nabla \phi_\xi$ and~$\nabla \psi_{\xi,\eta}$ are $\Zd$--stationary, have mean zero and satisfy
\begin{equation}
-\nabla \cdot \left( D_pL(\xi + \nabla \phi_\xi(x),x) \right) = 0 \quad \mbox{in} \ \Rd
\end{equation}
and 
\begin{equation}
-\nabla \cdot \left( D_p^2 L(\xi + \nabla \phi_\xi(x),x) \left( \eta+ \nabla \psi_{\xi,\eta}(x) \right) \right) = 0 \quad \mbox{in} \ \Rd.
\end{equation}
One can see from the proof of Proposition~\ref{p.nu.C2beta} or alternatively, by differentiating the equation for $\phi_\xi$ in $\xi$, that 
\begin{equation*}
\partial_{\xi_i} \nabla \phi_{\xi} = \nabla \psi_{\xi,e_i}. 
\end{equation*}
We also have the formula 
\begin{equation*}
D^2\overline{L} (\xi) \eta
=
\E \left[ 
\int_{[0,1]^d} 
D_p^2 L\left(\xi + \nabla \phi_\xi(x),x\right) \left( \eta+ \nabla \psi_{\xi,\eta}(x)\right)\,dx
\right],
\end{equation*}
together with the estimates
\begin{equation} \label{e.phidiffvarest}
\E \left[  \left\| \nabla \phi_\xi - \nabla \phi_{\xi'}  \right\|_{L^2 \left( [0,1]^d \right)}^2 \right] \leq C |\xi-\xi'|^2 
\quad \mbox{and} \quad
\E \left[ \left\| \nabla \psi_{\xi,\eta} \right\|_{L^2 \left( [0,1]^d \right)}^2 \right] \leq C |\eta|^2 . 
\end{equation}
Thus 
\begin{align*}
\lefteqn{
\left| \left( D^2\overline{L} (\xi) - D^2\overline{L} (\xi') \right) \eta \right| 
} \quad & 
\\ & \notag
\leq
\E \left[ \int_{[0,1]^d} \left| D_p^2 L\left(\xi + \nabla \phi_{\xi}(x),x\right) -D_p^2 L\left(\xi' + \nabla \phi_{\xi'}(x),x\right) \right| \left|  \eta+ \nabla \psi_{\xi,\eta}(x)\right| \,dx  \right] 
\\ & \notag \quad 
+ 
\E \left[ 
\int_{[0,1]^d} 
\left| D_p^2 L\left(\xi' + \nabla \phi_{\xi'}(x),x\right) \right| \left|  \nabla \psi_{\xi,\eta}(x) -  \nabla \psi_{\xi',\eta}(x) \right| \,dx
\right].
\end{align*}
To estimate the first term on the right side, we observe that, by~\eqref{e.extra.assump}, H\"older's inequality and~\eqref{e.phidiffvarest},
\begin{align*}
\lefteqn{
\E \left[ \int_{[0,1]^d} \left| D_p^2 L\left(\xi + \nabla \phi_{\xi}(x),x\right) -D_p^2 L\left(\xi' + \nabla \phi_{\xi'}(x),x\right) \right| \left|  \eta+ \nabla \psi_{\xi,\eta}(x)\right| \,dx  \right] 
} \qquad & 
\\ & 
\leq
C \E \left[ \int_{[0,1]^d}  \left| \left(\xi+\nabla \phi_{\xi}(x)\right) - \left(\xi'+ \nabla \phi_{\xi'}(x)\right) \right|^{\gamma} 
\left|  \eta+ \nabla \psi_{\xi,\eta}(x)\right| \,dx  
\right] 
\\ & 
\leq 
C
\E \left[ \int_{[0,1]^d} \left(|\xi - \xi'| +  \left| \nabla \phi_{\xi'}(x) -  \nabla \phi_\xi(x) \right| \right)^{2} \,dx \right]^{\frac {\gamma}2} 
\\ &  \quad 
\times  
\E \left[ \int_{[0,1]^d} \left|  \eta + \nabla \psi_{\xi,\eta}(x)\right|^2 \,dx   \right]^{\frac12} 
\\ & 
\leq 
C \left| \xi-\xi'\right|^{\gamma} \left| \eta \right|. 
\end{align*}
To estimate the other term, we have that 
\begin{multline*}
\E \left[ 
\int_{[0,1]^d} 
\left| D_p^2 L\left(\xi' + \nabla \phi_{\xi'}(x),x\right) \right| \left|  \nabla \psi_{\xi,\eta}(x) -  \nabla \psi_{\xi',\eta}(x) \right| \,dx
\right]
\\
\leq 
C \E \left[ 
\int_{[0,1]^d} 
\left|  \nabla \psi_{\xi,\eta}(x) -  \nabla \psi_{\xi',\eta}(x) \right|^2 \,dx
\right]^{\frac12}. 
\end{multline*}
Therefore to complete the proof, it suffices to show that 
\begin{equation}
\label{e.Lbarregwts}
\E \left[ 
\int_{[0,1]^d} 
\left|  \nabla \psi_{\xi,\eta}(x) -  \nabla \psi_{\xi',\eta}(x) \right|^2 \,dx
\right]
\leq 
C \left| \eta \right|^2 \left| \xi - \xi' \right|^{2\gamma}.
\end{equation}
To prove~\eqref{e.Lbarregwts}, we argue similarly as in Lemma~\ref{l.deterministic.linearization}.
Let $\zeta:=  \psi_{\xi,\eta} - \psi_{\xi',\eta}$ so that $\nabla \zeta$ is a $\Zd$--stationary gradient field satisfying
\begin{equation*}
-\nabla \cdot \left( D_p^2 L(\xi + \nabla \phi_\xi(x),x) \nabla \zeta \right) = 
\nabla \cdot \mathbf{f},
 \quad \mbox{in} \ \Rd
\end{equation*}
 where $\mathbf{f}$ is defined by
\begin{equation*}
 \mathbf{f} 
 =
 \left( D_p^2 L(\xi + \nabla \phi_\xi(x),x) - D_p^2 L(\xi' + \nabla \phi_{\xi'}(x),x) \right) \nabla \psi_{\xi',\eta}. 
\end{equation*}
We have that 
\begin{align*}
\lefteqn{
\E \left[ \int_{[0,1]^d} \left| \nabla \zeta(x) \right|^2\,dx \right] 
} \quad & 
\\ &
\leq
\E \left[ \int_{[0,1]^d} \left| \mathbf{f}(x) \right|^2\,dx \right] 
\\ & 
\leq 
\E \left[ \int_{[0,1]^d}  \left| \left(\xi+\nabla \phi_{\xi}(x)\right) - \left(\xi'+ \nabla \phi_{\xi'}(x)\right) \right|^{2\gamma} 
\left| \nabla \psi_{\xi',\eta}(x)\right|^2
\,dx \right] 
\\ & 
\leq 
\E \left[ \int_{[0,1]^d}  \left| \left(\xi+\nabla \phi_{\xi}(x)\right) - \left(\xi'+ \nabla \phi_{\xi'}(x)\right) \right|^{\frac{2\gamma(2+\delta)}{\delta}} \right]^{\frac{\delta}{2+\delta}}
\E \left[ \int_{[0,1]^d} \left| \nabla \psi_{\xi',\eta} \right|^{2+\delta}\,dx \right]^{\frac2{2+\delta}}.
\end{align*}
By Proposition~\ref{p.pointwise.differences} and the argument for~\eqref{e.yesdiffcorrectors}, 
\begin{align*}
\E \left[ \int_{[0,1]^d}  \left| \left(\xi+\nabla \phi_{\xi}(x)\right) - \left(\xi'+ \nabla \phi_{\xi'}(x)\right) \right|^{\frac{2\gamma(2+\delta)}{\delta}} \right]^{\frac{\delta}{2+\delta}}
& 
\leq C \left| \xi-\xi' \right|^{2\gamma}. 
\end{align*}
By the Meyers estimate and Proposition~\ref{p.C01.linearizedeq}, noticing that the latter implies bounds on $\nabla \psi_{\xi',\eta}$ analogous to~\eqref{e.yesdiffcorrectors} by the same argument, we get, for $\X$ as in Proposition~\ref{p.C01.linearizedeq}, 
\begin{align*}
\E \left[ \int_{[0,1]^d} \left| \nabla \psi_{\xi',\eta} \right|^{2+\delta}\,dx \right]^{\frac2{2+\delta}}
&
\leq
\E \left[  \left( \int_{[-1,2]^d} \left| \nabla \psi_{\xi',\eta} \right|^{2}\,dx \right)^{\frac{2+\delta}{2}} \right]^{\frac2{2+\delta}}
\\ &
\leq 
C \left|\eta\right|^2 \E \left[ \X^{\frac{d}{2}(2+\delta)} \right]
\leq
C\left| \eta \right|^2. 
\end{align*}
Combining the previous three displays yields ~\eqref{e.Lbarregwts} and completes the proof. 
\end{proof}

\appendix

\section{Regularity for constant Lagrangians}
\label{s.regconstantL}

In this appendix we recall some classical regularity estimates for solutions of constant-coefficient equations, summarized in the following proposition. Versions of the results stated here can be found in books such as~\cite{Giu}, but for our purposes in this paper we require somewhat sharper, more quantitative statements which we could not find in the literature. For this reason we give complete proofs here. 

\begin{proposition}[{$C^{1,\beta}$ regularity for differences}]
\label{p.C1beta.differences.Lbar}
Fix $\gamma \in (0,1)$, $R\in (0,\infty)$, and $\, \mathsf{K},\mathsf{M} \in (0,\infty)$.  Let $F:\Rd  \to \R$ be a Lagrangian satisfying $\left[ F \right]_{C^{2,\gamma}} \leq \mathsf{K}$ and, for every $p \in \R^d$, 
\begin{equation} \label{e.F.uniconvex}
I_d \leq D^2_p F(p) \leq \Lambda I_d.
\end{equation}
Suppose that $u$ and $v$ are local $F$-minimizers in $B_R$. Then $u,v\in C^{2,\gamma}(B_{R/2})$ and the following statements hold. 
\begin{itemize} 

\item There exist~$\theta(d,\Lambda)\in (0,1)$ and $C(d,\Lambda)<\infty$ such that, for every $x,y \in B_{R/2}$,
\begin{equation}  \label{e.C1alpha.Lbar}
\left| \nabla u(x) - \nabla u(y)  \right|  \leq C \left( \frac{|x-y|}{R}\right)^\theta \frac{1}{R} \inf_{\ell \in \mathcal{P}_1}  \left\|  u -  \ell \right\|_{\underline{L}^2(B_{R})}.
\end{equation}

\item 
There exists a constant $C(\mathsf{K},\gamma,d,\Lambda)<\infty$ and $\eta(d,\Lambda) < \infty$ such that
\begin{equation}  \label{e.ass.Lbar}
\frac1R \inf_{\ell \in \mathcal{P}_1}\left\| u - \ell \right\|_{\underline{L}^2(B_R)}  \leq \mathsf{M}
\end{equation}
implies
\begin{equation}  \label{e.C11.Lbar}
R \left\| \nabla^2 u \right\|_{L^\infty(B_{R/2})}   \leq C \left( \mathsf{M} \vee 1 \right)^{\eta}  \mathsf{M}
\end{equation}
and, for every $x,y \in B_{R/4}$,
\begin{multline}  \label{e.C2alpha.Lbardiff}
 R \left| \nabla^2 u(x) -  \nabla^2 u(y)  \right| 
 \\ \leq 
 C \left( \frac{|x-y|}{R} \right)^{\gamma}  \left( \inf_{q \in \mathcal{P}_2}\left\| \nabla u - \nabla q \right\|_{\underline{L}^2(B_{R/2})} +  \left( \mathsf{M} \vee 1 \right)^{\eta}  \mathsf{M}^{1+\gamma} \right).
\end{multline}

\item Suppose that 
\begin{equation}  \label{e.ass.Lbardiff}
\frac1R \inf_{\ell \in \mathcal{P}_1}\left\| u - \ell \right\|_{\underline{L}^2(B_R)} + \frac1R \inf_{\ell \in \mathcal{P}_1} \left\| v - \ell \right\|_{\underline{L}^2(B_R)} \leq \mathsf{M}. 
\end{equation}
Then there exist constants  $C(\mathsf{K},\gamma,d,\Lambda)<\infty$ and $\eta(d,\Lambda) < \infty$ such that, for any $s\in \left(0,\frac{R}{2}\right]$, we have both
\begin{equation}  \label{e.Lip.Lbardiff}
\left\| \nabla u - \nabla v \right\|_{\underline{L}^\infty(B_{R/2})}  \leq  C \left(  \mathsf{M} \vee 1 \right)^{\eta} \frac 1R \left\| u -v - (u-v)_{B_R}\right\|_{\underline{L}^2(B_R)} 
\end{equation}
and, for $r \in (0,s]$, 
\begin{multline} \label{e.C1gamma.Lbardiff}
\inf_{\ell \in \mathcal{P}_1}\left\| u - v - \ell \right\|_{\underline{L}^2(B_{r})}
\leq C\left(\frac rs \right)^2 \inf_{\ell \in \mathcal{P}_1}\left\| u - v - \ell \right\|_{\underline{L}^2(B_{s})} 
\\ + C \left( \mathsf{M}\vee 1 \right)^{\eta}  \left(\frac{s}{r}\right)^{\frac d2}  \left( \frac{s}{R} \right)^{1+\gamma} \mathsf{M}^\gamma \left\| u-v  - (u-v)_{B_R}\right\|_{\underline{L}^2(B_R)}.
\end{multline}
\end{itemize}
\end{proposition}

\begin{remark} \label{r.C1beta.differences.Lbar}
For the estimate~\eqref{e.C1alpha.Lbar} it is enough to assume~\eqref{e.F.uniconvex}. 
In fact, is suffices to assume that $F \in C^{1}$ and, for every $p_1,p_2 \in \R^d$, 
\begin{equation*} 
|p_1-p_2|^2 \leq \left( D_p F(p_1) - D_p F(p_2) \right)  \cdot (p_1-p_2)  \leq \Lambda |p_1-p_2|^2.
\end{equation*}
\end{remark}

\begin{proof}[Proof of Proposition~\ref{p.C1beta.differences.Lbar}]
We divide the proof into five steps. In the first step we will prove~\eqref{e.C1alpha.Lbar}, and in the second both~\eqref{e.ass.Lbardiff}. In Step 3 we will prove~\eqref{e.C11.Lbar} using~\eqref{e.Lip.Lbardiff},  and in Step 4 we will show~\eqref{e.C1gamma.Lbardiff}. Finally, in the last step, we will prove~\eqref{e.C2alpha.Lbardiff} using~\eqref{e.Lip.Lbardiff} and~\eqref{e.C1gamma.Lbardiff}. 

\smallskip

\emph{Step 1}. We first prove~\eqref{e.C1alpha.Lbar}. 
Due to the first variation and smoothness of $F$, $u$ satisfies the equation
\begin{equation*} 
\nabla \cdot \left( D_p^2 F(\nabla u) \nabla^2 u \right) = 0.
\end{equation*}
By~\eqref{e.F.uniconvex}, we may apply the classical De Giorgi-Nash-Moser theory to obtain that $\nabla u \in C^{0,\theta}$ and, in particular, 
\begin{equation*} 
\sup_{x,y\in B_r} \left| \nabla u(x) - \nabla u(y)  \right|  \leq C \left( \frac{r}{R}\right)^\theta \left\| \nabla u - (\nabla u)_{B_{R/2}} \right\|_{\underline{L}^2(B_{R/2})} .
\end{equation*}
Furthermore, for all $\ell \in \mathcal{P}_1$ we have
\begin{equation*} 
\nabla \cdot \left( D_p F(\nabla u) - D_p F(\nabla \ell) \right) = 0,
\end{equation*}
and again~\eqref{e.F.uniconvex} provides us a Caccioppoli inequality
\begin{equation}  \label{e.Cacc2nd.Lbar}
\left\| \nabla u - (\nabla u)_{B_{R/2}} \right\|_{\underline{L}^2(B_{R/2})} 
= \inf_{\ell \in \mathcal{P}_1}  \left\| \nabla u - \nabla \ell \right\|_{\underline{L}^2(B_{R/2})} 
\leq \frac{C}{R} \inf_{\ell \in \mathcal{P}_1}  \left\|  u -  \ell \right\|_{\underline{L}^2(B_{R})}.
\end{equation}
Combining the estimates proves~\eqref{e.C1alpha.Lbar}.

\smallskip

\emph{Step 2}.
We next prove~\eqref{e.Lip.Lbardiff}. The difference $u-v$ satisfies the equation 
\begin{equation*} 
\nabla \cdot \left( \b(x) \nabla (u-v) \right) = 0,
\end{equation*}
where
\begin{equation*} 
\b(x) := \int_{0}^1 D_p^2 F \left(t \nabla u(x) + (1-t) \nabla v(x)  \right) \, dt. 
\end{equation*}
Observe that, by~\eqref{e.F.uniconvex}, $I_d \leq \b(x) \leq \Lambda I_d$.  By freezing the coefficients at the origin we get that 
\begin{equation*} 
\nabla \cdot \left( \b(0) \nabla (u-v) \right)  = \nabla \cdot \left( (\b(0) - \b(x)) \nabla (u-v) \right)  .
\end{equation*}
By the smoothness of $F$ and~\eqref{e.C1alpha.Lbar} (applied for both $u$ and $v$), we get, for all $x \in B_{R/2}$, 
\begin{equation*} 
\left| \b(0) - \b(x)\right| \leq C \mathsf{K} \left( \frac{|x|}{R} \right)^{\theta \gamma} \mathsf{M}^\gamma
\end{equation*}
Fix $s \in \left(0,\frac R2 \right)$ and  let $w_s \in u-v + H_0^1(B_s)$ solve $\nabla \cdot \left( \b(0) \nabla w_s \right) = 0$ in $B_s$. Then
\begin{equation*} 
\left\| \nabla (u-v) - \nabla w_s \right\|_{\underline{L}^2(B_s)} \leq C  \left( \frac{s}{R} \right)^{\theta \gamma} \mathsf{M}^\gamma \left\| \nabla (u-v) \right\|_{\underline{L}^2(B_s)} . 
\end{equation*}
Since $w_s$ satisfies the equation with constant coefficients, we find a constant $C(\mathsf{K},d,\Lambda)$ such that, for any $r \in (0,1)$, 
\begin{equation*} 
\inf_{\ell \in \mathcal{P}_1}\left\| w_s - \ell \right\|_{\underline{L}^2(B_{ r})} \leq C \left(\frac rs \right)^2  \inf_{\ell \in \mathcal{P}_1}\left\| w_s - \ell \right\|_{\underline{L}^2(B_{s})} . 
\end{equation*}
By Poincar\'e's inequality, the Caccioppoli inequality for $u-v$, and the triangle inequality, the above two displays imply 
\begin{multline*} 
\inf_{\ell \in \mathcal{P}_1}\left\| u - v - \ell \right\|_{\underline{L}^2(B_{r})}
\leq C\left(\frac rs \right)^2 \inf_{\ell \in \mathcal{P}_1}\left\| u - v - \ell \right\|_{\underline{L}^2(B_{s})} 
\\ + C \left(\frac{s}{r}\right)^{\frac d2}  \left( \frac{r}{R} \right)^{\theta \gamma} \mathsf{M}^\gamma \left\| u-v  - (u-v)_{B_s}\right\|_{\underline{L}^2(B_s)}.
\end{multline*}
Setting
\begin{equation*} 
E_1(s) := \left( \frac{R}{s} \right)^{1+\theta \gamma} \inf_{\ell \in \mathcal{P}_1}\left\| u - v - \ell \right\|_{\underline{L}^2(B_{s})} , \quad 
E_0(s) := \frac{R}{s} \left\| u-v  - (u-v)_{B_s}\right\|_{\underline{L}^2(B_s)},
\end{equation*}
we find a small constant $\ep(\mathsf{K},d,\Lambda) \in (0,1)$ such that 
\begin{equation*} 
E_1(\ep s) \leq \frac 12 E_1(s) + C  \mathsf{M}^\gamma E_0(s)
\end{equation*}
This implies, after taking supremum and reabsorbing, that 
\begin{equation}  \label{e.Lbar.supE_1}
\sup_{s \in (0,R)} E_1(s) \leq CE_1(R) + C  \mathsf{M}^\gamma  \sup_{s \in (0,R)}  E_0(s) . 
\end{equation}
Notice that the term on the right is finite since both $u$ and $v$ are Lipschitz continuous. 
Letting $\ell_s$ be the affine function realizing the infimum in the definition of $E_1(s)$, we obtain by the triangle inequality that
\begin{multline*} 
| \nabla \ell_{s} - \nabla \ell_{2s}| \leq C \frac 1R \left( \frac sR \right)^{\theta \gamma} \left( E_1(s)  + E_1(2s)\right) 
\\ \leq 
C \frac 1R \left( \frac sR \right)^{\theta \gamma} \left(E_1(R) +  \mathsf{M}^\gamma  \sup_{s \in (0,R)}  E_0(s) \right).
\end{multline*}
After telescoping, we deduce, for $s \in (0,R)$, the estimate
\begin{equation*} 
\sup_{r \in (0,s)} | \nabla \ell_{r}| \leq | \nabla \ell_{s}| + C \frac 1R \left( \frac sR \right)^{\theta \gamma} \left(E_1(R) +  \mathsf{M}^\gamma  \sup_{s \in (0,R)}  E_0(s) \right).
\end{equation*}
Using the trivial estimate $| \nabla \ell_{s}| \leq CR^{-1} E_0(s)$, we also get, for any $s \in (0,R]$, 
\begin{multline*} 
\sup_{r \in (0,s]} E_0(r) \leq R \sup_{r \in (0,s]} | \nabla \ell_{r}| + \left( \frac sR \right)^{\theta \gamma} \sup_{r \in (0,s]} E_1(s)  
\\ \leq C E_0(s) + C \left( \frac sR \right)^{\theta \gamma} \left(E_1(R) + \mathsf{M}^\gamma  \sup_{s \in (0,R]}  E_0(s) \right)
\end{multline*}
Choosing, in particular, $s = s^*$, where
\begin{equation*} 
s^* := R \left[  2C \mathsf{M}^{\gamma} \vee 1 \right]^{-\frac 1{\theta \gamma}},
\end{equation*}
and using $E_1(R)\leq E_0(R)$ and 
\begin{equation*} 
\sup_{r\in (s^*,R]} E_0(r)  \leq \left(\frac{R}{s^*}\right)^{1+ \frac d2} E_0(R) = \left(  2C  \mathsf{M}^\gamma \vee 1 \right)^{\frac{d+2}{2\theta\gamma}} E_0(R), 
\end{equation*}
we can reabsorb $\sup_{r \in (0,s^*)} E_0(r)$ from the right to get
\begin{equation*} 
\sup_{s \in (0,R)} E_0(s) \leq  C\left( \mathsf{M}\vee 1 \right)^{\frac{d+2}{2\theta}} E_0(R).
\end{equation*}
Thus,~\eqref{e.Lip.Lbardiff} follows. 

\smallskip

\emph{Step 3}. We prove~\eqref{e.C11.Lbar}. For fixed $h\in \R^d$ with $|h|\ll R$, we set $v(x) := u(x+h)$, which is still $F$-minimizer in $B_{R-|h|}$. Moreover,~\eqref{e.ass.Lbar} gives, for small enough $|h|$,
\begin{equation*} 
\sup_{y \in B_{R/2-|h|}} \left( \frac2R \inf_{\ell \in \mathcal{P}_1}\left\| u - \ell \right\|_{\underline{L}^2 \left(B_{R/2}(y)\right)} + \frac2R \inf_{\ell \in \mathcal{P}_1} \left\| v - \ell \right\|_{\underline{L}^2\left(B_{R/2}(y)\right)} \right) \leq 2^{\frac d2 +1 }  \mathsf{M}. 
\end{equation*}
We can thus apply~\eqref{e.Lip.Lbardiff} for $u-v$ in $B_{R/2-|h|}(y)$. Dividing the resulting inequality with $|h|$ and recalling that $u \in H_{\textrm{loc}}^2(B_R)$, we may send $|h| \to 0$ and obtain, for $y \in B_{R/2}$,  
\begin{equation*} 
\frac1s \left\| \nabla u - \left( \nabla u \right)_{B_s(y)}  \right\|_{\underline{L}^2(B_s(y))}  \leq C\left( \mathsf{M}\vee 1 \right)^{\frac{d+2}{2\theta}}  \frac1R \left\| \nabla u - \left( \nabla u \right)_{B_{R/4}(y)} \right\|_{\underline{L}^2\left(B_{R/4}(y)\right)} .
\end{equation*}
Now~\eqref{e.C11.Lbar} follows by the previous inequality and an application of the Caccioppoli estimate, for the differentiated equation $\nabla \cdot \left( D_p^2 F(\nabla u) \nabla^2 u \right) = 0$, giving
\begin{align*} 
\left\| \nabla^2 u \right\|_{\underline{L}^2(B_{s/2}(y))} & 
\leq  \frac Cs \left\| \nabla u - \left( \nabla u \right)_{B_s}  \right\|_{\underline{L}^2(B_s(y))} 
\\ &
\leq
C\left( \mathsf{M}\vee 1 \right)^{\frac{d+2}{2\theta}}  \frac1R \left\| \nabla u - \left( \nabla u \right)_{B_{R/4}(y)} \right\|_{\underline{L}^2\left(B_{R/4}(y)\right)}, 
\end{align*}
after letting $s \to 0$ and applying the Caccioppoli estimate~\eqref{e.Cacc2nd.Lbar}.

\smallskip

\emph{Step 4}. We then prove~\eqref{e.C1gamma.Lbardiff}. We return to Step 2, and observe that by using~\eqref{e.C11.Lbar} we actually get an improved estimate
\begin{equation*} 
\left| \b(0) - \b(x)\right| \leq C \left( \mathsf{M}\vee 1 \right)^{\gamma \eta}   \left( \frac{|x|}{R} \right)^{\gamma} \mathsf{M}^\gamma,
\end{equation*}
and using this as in the Step 2, we also get 
\begin{multline*} 
\inf_{\ell \in \mathcal{P}_1}\left\| u - v - \ell \right\|_{\underline{L}^2(B_{r})}
\leq C\left(\frac rs \right)^2 \inf_{\ell \in \mathcal{P}_1}\left\| u - v - \ell \right\|_{\underline{L}^2(B_{s})} 
\\ + C \left( \mathsf{M}\vee 1 \right)^{\gamma \eta}  \left(\frac{s}{r}\right)^{\frac d2}  \left( \frac{r}{R} \right)^{\gamma} \mathsf{M}^\gamma \left\| u-v  - (u-v)_{B_s}\right\|_{\underline{L}^2(B_s)}.
\end{multline*}
Setting this time
\begin{equation*} 
E_1(s) := \left( \frac{R}{s} \right)^{1+ \gamma} \inf_{\ell \in \mathcal{P}_1}\left\| u - v - \ell \right\|_{\underline{L}^2(B_{s})},
\end{equation*}
we find a small constant $\ep(\mathsf{K},d,\Lambda) \in (0,1)$ such that 
\begin{equation*} 
E_1(\ep s) \leq \frac 12 E_1(s) + C \left( \mathsf{M}\vee 1 \right)^{\gamma \eta} \mathsf{M}^\gamma E_0(s)
\end{equation*}
By~\eqref{e.Lip.Lbardiff}, $E_0(s) \leq C \left( \mathsf{M}\vee 1 \right)^{\eta} E_0(R)$, and thus, after reabsorption, we deduce~\eqref{e.C1gamma.Lbardiff}.

\smallskip

\emph{Step 5}. We finally prove~\eqref{e.C2alpha.Lbardiff}.  As in Step 3,~\eqref{e.C1gamma.Lbardiff} and~\eqref{e.Cacc2nd.Lbar} yield that, for $r \leq s \leq \frac{R}{2}$,  
\begin{equation*} 
\inf_{\ell \in \mathcal{P}_1}\left\| \nabla u - \ell \right\|_{\underline{L}^2(B_{r})}
\leq C\left(\frac rs \right)^2 \inf_{\ell \in \mathcal{P}_1}\left\| \nabla u- \ell \right\|_{\underline{L}^2(B_{s})} 
+ C \left( \mathsf{M}\vee 1 \right)^{\gamma \eta}  \left(\frac{s}{r}\right)^{\frac d2}  \left( \frac{s}{R} \right)^{1+\gamma} \mathsf{M}^{1+\gamma} .
\end{equation*}
Setting 
\begin{equation*} 
E_2(s) := \left( \frac{R}{s}\right)^{1+\gamma} \inf_{\ell \in \mathcal{P}_1}\left\| \nabla u - \ell \right\|_{\underline{L}^2(B_{s})}, 
\end{equation*}
we find $\ep(\mathsf{K},\gamma,d,\Lambda) \in (0,1)$ such that 
\begin{equation*} 
E_2(\ep s) \leq \frac12 E_2(s)  +  C \left( \mathsf{M}\vee 1 \right)^{\gamma \eta} \mathsf{M}^{1+\gamma} .
\end{equation*}
Thus, we get after reabsorption that
\begin{equation*} 
\sup_{s \in \left(0 , \frac{R}{2}\right)} E_2(s) \leq C \left( E_2\left(\frac R2 \right) + \left( \mathsf{M}\vee 1 \right)^{\gamma \eta}  \mathsf{M}^{1+\gamma} \right) . 
\end{equation*}
Letting $\ell_s$ be the affine function realizing the infimum in $E_2(s)$, we have analogously to Step 3 that 
\begin{equation*} 
\left| \nabla \ell_{2s} - \nabla \ell_{s} \right| \leq \frac{C}{R} \left( \frac{s}{R}\right)^{\gamma}  \left( E_2\left(\frac R2 \right) + \left( \mathsf{M}\vee 1 \right)^{\gamma \eta}  \mathsf{M}^{1+\gamma} \right) . 
\end{equation*}
Thus $\{\nabla \ell_{s}\}_s$ is a Cauchy sequence and there exists an affine function $\overline{\ell}$ such that 
\begin{equation*} 
\left| \nabla \ell_{s} - \nabla \overline{\ell} \right| \leq  \frac{C}{R} \left( \frac{s}{R}\right)^{\gamma}  \left( E_2\left(\frac R2 \right) + \left( \mathsf{M}\vee 1 \right)^{\gamma \eta}  \mathsf{M}^{1+\gamma} \right) .
\end{equation*}
A similar argument shows also that 
\begin{equation*} 
\left| \ell_{s}(0) - \overline{\ell}(0) \right| \leq  C \left( \frac{s}{R}\right)^{1+\gamma}  \left( E_2\left(\frac R2 \right) + \left( \mathsf{M}\vee 1 \right)^{\gamma \eta}  \mathsf{M}^{1+\gamma} \right) .
\end{equation*}
As a consequence we obtain that that for this $\overline{\ell}$ we have, for $s \in \left(0 , \frac{R}{2}\right)$, that
\begin{equation*} 
\left\| \nabla u - \overline{\ell} \right\|_{\underline{L}^2(B_{s})} \leq C \left( \frac{s}{R} \right)^{1+\gamma} \left( 
\inf_{\ell \in \mathcal{P}_1}\left\| \nabla u - \ell \right\|_{\underline{L}^2(B_{R/2})} +  \left( \mathsf{M}\vee 1 \right)^{\gamma \eta} \mathsf{M}^{1+\gamma}  \right).
\end{equation*}
It is straightforward 
to show that translating this estimate implies that $u \in C^{2,\gamma}$ and that $\overline{\ell}(x) = \nabla q(x)$, where 
\begin{equation*} 
q(x) :=  u(0) + \nabla u(0) \cdot x + \frac 12 x \cdot \nabla^2 u(0) x,
\end{equation*}
and the previous inequality reads as 
\begin{equation}  \label{e.C2alpha.Lbardiffpre}
\left\| \nabla u - \nabla q \right\|_{\underline{L}^2(B_{s})} \leq C \left( \frac{s}{R} \right)^{1+\gamma} \left( 
\inf_{\ell \in \mathcal{P}_1}\left\| \nabla u - \ell \right\|_{\underline{L}^2(B_{R/2})} +  \left( \mathsf{M}\vee 1 \right)^{\gamma \eta} \mathsf{M}^{1+\gamma}  \right).
\end{equation}
Furthermore, we have the equation 
\begin{equation*} 
- \nabla \cdot \left( D_p^2 F(\nabla u)  \nabla \left( \nabla u -  \nabla q \right) \right) = 
- \nabla \cdot \left( \left( D_p^2 F(\nabla u) - D_p^2 F(\nabla u(0)) \right)   \nabla^2 u(0)\right).
\end{equation*}
The term on the right can be controlled as
\begin{equation*} 
\left| \left( D_p^2 F(\nabla u(x)) - D_p^2 F(\nabla u(0)) \right)   \nabla^2 u(0)  \right| \leq C \left| \nabla u(x)  - \nabla u(0) \right|^\gamma \left| \nabla^2 u(0) \right|,
\end{equation*}
and hence, by~\eqref{e.C11.Lbar},
\begin{equation*} 
\left| \left( D_p^2 F(\nabla u(x)) - D_p^2 F(\nabla u(0)) \right)   \nabla^2 u(0)  \right|  \leq \frac{C}{R} \left( \frac{|x|}{R} \right)^\gamma \left( \mathsf{M} \vee 1 \right)^{\eta(1+\gamma)}  \mathsf{M}^{1+\gamma}. 
\end{equation*}
We thus get a Caccioppoli estimate, using also~\eqref{e.C2alpha.Lbardiffpre},  
\begin{align} \notag 
\left\| \nabla^2 u -  \nabla^2 u(0)  \right\|_{\underline{L}^2(B_{s/2})} & \leq \frac{C}{s} \left\| \nabla u -  \nabla q  \right\|_{\underline{L}^2(B_{s})} + \frac{C}{R} \left( \frac{s}{R} \right)^\gamma \left( \mathsf{M} \vee 1 \right)^{\eta(1+\gamma)}  \mathsf{M}^{1+\gamma}
\\ \notag & 
\leq \frac {C}{R} \left( \frac{s}{R} \right)^\gamma  \left( \inf_{\ell \in \mathcal{P}_1}\left\| \nabla u - \ell \right\|_{\underline{L}^2(B_{R/2})} +  \left( \mathsf{M} \vee 1 \right)^{\eta(1+\gamma)}  \mathsf{M}^{1+\gamma} \right).
\end{align}
By translating the previous estimate and applying the triangle inequality yields, for any $x,y \in B_{R/2}$ with $|x-y|\leq \frac R{32}$, 
\begin{align} \notag 
\lefteqn{ \left| \nabla^2 u(x) -  \nabla^2 u(y)  \right|} \quad &
\\ \notag & \leq 2^d \left\| \nabla^2 u - \nabla^2 u(x)  \right\|_{\underline{L}^2(B_{4|x-y|}(x))} + 2^d \left\| \nabla^2 u - \nabla^2 u(y)  \right\|_{\underline{L}^2(B_{4|x-y|}(y))}
\\ \notag & \leq  \frac{C}{R}  \left( \frac{|x-y|}{R} \right)^{\gamma}  \left( \inf_{\ell \in \mathcal{P}_1}\left\| \nabla u - \ell \right\|_{\underline{L}^2(B_{R/2})} +  \left( \mathsf{M} \vee 1 \right)^{\eta(1+\gamma)}  \mathsf{M}^{1+\gamma} \right),
\end{align}
which is~\eqref{e.C2alpha.Lbardiff} in the case $|x-y|\leq \frac R{32}$ after noticing that we trivially have $$\inf_{\ell \in \mathcal{P}_1}\left\| \nabla u - \ell \right\|_{\underline{L}^2(B_{R/2})}  \leq \inf_{q \in \mathcal{P}_2}\left\| \nabla u - \nabla q \right\|_{\underline{L}^2(B_{R/2})}.$$  If, on the other hand, $|x-y|> \frac R{32}$, the estimate follows easily by setting $x_j =  \frac{j}{32} x$  and $y_j =  \frac{j}{32} y$, $j \in \{0,1,\ldots,32\}$, and applying the previous estimate with $x = x_j$ and $y=x_{j+1}$, and similarly for $y_j$ and $y_{j+1}$, and we thus deduce~\eqref{e.C2alpha.Lbardiff} also in this case. The proof is complete.
\end{proof}

\begin{proposition}[{$C^{1,\beta}$ Schauder estimate}]
\label{p.C1beta.coeff}
Fix $\gamma' \in (0,1]$ and $\, \mathsf{K} \in (0,\infty)$. 
Suppose that  $F:\Rd \times B_1  \to \R$ is a Lagrangian satisfying, for every $p \in \R^d$, 
\begin{equation}  \label{e.Freg2}
 \left[ \frac{D_p F(p,\cdot)}{1 + |p|} \right]_{C^{0,\gamma'}(B_1)} 
 \leq \mathsf{K}
\end{equation}
and
\begin{equation} \label{e.F.uniconvex2}
I_d \leq D^2_p F(p,\cdot) \leq \Lambda I_d.
\end{equation}
There exist constants $\beta(\mathsf{K},\gamma',d,\Lambda) \in (0,\gamma']$ and $C(\mathsf{K},\gamma',d,\Lambda)$ such that, for every local $F$-minimizers $u\in H^1(B_1)$, 
\begin{equation} \label{e.C1betacoeff.res1}
\|\nabla u  \|_{L^\infty(B_{1/2})} + \left[ \nabla u  \right]_{C^{0,\beta}(B_{1/2})} \leq C \left( 1 + \|\nabla u \|_{\underline{L}^2(B_1)} \right).
\end{equation}
\end{proposition}

\begin{proof}[Proof of Proposition~\ref{p.C1beta.coeff}] 
Let $u_r$ be the minimizer of 
\begin{equation*} 
\min_{v \in u + H_0^1(B_r)} \fint_{B_r} F(\nabla v(x),0) \, dx,
\end{equation*}
so that, by the first variation, 
\begin{equation*} 
 - \nabla \cdot \left( D_pF(\nabla u,0) - D_pF(\nabla u_r,0)  \right) =  \nabla \cdot \left( D_pF(\nabla u,x) - D_pF(\nabla u,0) \right) 
\end{equation*}
in $B_r$. From this, together with~\eqref{e.Freg2},  we obtain 
\begin{equation*} 
\|\nabla u - \nabla u_r\|_{\underline{L}^2(B_r)} \leq C r^{\gamma'} \left(1 + \|\nabla u \|_{\underline{L}^2(B_r)} \right)
\end{equation*}
Then, for $\theta \in (0,1)$, we get by~\eqref{e.C1alpha.Lbar} (recalling Remark~\ref{r.C1beta.differences.Lbar}) and the triangle inequality that
\begin{multline} \label{e.C1betacoeff.res1.temp1}
\left\| \nabla u - \left( \nabla u \right)_{B_{\theta r}} \right\|_{\underline{L}^2(B_{\theta r})} 
\\\leq
C \theta^\beta \left\| \nabla u - \left( \nabla u \right)_{B_{r}} \right\|_{\underline{L}^2(B_{r})}
+ C \theta^{-\frac d2} r^{\gamma'} \left(1 + \|\nabla u \|_{\underline{L}^2(B_r)} \right).
\end{multline}
Similarly to Step 2 of the proof of Proposition~\ref{p.C1beta.differences.Lbar} this yields, for $\beta' := \frac12 \left( \beta \wedge \gamma' \right)$, that
\begin{equation*} 
\sup_{r \in (0,1)} \left(r^{-\beta'}\|\nabla u - (\nabla u)_{B_r} \|_{\underline{L}^2(B_r)} +  \|\nabla u \|_{\underline{L}^2(B_r)}  \right) \leq  C \left(1 + \|\nabla u \|_{\underline{L}^2(B_1)} \right) .
\end{equation*}
By translating this estimate, we conclude by an iteration argument that
\begin{equation*} 
\|\nabla u  \|_{L^\infty(B_{1/2})} + \left[ \nabla u  \right]_{C^{0,\beta'}} \leq C \left(1 + \|\nabla u \|_{\underline{L}^2(B_1)} \right).
\end{equation*}
Going back to~\eqref{e.C1betacoeff.res1.temp1}, we see that 
\begin{multline} \label{e.C1betacoeff.res1.temp2}
\left\| \nabla u - \left( \nabla u \right)_{B_{\theta r}} \right\|_{\underline{L}^2(B_{\theta r})} 
\\\leq
C \theta^\beta \left\| \nabla u - \left( \nabla u \right)_{B_{r}} \right\|_{\underline{L}^2(B_{r})}
+ C \theta^{-\frac d2} r^{\gamma'} \left(1 + \|\nabla u \|_{\underline{L}^2(B_1)} \right),
\end{multline}
and may repeat the argument to obtain H\"older exponent $\beta' \in (0,\beta) \cap [0,\gamma']$. Thus we conclude 
with~\eqref{e.C1betacoeff.res1} after relabelling $\beta'$. 
\end{proof}

\section{Homogenization estimates}
\label{ap.ASestimates}

In this appendix we derive the estimate~\eqref{e.vztildevz} from estimates in~\cite{AS}. In fact,~\eqref{e.vztildevz} is an immediate consequence of the triangle inequality, a basic energy estimate and the following stronger form of~\eqref{e.ellztildevz}: there exist~$\alpha(d,\Lambda)>0$ and $C(\mathsf{K},q,d,\Lambda)<\infty$ and a random variable $\X$ satisfying $\X\leq \O_1(C)$ such that, for every $M,N\in\N$ with $M\leq N$,
\begin{multline}
\label{e.ellztildevz.strong}
\sup_{\xi \in B_{\mathsf{K}3^{Mq}}} 
(1+|\xi|)^{-2} 3^{-d(N-M)} 
\sum_{z \in 3^M\Zd\cap \cu_N}
3^{-2M}
\left\| 
v(\cdot,z+\cu_M,\xi) - \ell_\xi
\right\|_{\underline{L}^2 \left(z+\cu_{M} \right)}^2
\\
\leq C3^{-M\alpha(d-\sigma)} + \X 3^{-\sigma'N}.
\end{multline}
What~\cite[Corollary 3.5]{AS} gives us is the following bound, valid for every $z\in 3^M\Zd$: 
\begin{multline}
\label{e.ellztildevz.strong.2}
\sup_{\xi \in B_{\mathsf{K}3^{Mq}}} 
(1+|\xi|)^{-2} 
3^{-2M}
\left\| 
v(\cdot,z+\cu_M,\xi) - \ell_\xi
\right\|_{\underline{L}^2 \left(z+\cu_{M} \right)}^2
\\
\leq C3^{-M\alpha(d-\sigma)} + \X_z 3^{-\sigma'M},
\end{multline}
where the random variables $\{ \X_z \}_{z\in 3^M\Zd}$ are identically distributed, satisfy the bound $\X_z\leq \O_1(C)$, and each $\X_z$ is measurable with respect to $\F(z+\cu_M)$, that is, it depends on the restriction of the Lagrangian to $z+\cu_M$. In particular, for every $z,z'\in3^M\Zd$, the random variables $\X_z$ and $\X_{z'}$ are independent provided that the cubes $z+\cu_M$ and $z'+\cu_M$ are not adjacent. To obtain~\eqref{e.ellztildevz.strong}, it suffices to prove
\begin{equation*}
3^{-d(N-M)} \sum_{z\in 3^M\Zd\cap \cu_N} \X_z 
\leq 2 \E \left[ \X_0 \right] + \O_1 \left( 3^{-d(N-M)} \right). 
\end{equation*}
This is a very crude large deviation-type estimate that is relatively straightforward to derive. It can be obtained for instance from~\cite[Lemma 2.14]{AS} or its proof.

\subsection*{Acknowledgments}
SA was partially supported by the National Science Foundation through grant DMS-1700329. SF was partially supported by NSF grants DMS-1700329 and DMS-1311833.
T.K.~was supported by the Academy of Finland and the European Research Council (ERC) under the European Union's Horizon 2020 research and innovation programme (grant agreement No 818437).


\small
\bibliographystyle{abbrv}
\bibliography{linearization}

\end{document}